\documentclass[dvips,imslayout,preprint]{imsart}

\RequirePackage[OT1]{fontenc}
\RequirePackage[numbers]{natbib}
\RequirePackage[colorlinks,citecolor=blue,urlcolor=blue]{hyperref}
\RequirePackage{hypernat}
\usepackage{amsmath,amsfonts,amssymb,mathrsfs,amsthm}
\usepackage{hyperref}
\usepackage[normalem]{ulem}
\usepackage{color,colortbl}
\usepackage{graphicx}
\usepackage{bm}
\usepackage{bbm}

\newcounter{counterL}

\newcounter{counterRm}
\newcounter{counterCor}

\theoremstyle{plain}
\newtheorem{theorem}{Theorem}
\newtheorem{corollary}[counterCor]{Corollary}
\newtheorem{proposition}[theorem]{Proposition}
\newtheorem{lemma}[counterL]{Lemma}

\theoremstyle{definition}
\newtheorem{definition}{Definition}
\newtheorem{remark}[counterRm]{Remark}

\newcommand{\ab}{{}_{\alpha,\beta}}
\newcommand{\ba}{{}_{\beta, \alpha}}

\newcommand{\bg}{{}_{\beta}}
\newcommand{\ag}{{}_{\alpha}}

\newcommand{\ux}{{}_{x}}
\newcommand{\uxi}{{}_{\xi}}

\newcommand{\bM}{\bm{M}}
\newcommand{\bR}{\bm{R}}
\newcommand{\bL}{\bm{L}}
\newcommand{\bZ}{\bm{Z}}
\newcommand{\bX}{\bm{X}}
\newcommand{\bF}{\bm{F}}
\newcommand{\bxxi}{\bm{\xi}}
\newcommand{\bx}{\bm{x}}
\newcommand{\by}{\bm{y}}
\newcommand{\bgg}{\bm{g}}
\newcommand{\bff}{\bm{f}}
\newcommand{\bh}{\bm{h}}
\newcommand{\bdelta}{\bm{\delta}}

\newcommand{\N}{\mathbb{N}}
\newcommand{\R}{\mathbb{R}}
\newcommand{\I}{\mathcal{I}}
\newcommand{\Exp}[1]{\mathbb{E}\left[#1\right]}
\newcommand{\ExpN}[1]{\mathbb{E}_{N}\left[#1\right]}

\newcommand{\bpsi}{\bm{\psi}}
\newcommand{\bxi}{\bm{\xi}}
\newcommand{\bXi}{\bm{\Xi}}

\newcommand{\bphi}{ \bm{\phi}}
\newcommand{\bPhi}{ \bm{\Phi}}
\newcommand{\bDelta}{ \bm{\Delta}}
\newcommand{\bPsi}{ \bm{\Psi}}
\newcommand{\balpha}{\bm{\alpha}}

\newcommand{\ProbXi}{\bm{\mathrm{P}}_{\Xi^{-1}}}
\newcommand{\Probxi}{\bm{\mathrm{P}}_{\xi}}

\newcommand{\uX}{{}_{ \scriptscriptstyle X \displaystyle}}
\newcommand{\uY}{{}_{ \scriptscriptstyle Y \displaystyle}}
\newcommand{\uZ}{{}_{ \scriptscriptstyle Z \displaystyle}}
\newcommand{\dX}{{ \scriptscriptstyle X \displaystyle}}
\newcommand{\dY}{{ \scriptscriptstyle Y \displaystyle}}
\newcommand{\dZ}{{ \scriptscriptstyle Z \displaystyle}}

\arxiv{math.PR/0000000}

\startlocaldefs
\numberwithin{equation}{section}
\theoremstyle{plain}

\endlocaldefs

\begin{document}

\begin{frontmatter}
\title{Multi-Resolution Schauder Approach to Multidimensional Gauss-Markov Processes}
\runtitle{Gauss-Markov Processes}

\begin{aug}
\author{\fnms{Thibaud} \snm{Taillefumier}\thanksref{m1}\ead[label=e1]{ttaillefum@rockefeller.edu}}
\and
\author{\fnms{Jonathan} \snm{Touboul}\thanksref{m1,m2}\ead[label=e2]{jonathan.touboul@sophia.inria.fr}}
\ead[label=u2,url]{http://www-sop.inria.fr/members/Jonathan.Touboul}

\runauthor{T. Taillefumier \and J. Touboul}

\affiliation[m1]{Laboratory of Mathematical Physics, The Rockefeller University, New-York, USA\\}
\affiliation[m2]{ NeuroMathComp Laboratory, INRIA/ENS Paris, Paris\\}

\address{Laboratory  of Mathematical Physics,\\The Rockefeller University,\\ 1230 York Avenue,\\New-York, NY 10065\\
\printead{e1}\\
\phantom{E-mail:\ }}

\address{NeuroMathComp Laboratory\\
INRIA/ ENS Paris\\
23, avenue d'Italie\\
75013 Paris\\
\printead{e2}\\
\printead{u2}
}
\end{aug}

\begin{abstract}
The study of multidimensional stochastic processes involves complex computations in intricate functional spaces. 
In particular, the diffusion processes, which include the practically important Gauss-Markov processes, are ordinarily defined through the theory of stochastic integration. 
Here, inspired by the L\'{e}vy-Ciesielski construction of the Wiener process, we propose an alternative representation of multidimensional Gauss-Markov processes as  expansions on well-chosen Schauder bases, with independent random coefficients of  normal law with zero mean and unitary variance. 
We thereby offer a natural multi-resolution description of Gauss-Markov processes as limits of the finite-dimensional partial sums of the expansion, that are strongly almost-surely convergent.
Moreover, such finite-dimensional random processes constitute an optimal approximation of the process, in the sense of minimizing the associated Dirichlet energy under interpolating constraints. 
This approach allows simpler treatment in many applied and theoretical fields and we provide a short overview of applications we are currently developing. 
\end{abstract}



\end{frontmatter}


\section*{Introduction}


Multidimensional continuous processes are easily defined for a wide range of stochastic equations through standard It\^{o} integration theory (see e.g. \cite{Protter:2004}).
However, studying their properties proves surprisingly challenging, even for the simplest multidimensional processes.
The high dimensionality of the ambient space and the nowhere-differentiability of the sample paths conspire to heighten the intricacy of the sample paths spaces.
In this regard, such spaces have been chiefly studied for multidimensional diffusion processes~\cite{Strook:2006}, and more recently, the development of rough paths theory has attracted renewed interest in the field~\cite{Lyons:1994,Lyons:2002,Friz:2010}.
Aside from these remarkable theoretical works, little emphasis  is put on the sample paths since most of the available results only make sense in distribution. 
This is particularly true in It\^{o} integration theory, where the sample path is completely neglected for the It\^{o} map being defined up to null sets of paths.\\
Adopting a discrete representation of a process that allows the inference of  sample paths properties from finite-dimensional approximating processes, alleviates the difficulty of working in multidimensional spaces.
This is achieved through writing a process $\bX$ as an {\it almost-sure path-wise} convergent series of random functions 
\begin{equation}
\bX_{t} = \lim_{N \to \infty} \bX_{N} \quad \mathrm{with} \quad \bX_{N} = \sum_{n=0}^{N} \bm{f}_{n} (t)\cdot \bXi_{n}  \, ,
\nonumber
\end{equation}
where the $\bm{f}_{n}$ are deterministic functions and the $\bXi_{n}$ are independently identically distributed random variables.\\
The L\'evy-Cesielski construction of the $d$-dimensional Brownian motion $\bm{W}$ (also referred as Wiener process) provides us with an example of discrete representation for a continuous stochastic process~\cite{Levy}.
Noticing the simple form of the probability density of a Brownian bridge, it is based on completing sample paths by interpolation according to the conditional probabilities of the Wiener process.
More specifically, the coefficients $\bXi_{n}$ are Gaussian independent and the elements $\bm{f}_{n}$, called Schauder elements and denoted $\bm{s}_{n}$, are obtained by time-dependent integration of the Haar basis elements: $\bm{s}_{0,0}(t) = t \, \bm{I}_d$ and  $\bm{s}_{n,k}(t) = s_{n,k}(t) \, \bm{I}_d$, with for all $n>0$
\[
s_{n,k}(t) = \begin{cases}
 \displaystyle 2^{\frac{n-1}{2}} (t-l_{n,k}) \, ,& k2^{-n\!+\!1}\leq t \leq (2k\!+\!1)2^{-n} \, ,\\
 \displaystyle 2^{\frac{n-1}{2}} (r_{n,k}-t) \, ,& (2k\!+\!1)2^{-n} \leq t \leq (k\!+\!1)2^{-n\!+\!1} \, ,\\
 \displaystyle 0 \, , & \mathrm{otherwise} \, . \\
 \end{cases}\]
This latter point is of relevance since, for being a Hilbert system, the introduction of the Haar basis greatly simplify the demonstration of the existence of the Wiener process~\cite{Ciesielski}.
It is also important for our purpose to realize that the Schauder elements $\bm{s}_{n}$ have compact supports that exhibits a nested structure: this fact entails that the finite sums $\bm{W}_{N}$ are processes that interpolates the limit process $\bm{W}$ on the endpoints of the supports, i.e. on the dyadic points $k2^{-N}$, $0 \leq k \leq 2^{N}$.
One of the specific goal of our construction is to maintain such a property in the construction of all multidimensional Gauss-Markov processes $\bm{X}$, being successively approximated by finite dimensional processes ${X}^{N}$ that interpolates $\bm{X}$ at ever finer resolution.
In that respect, , it is only in that sense that we refer to our framework as a multi-resolution approach as opposed to the wavelet multi-resolution theory~\cite{Mallat:1989}.
Extensions of this method to the fractional Brownian motion were also developed~\cite{Meyer:1999}, but applied to some very specific processes.\\

In view of this, we propose a construction of multidimensional Gaussian Markov processes using a multi-resolution Schauder basis of functions. 
As for the L\'evy-Ciesielski construction, our basis is not made of orthogonal functions but the elements are such that the random coefficients $\bXi_{n}$ are always independent and Gaussian (for convenience with law $\mathcal{N}(\bm{0},\bm{I}_d)$, i.e. with zero mean and unitary variance).
We first develop a heuristic approach for the construction of stochastic processes reminiscent of the midpoint displacement technique~\cite{Levy,Ciesielski}, before rigorously deriving the multi-resolution basis that we will be using in all the paper. 
This set of functions is then studied as a multi-resolution Schauder basis of functions:
in particular, we derive explicitly from the multi-resolution basis an Haar-like Hilbert basis,  which is the underlying structure explaining the dual relationship between basis elements and coefficients. 
Based on these results, we study the construction application and its inverse, the coefficient applications, that relate coefficients on the Schauder basis to sample paths. 
We pursue by proving the almost sure and strong convergence of the process having independent standard normal coefficients on the Schauder basis to a Gauss-Markov process. 
We also show that our decomposition is optimal in some sense that is strongly evocative of spline interpolation theory~\cite{deBoor:2001}: the construction yields successive interpolations of the process at the interval endpoints that minimizes the Dirichlet energy induced by the differential operator associated with the Gauss-Markov process~\cite{Fukushima:1994,Pitt:1971kl}. 
We also provide a series of examples for which the proposed Schauder framework yields bases of functions that have simple closed form formulae: in addition to the simple one-dimensional Markov processes, we explicit our framework for two classes of multidimensional processes, the Gauss-Markov rotations and the iteratively integrated Wiener processes (see e.g \cite{McKean:1963,Goldman:1971,Lefebvre:1989}).\\

The ideas underlying this work can be directly traced back to the original work of L\'evy. 
Here, we intend to develop a self-contained Schauder dual framework to further the description of multidimensional Gauss-Markov processes, and in doing so, we extend some well-known results of interpolation theory in signal processing~\cite{deBoor:2008,Kimel1,Kimel2}.
To our knowledge, such an approach is yet to be proposed. 
By restraining our attention to Gauss-Markov processes, we obviously do not pretend to generality.
However, we hope our construction proves of interest for a number of points, which we tentatively list in the following.
First, the almost-sure path-wise convergence of our construction together with the interpolation property of the finite sums allows to reformulate results of stochastic integration in term of  the geometry of finite-dimensional sample paths.
In this regard, we found it appropriate to illustrate how in our framework, the Girsanov theorem for Gauss-Markov processes appears as a direct consequence of the finite-dimensional change of variable formula.
Second, the characterization of our Schauder elements as the minimizer of a Dirichlet form paves the  way to the construction of infinite-dimensionel Gauss-Markov processes, i.e. processes whose sample points themselves are infinite dimensional~\cite{Kolsrud:1988, DaPrato:1992}.
Third, our construction shows that  approximating a Gaussian process by a sequence of interpolating processes relies entirely on the existence of a regular triangularization of the covariance operator, suggesting to further investigate this property for non-Markov Gaussian processes~\cite{Kailath:1978}.
Finally, there is a number of practical applications where applying the Schauder basis framework clearly provides an advantage compared to standard stochastic calculus methods, among which first-hitting times of stochastic processes, pricing of multidimensional path-dependant options~\cite{Baldi:1995,Baldi:2000,Baldi:2002,Gobet2000167}, regularization technique for support vector machine learning~\cite{Scholkopf:2001} and more theoretical work on uncovering the differential geometry structure of the space of Gauss-Markov stochastic processes~\cite{Sekine:2001}. 
We conclude our exposition by developing in more details some of these direct implications which will be the subjects of forthcoming papers.


\section{Rationale of the Construction}

In order to provide a discrete multi-resolution description of Gauss-Markov processes, we first establish basic results about the law of Gauss-Markov bridges in the multidimensional setting. 
We then use them to infer the candidate expressions for our desired bases of functions, while imposing its elements  to be compactly supported on a nested sequence segments.
Throughout this article, we are working in an underlying probability space $\left( \Omega, \mathcal{F}, \mathrm{\bold{P}} \right)$. 


\subsection{Multidimensional Gauss-Markov Processes}

After recalling the definition of multidimensional Gauss-Markov processes in terms of stochastic integral, we use the well known conditioning formula for Gaussian vectors to characterize the law of Gauss-Markov bridge processes.


\subsubsection{Notations and Definitions}

Let $(\bm{W}_t, \mathcal{F}_t, t\in [0,1])$ be a $m$-dimensional Wiener process, consider the continuous functions $\bm{\alpha}:[0,1] \to \R^{d\times d}$  $\sqrt{\bm{\Gamma}}:[0,1] \to \R^{d\times m}$  and define the positive bounded continuous function $\bm{\Gamma} = \sqrt{\bm{\Gamma}} \cdot \sqrt{\bm{\Gamma}}^{T}:[0,1] \to \R^{d\times d}$. 
The $d$-dimensional Ornstein-Uhlenbeck process associated with these parameters is solution of the equation
\begin{equation}\label{eq:eqStoch}
d\bm{X}_t = \bm{\alpha}(t) \cdot \bm{X}_t + \sqrt{\bm{\Gamma}(t)} \cdot d\bm{W}_t \, ,
\end{equation}
and  with initial condition $\bm{X}_{t_0}$ in $t_{0}$, it reads
\begin{equation}\label{eq:OUMult}
	\bm{X}_t = \bF(t_0,t) \cdot \bm{X}_{t_0}  + F(t_0,t) \cdot \int_{t_0}^t \bF(s,t_0) \cdot \sqrt{\bm{\Gamma}(s)} \cdot d\bm{W}_s \, ,
\end{equation}
where $\bF(t_0,t)$ is the flow of the equation, namely the solution in $\R^{d\times d}$ of the linear equation:
\begin{eqnarray} \label{eq:linDeter}
\begin{cases}
	\displaystyle \frac{\partial \bF(t_0,t)}{\partial t} &= \bm{\alpha}(t) \bF(t)\\
	\bF(t_0,t_0)  &= \bm{I}_d
\end{cases} \, .
\end{eqnarray}
Note that the flow $\Phi(t_0,t)$ enjoys the chain rule property:
\[\bF(t_0,t)= \bF(t_1,t)\cdot \bF(t_0,t_1).\]
For all $t,s$ such that $t_{0}<s,t$, the vectors $\bX_{t}$ and $\bX_{s}$ admit the covariance
\begin{eqnarray}
\nonumber\bm{C}_{t_{0}}(s,t) &= & \bF(t_{0},t) \left( \int_{t_0}^{t \wedge s}  \bF(w,t_{0})  \bm{\Gamma}(w)  \bF(w,t_{0})^{T} \, dw \right)  \bF(t_{0},s)^{T} \,\\
\label{eq:covExp} &=&  \bF(t_{0},t) \bm{h}_{t_0}(s,t)  \bF(t_{0},s)^{T} \,
\end{eqnarray}
where we further defined $\bm{h}_u(s,t)$ the function 
\[\bh_u(s,t) = \int_{s}^t  \bF(w,u) \cdot \bm{\Gamma}(w)\cdot  \bF(w,u)^T \; dw\]
which will be of particular interest in the sequel. 
Note that because of the chain rule property of the flow, we have:
\begin{equation}\label{eq:hChainRule}
	\bh_v(s,t) =  \bF(v,u) \,\bh_u(s,t)\, \bF(v,u)^T
\end{equation}
We suppose that the process $\bX$ is never degenerated, that is, for all $t_0<u<v$, all the components of the vector $\bX_{v}$ knowing $\bX_{u}$ are non-deterministic random variables, which is equivalent to say that the covariance matrix of $\bm{X}_{v}$ knowing $\bX_{u}$, denoted $\bm{C}_{u}(v,v)$ is symmetric positive definite for any $u\neq v$.
Therefore, assuming the initial condition $\bX_{0}= \bm{0}$, the multidimensional centered process $\bm{X}$ has a representation (similar to Doob's representation for one-dimensional processes, see~\cite{Karatzas}) of form:
\[\bm{X}_t = \bm{g}(t)\int_{0}^t\bm{f}(s)\cdot d\bm{W}_s \, , \] 
with $\bm{g}(t) =  \bF(0,t)$ and $\bm{f}(t) =  \bF(t,0) \cdot \sqrt{\bm{\Gamma}(t)}$. 

Note that the processes considered in this paper are defined on the time interval $[0,1]$. However, because of the time-rescaling property of these processes, considering the processes on this time interval is equivalent to considering the process on any other bounded interval without loss of generality.


\subsubsection{Conditional Law and Gauss-Markov Bridges }

As stated in the introduction, we aim at defining a multi-resolution description of Gauss-Markov processes. 
Such a description can be seen as a multi-resolution interpolation of the process that is getting increasingly finer. 
This principle, in addition to the Markov property, prescribes to characterize the law of the corresponding Gauss-Markov bridge, i.e. the Gauss-Markov process under consideration, conditioned on its initial and final values. 
The bridge process of Gauss process is still a Gauss process and, for a Markov process, its law can be computed as follows:

\begin{proposition} \label{bridgeProp}
Let $t_x \leq t_z$ two times in the interval $[0,1]$. For any $t\in [t_x,t_z]$, the random variable $\bX_{t}$ conditioned on $\bX_{t_{x}} = \bm{x}$ and $\bX_{t_{z}} = \bm{z}$ is a Gaussian variable with  covariance matrix $\bm{\Sigma}(t)$ and mean vector $\bm{\mu}(t)$  given by:
\begin{eqnarray*}
\bm{\Sigma}(t;t_x,t_z) &=& \bm{h}_{t}(t_{x},t) \left( \bm{h}_{t}(t_{x},t_{z}) \right)^{-1} \bm{h}_{t}(t,t_{z}) \, , \\
\bm{\mu}(t) &=& \bm{\mu}^l(t;\,t_x,\,t_z) \cdot \bm{x} + \bm{\mu}^r(t;\,t_x,\,t_z) \cdot \bm{z} \, , 
\end{eqnarray*}
where the continuous matrix functions $\bm{\mu}^l(\cdot;\,t_x,\,t_z)$ and $\bm{\mu}^r(\cdot;\,t_x,\,t_z)$ of $\R^{d \times d}$ are given by:
\begin{equation*}
	\begin{cases}
		\bm{\mu}^l(t;\,t_x,\,t_z) &=  \bF(t_{x},t) \,  \bm{h}_{t_{x}}(t,t_{z}) \,   \left( \bm{h}_{t_{x}}(t_{x},t_{z})\right)^{-1} \, \\
		\bm{\mu}^r(t;\,t_x,\,t_z) &=  \bF(t_{z},t) \bm{h}_{t_{z}}(t_{x},t) \left( \bm{h}_{t_{z}}(t_{x},t_{z})\right)^{-1}.
	\end{cases}
\end{equation*}
\end{proposition}

Note that the functions $\mu^l$ and $\mu^r$ have the property that $\bm{\mu}^l(t_x;\,t_x,\,t_z)=\bm{\mu}^r(t_z;\,t_x,\,t_z)= \bm{I}_d$ and $\bm{\mu}^l(t_z;\,t_x,\,t_z)=\bm{\mu}^r(t_x;\,t_x,\,t_z)= \bm{0}$ ensuring that the process is indeed equal to $\bm{x}$ at time $t_x$ and $\bm{z}$ at time $t_z$. 

\begin{proof}
Let $t_x,\,t_z$ be two times of the interval $[0,1]$ such that $t_x<t_z$, and let $t \in [t_x,t_z]$. We consider the Gaussian random variable $\bm{\xi}=\left( \bm{X}_{t}, \bm{X}_{t_{z}} \right)$  conditioned on the fact that $\bm{X}_{t_{x}} = \bm{x}$. Its mean can be easily computed from the expression \eqref{eq:OUMult} and reads:
\begin{eqnarray*}
\left( \bm{m}_t, \bm{m}_{t_z} \right) = \left(  \bF(t_{x},t) \bm{x} ,   \bF(t_{x},t_{z}) \bm{x} \right) = \left( \bm{g}(t) \, \bm{g}^{-1}(t_{x}) \, \bm{x} , \bm{g}(t_{z}) \, \bm{g}^{-1}(t_{x}) \, \bm{x} \right)\, ,
\end{eqnarray*}
and its covariance matrix, from equation \eqref{eq:covExp}, reads:
\begin{eqnarray*}
\left[
\begin{array}{cc}
\bm{C}_{t,t}  &  \bm{C}_{t,t_z}  \\
\bm{C}_{t_z,t}  &  \bm{C}_{t_z,t_z}     
\end{array}
\right]
&=&
\left[
\begin{array}{cc}
  \bF(t_{x},t) \bm{h}_{t_{x}}(t_{x},t)   \bF(t_{x},t)^{T}  &    \bF(t_{x},t)\bm{h}_{t_{x}}(t_{x},t)   \bF(t_{x},t_{z})^{T}  \\
  \bF(t_{x},t_{z}) \bm{h}_{t_{x}}(t_{x},t)   \bF(t_{x},t)^{T}  &    \bF(t_{x},t_{z}) \bm{h}_{t_{x}}(t_{x},t_{z})   \bF(t_{x},t_{z})^{T}     
\end{array} 
\right]\\
&=&
\left[
\begin{array}{cc}
 \bm{h}_{t}(t_{x},t)   &  \bm{h}_{t}(t_{x},t) \,  \bF(t,t_{z})^{T}    \\
  \bF(t,t_{z}) \, \bm{h}_{t}(t_{x},t)   &    \bF(t,t_{z}) \, \bm{h}_{t}(t_{x},t_{z}) \,  \bF(t,t_{z})^{T}       
\end{array} 
\right] \, .
\end{eqnarray*}
From there, we apply the conditioning formula for Gaussian vectors (see e.g.~\cite{Bogachev:1998}) to infer the law of $\bm{X}_{t}$ conditionned on $\bm{X}_{t_{x}} = \bm{x}$ and $\bm{X}_{t_{z}} = \bm{z}$, that is the law $\mathcal{N}(\bm{\mu}(t), \bm{\Sigma}(t;t_x,t_z))$ of $\bm{B}_{t}$ where $\bm{B}$ denotes the bridge process obtained by pinning $\bm{X}$ in $t_{x}$ and $t_{z}$. The covariance matrix is given by
\begin{eqnarray*}
\bm{\Sigma}(t;t_x,t_z)
&=&
\bm{C}_{y,y} - \bm{C}_{y,z}\bm{C}^{-1}_{z,z}\bm{C}_{z,y} \, ,\\
&=&
\bm{h}_{t}(t_{x},t) - \bm{h}_{t}(t_{x},t) \left( \bm{h}_{t}(t_{x},t_{z}) \right)^{-1} \bm{h}_{t}(t_{x},t) \, ,\\
&=&
\bm{h}_{t}(t_{x},t) \left( \bm{h}_{t}(t_{x},t_{z}) \right)^{-1} \bm{h}_{t}(t,t_{z}) \, ,
\end{eqnarray*}
and the mean reads
\begin{eqnarray*}
\bm{\mu}(t) 
&=&
 \bm{m}_{y} + \bm{C}_{y,z} \bm{C}_{z,z}^{-1} \, \left( \bm{z} - \bm{m}_{z} \right) \\
 &=& 
  \bF(t_{x},t) \left( \bm{I}_d - \bm{h}_{t_{x}}(t_{x},t) \left( \bm{h}_{t_{x}}(t_{x},t_{z})\right)^{-1}\right) \, \bm{x} \\
&& \qquad  \qquad  \qquad +
   \bF(t_{z},t) \bm{h}_{t_{z}}(t_{x},t) \left( \bm{h}_{t_{z}}(t_{x},t_{z})\right)^{-1} \, \bm{z} \, ,\\
 &=&
 \underbrace{  \bF(t_{x},t) \,  \bm{h}_{t_{x}}(t,t_{z}) \,   \left( \bm{h}_{t_{x}}(t_{x},t_{z})\right)^{-1} }_{{\displaystyle \bm{\mu}^l(t;t_x,t_z)}}\, \cdot \, \bm{x} \\
 && \qquad  \qquad  \qquad +
  \underbrace{    \bF(t_{z},t) \bm{h}_{t_{z}}(t_{x},t) \left( \bm{h}_{t_{z}}(t_{x},t_{z})\right)^{-1} }_{\displaystyle \bm{\mu}^r(t;t_x,t_z)}\, \cdot \, \bm{z} \, ,
\end{eqnarray*}
where we have used the fact that $\bm{h}_{t_{x}}(t_{x},t_{z}) = \bm{h}_{t_{z}}(t_{x},t) +\bm{h}_{t_{x}}(t,t_{z})$.
The regularity of the thus-defined functions $\bm{\mu}_{x}$ and $\bm{\mu}_{z}$ directly stems from the regularity of the flow operator $ \bF$.
Moreover, since for any $0 \leq t,u \leq 1$, we observe that $ \bF(t,t) = \bm{I}_d$ and $h_{u}(t,t) = \bm{0}$, we clearly have $\bm{\mu}_{x}(t_{x})=\bm{\mu}_{y}(t)=\bm{I}_d$ and $\bm{\mu}_{x}(t)=\bm{\mu}_{y}(t_{x})=0$. 
\end{proof}

\begin{remark}
	Note that these laws can also be computed using the expression of the density of the processes, but involves more intricate calculations. Moreover, this approach allows to compute the probability distribution of the Gauss-Markov bridge as a process (i.e allows to compute the covariances), but since this will be of no use in the sequel, we do not provide the expressions. 
\end{remark}



\subsection{The Multi-Resolution Description of Gauss-Markov processes}

Recognizing the Gauss property  and the Markov property as the two crucial elements for a stochastic process to be expanded {\it \`{a} la} L\'{e}vy-Cesielsky, our approach first supposes to exhibit  bases of deterministic functions that would play the role of the Schauder bases for the Wiener process.
In this regard, we first expect such functions to be continuous and compactly supported on increasingly finer supports (i.e. sub-intervals of the definition interval $[0,1]$) in a similar nested binary tree structure. 
Then, as in the L\'{e}vy-Cesielsky construction, we envision that, at each resolution (i.e. on each support), the partially constructed process (up to the resolution of the support) has the same conditional expectation as the Gauss-Markov process when conditioned on the endpoints of the supports. 
The partial sums obtained with independent Gaussian coefficients of law $\mathcal{N}(0,1)$ will thus approximate the targeted Gauss-Markov process in a multi-resolution fashion, in the sense that, at every resolution, considering these two processes on the intervals endpoints yields finite-dimensional Gaussian vectors of same law.


\subsubsection{Nested Structure of the Sequence of Supports}\label{sec:supports}

Here, we define the nested sequence of  segments that constitute the supports of the multi-resolution basis.
We construct such a sequence by recursively partitioning the interval $[0,1]$.\\
More precisely, starting from $S_{1,0} = [l_{1,0},r_{1,0}]$ with $l_{1,0}=0$ and $r_{1,0}=1$, we iteratively apply the following operation.
Suppose that, at the $n$th step, the interval $[0,1]$ is decomposed into $2^{n-1}$ intervals  $S_{n,k} = [l_{n,k},r_{n,k}]$, called supports, such that $l_{n,k+1} = r_{n,k}$ for $0 \leq k<2^{n-1}$. 
Each of these intervals is then subdivided into two child intervals, a left child $S_{n+1,2k}$ and a right child $S_{n+1,2k+1}$, and the subdivision point $r_{n+1,2k}=l_{n+1,2k+1}$ is denoted $m_{n,k}$. 
Therefore, we have defined three sequences of real $l_{n,k}$, $m_{n,k}$, and $r_{n,k}$ for $n>0$ and $0\leq k < 2^{n-1}$ satisfying  $l_{0,0} = 0 \leq l_{n,k}<m_{n,k}<r_{n,k} \leq r_{0,0} = 1$ and 
\begin{eqnarray}
l_{n+1,2k} = l_{n,k},  \quad m_{n,k} = r_{n+1,2k}=l_{n+1,2k+1},  \quad r_{n+1,2k+1} = r_{n,k} \, .
\nonumber
\end{eqnarray}
where we have posit $l_{0,0}=0$ and $r_{0,0}=1$ and $S_{0,0} = [0,1]$.
The resulting sequence of supports $\lbrace S_{n,k} \, ; n \leq 0 , \, 0 \leq k < 2^{n-1} \rbrace$ clearly has a binary tree structure. \\
For the sake of compactness of notations, we define $\I$ the set of indices
 \[\I = \bigcup_{n < N} \I_{n} \quad \mathrm{with} \quad \I_{N} = \left \{ (n,k)\in \N^2 \; \vert \; 0 < n \leq N \, ,\; 0 \leq k < 2^{n-1} \right\},\]
and for $N>0$, we define $D_N = \{m_{n,k}\;,\; (n,k) \in \I_{N-1} \} \cup \{0,1\} $, the set of endpoints of the intervals $S_{N,k}$. 
We additionally require that there exists $\rho \in (0,1)$ such that for all $(n,k) \in \I$
$ \max (r_{n,k}-m_{n,k}, m_{n,k}-l_{n,k}) < \rho (r_{n,k}-l_{n,k})$
which in particular implies that 
\begin{eqnarray*} \label{eq:stepCons}
\lim_{n \to \infty} \sup_{k} r_{n,k} -l_{n,k} = 0 \, .
\end{eqnarray*}
and ensures that the set of endpoints $\cup_{N \in \mathrm{N}} D_N$ is dense in $[0,1]$. 
The simplest case of such partitions is the dyadic partition of $[0,1]$, where the end points for $(n,k)\in \I$ read
\begin{equation}
l_{n,k} = k \, 2^{-n+1}, \quad m_{n,k} = (2k+1)2^{-n}, \quad r_{n,k} = (k+1) 2^{-n+1} \, .
\nonumber
\end{equation}
in which case the endpoints are simply the dyadic points $\cup _{N} D_N  = \{ k2^{-N} \, \vert \, 0 \leq k \leq 2^{N} \}$.
\begin{figure}
	\centering
		\includegraphics[width=.6\textwidth]{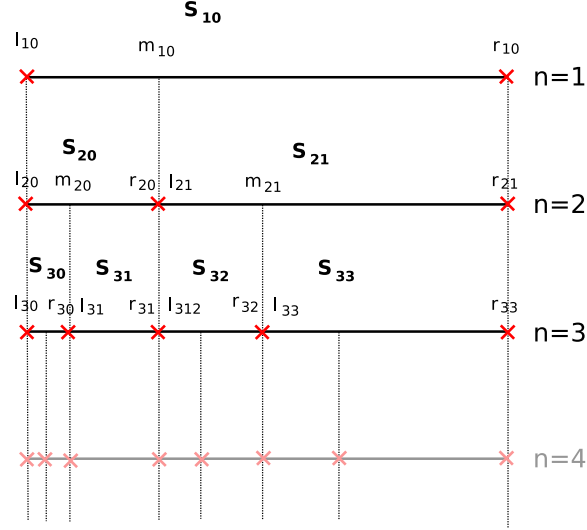}
	\caption{A sequence of nested intervals}
	\label{fig:partition}
\end{figure}
Figure \ref{fig:partition} represents the global architecture of the nested sequence of intervals.
 
The nested structure of the supports, together with constraint of continuity of the bases elements, implies that only  a finite number of coefficients are needed to construct the exact value of the process at a given endpoints, thus providing us with an exact schema to simulate sample values of the process on the endpoint up to an arbitrary resolution, as we will further explore.


\subsubsection{Innovation processes  for Gauss-Markov processes}

For $\bm{X}_t$ a multidimensional Gauss-Markov process, we call multi-resolution description of a process the sequence of conditional expectations on the nested sets of endpoints $D_n$. In details, if we denote $\mathcal{F}_{N}$ the filtration generated by $\lbrace \bX_{t} ; t \in D_N \rbrace$ the values of the process at the endpoints $D_N$ of the partition, we introduce the sequence of Gaussian processes $(\bm{Z}^N_t)_{N\geq 1}$ defined by:
\begin{equation*}\label{eq:ZDef}
	\bZ_t^N = \mathbb{E} \left [ \bm{X}_t \big \vert \mathcal{F}_{N} \right] =  \ExpN{ \bX_{t}} \, .
\end{equation*}
These processes $\bm{Z}^N$ constitute a martingale taking values in the processes spaces that can be naturally viewed as an interpolation of the process $\bX$ sampled at the increasingly finer partitions times$D_N$, since for all $t \in D_N$ we have $\bm{Z}^N_{t} = \bm{X}^N_{t}$. 
The innovation process $(\bdelta^N_t, \mathcal{F}_{t}, t \in [0,1] )$ is defined as the update transforming the process $ \bZ_t^N$ into $\bZ_t^{N+1}$, i.e.
\begin{equation}\label{eq:innov}
	\bdelta^N_t = \bZ_t^{N+1} - \bZ_t^N.
\end{equation}
It corresponds to the difference the additional knowledge of the process at the points $m_{N,k}$ make on the conditional expectation of the process. This process satisfies the following important properties that found our multi-resolution construction

\begin{proposition}
	The innovation process $\bm{\delta}^N_t$ is a centered Gaussian process independent of the processes $\bZ^n_t$ for any $n\leq N$. For $s\in S_{N,k}$ and $t\in S_{N,p}$ with $k,p \in \I_N$, the covariance of the innovation process reads:
	\begin{equation} \label{eq:covInnov}
	\ExpN{\bdelta^{N}_t \cdot {\left( \bdelta^{N}_s \right)}^{T}}
	= \begin{cases}
	\bm{\mu}_{N,k}(t) \cdot \bm{\Sigma}_{N,k}	\cdot {\bm{\mu}_{N,k}(t)}^{T} \quad & \text{if } k=p\\
	\bm{0} & \text{if } k\neq p,
\end{cases}
	\end{equation}
	where 
	\[\bm{\mu}_{N,k}(t) = \begin{cases} \bm{\mu}^r(t;\,l_{N,k},\,m_{N,k}) & \qquad t\in [l_{N,k},m_{N,k}]\\
	\bm{\mu}^l(t;\,m_{N,k},\,r_{N,k}) & \qquad t\in [m_{N,k},r_{N,k}].
	\end{cases}\]
	with $ \bm{\mu}^l$, $\bm{\mu}^r$ and $\bm{\Sigma}_{N,k}=\bm{\Sigma}(m_{N,k};l_{N,k},r_{N,k})$ as defined in Proposition~\ref{bridgeProp}. 
\end{proposition}

\begin{proof}
	Because of the Markovian property of the process $\bX$, the law of the process $\bZ^N$ can be computed from the bridge formula derived in Proposition~\ref{bridgeProp} and we have:
	\[\bZ^N_t= \bm{\mu}^l(t;\,l_{N,k},\,r_{N,k}) \cdot \bX_{l_{N,k}} + \bm{\mu}^r(t;\,l_{N,k},\,r_{N,k}) \cdot \bX_{r_{N,k}}.\]
and 
\[\bZ^{N+1}_{t}= 
\begin{cases}
	\bm{\mu}^l(t;\,l_{N,k},\,m_{N,k}) \cdot \bX_{l_{N,k}} + \bm{\mu}^r(t;\,l_{N,k},\,m_{N,k}) \cdot \bX_{m_{N,k}} \, , \\ \qquad \qquad  \qquad \mathrm{for} \qquad t\in [l_{N,k}, m_{N,k}] \, ,\\
	\\
	\bm{\mu}^l(t;\,m_{N,k},\,r_{N,k}) \cdot \bX_{m_{N,k}} + \bm{\mu}^r(t;\,m_{N,k},\,r_{N,k}) \cdot \bX_{r_{N,k}} \, , \\ \qquad \qquad  \qquad \mathrm{for} \qquad t\in [l_{N,k}, m_{N,k}] \, .
\end{cases}\]
Therefore, the innovation process can be written for $t\in S_{N,k}$ as 
\[\bm{\delta}^N_t = \bm{\mu}_{N,k}^N(t)\cdot \bX_{m_{N,k}} + \bm{\nu}^N(t) \cdot \bm{Q}_t^N\]
where $\bm{Q}_t^N$ is a $\mathcal{F}_N$ measurable process, $\bm{\nu}^N(t)$ a deterministic matrix function and 
\[\bm{\mu}_{N,k}(t) = \begin{cases} \bm{\mu}^r(t;\,l_{N,k},\,m_{N,k}) & \qquad t\in [l_{N,k},m_{N,k}]\\
\bm{\mu}^l(t;\,m_{N,k},\,r_{N,k}) & \qquad t\in [m_{N,k},r_{N,k}].
\end{cases}\]
The expressions of $\bm{\nu}$ and $\bm{Q}$ are quite complex, but are highly simplified when noting that 
\begin{align*}
	\mathbbm{E}[\bdelta_t^{N} \vert \mathcal{F}_N] &=\mathbbm{E}[\bZ_t^{N+1} \vert \mathcal{F}_N]-\bZ_t^{N}\\
	&= \mathbbm{E}\Big[\mathbbm{E}\big[\bZ_t \vert \mathcal{F}_{N+1} \big] \vert \mathcal{F}_N \Big]-\bZ_t^{N}\\
	&= \bm{0}
\end{align*}
directly implying that $\bm{\nu}(t) \cdot \bm{Q}_t^N =\bm{\mu}^N(t)\cdot \bZ^N_{m_{N,k}}$ and yielding the remarkably compact expression:
\begin{equation} \label{eq:innovExp}
\bdelta^{N}_t = \bm{\mu}_{N,k}(t) \cdot \big( \bX_{m_{N,k}} - \bZ^{N}_{m_{N,k}} \big) \, .
\end{equation}
This process is a centered Gaussian process. Moreover, observing that it is $\mathcal{F}_N$-measurable, it can be written as:
\[
\bdelta^{N}_t = \bm{\mu}_{N,k}(t) \cdot \big( \lbrace \bX_{m_{N,k}} \vert \mathcal{F}_N\rbrace - \bZ^{N}_{m_{N,k}} \big).
\]
and the process $\lbrace \bX_{m_{N,k}} \vert \mathcal{F}_N\rbrace$ appears as the Gauss-Markov bridge conditioned at times $l_{N,k}$ and $r_{N,k}$, and whose covariance is given by Proposition \ref{bridgeProp} and that has the expression 
\begin{eqnarray}\label{eq:Sigmank}
\bm{\Sigma}_{N,k}&=&\bm{\Sigma}(m_{N,k};l_{N,k},r_{N,k}) \nonumber\\
&=&
\bm{h}_{m_{n,k}}(l_{n,k},m_{n,k}) \big( \bm{h}_{m_{n,k}}(l_{n,k},r_{n,k}) \big)^{-1} \bm{h}_{m_{n,k}}(m_{n,k},r_{n,k})  \, .
\end{eqnarray}
 Let $(s,t)\in [0,1]^2$, and assume that $s\in S_{N,k}$ and $t\in S_{N,p}$. If $k\neq p$, then because of the Markov property of the process $\bX$, the two bridges are independent and therefore the covariance $\ExpN{\bdelta^{N}_t \cdot {\left( \bdelta^{N}_s \right)}^{T}}$ is zero.
If $k=p$, we have:
\[
\ExpN{\bdelta^{N}_t \cdot {\left( \bdelta^{N}_s \right)}^{T}}
=\bm{\mu}_{N,k}(t) \cdot \bm{\Sigma}_{N,k} \cdot {\bm{\mu}_{N,k}(s)}^{T}.
\]
Eventually, the independence property stems from simple properties of the conditional expectation. Indeed, let $n\leq N$. We have:
\begin{align*}
	\Exp{\bZ^n_t \cdot { \left( \bdelta^{N}_s \right)}^{T}} &= \Exp{ \bZ^n_t \cdot \left(\bZ^{N+1}_s - \bZ^{N}_s \right)^{T}}\\
	& = \Exp{ \Exp{\bX_t\vert \mathcal{F}_n} \cdot \left(\Exp{\bX_s^T\vert \mathcal{F}_{N+1}} - \Exp{\bX_s^T\vert \mathcal{F}_{N}} \right)}\\
	& = \Exp{ \Exp{\bX_t\vert \mathcal{F}_n} \cdot \Exp{\bX_s^T\vert \mathcal{F}_{N+1}}} - \Exp{ \Exp{\bX_t\vert \mathcal{F}_n} \cdot \Exp{\bX_s^T\vert \mathcal{F}_{N}}}\\
	&= \Exp{ \bZ^n_t (\bZ^n_s)^T} - \Exp{ \bZ^n_t (\bZ^n_s)^T}\\
	&=\bm{0}
\end{align*}
and the fact that a zero covariance between two Gaussian processes implies the independence of these processes concludes the proof.
\end{proof}


\subsubsection{Derivation of the Candidate Multi-Resolution Bases of Functions}\label{sec:Identification}

We deduce from the previous proposition the following fundamental theorem of this paper
\begin{theorem}\label{theo:basis}
	For all $N\in \N$, there exists a collection of functions $\bpsi_{N,k}:[0,1]\mapsto \R^{d\times d}$ that are zero outside the sub-interval $S_{N,k}$ and such that in distribution we have:
	\[ \bdelta^N_t = \sum_{k \in \I_N} \bpsi_{N,k}(t) \cdot \bXi_{N,k}\]
	where $\bXi_{N,k}$ are independent $d$-dimensional standard normal random variables (i.e. of law $\mathcal{N}(0,\bm{I}_d)$). This basis of functions is unique up to an orthogonal transformation.
\end{theorem}

\begin{proof}
	The two processes $\bdelta^N_t$ and $\bm{d}^N_t \stackrel{def}{=} \sum_{k \in \I_N} \bpsi_{N,k}(t) \cdot \bXi_{N,k}$ are two Gaussian processes of mean zero. Therefore, we are searching for functions $\bpsi_{N,k}$ vanishing outside $S_{N,k}$ and ensuring that the two processes have the same probability distribution. A necessary and sufficient condition for the two processes to have the same probability distribution is to have the same covariance function (see e.g.~\cite{Bogachev:1998}). We therefore need to show the existence of a collection of functions $ \bpsi_{N,k}(t)$ functions that vanish outside the sub-interval $S_{N,k}$ and that ensure that the covariance of the process $\bm{d}^N$ is equal to the covariance of $\bdelta^N$. Let $(s,t) \in [0,1]$ such that $s\in S_{N,k}$ and $t\in S_{N,p}$. If $k\neq p$, the assumption fact that the functions $\psi_{N,k}$ vanish outside $S_{N,k}$ implies that 
	\[\Exp{\bm{d}^N_t \cdot (\bm{d}^N_s)^T} = \bm{0}.\]
	If $k=p$, the covariance reads:
	\begin{align*}
		\Exp{\bm{d}^N_t \cdot (\bm{d}^N_s)^T} &= \Exp{\bpsi_{N,k}(t)\cdot \bXi_{N,k}\cdot \bXi_{N,k}^T\cdot (\bpsi_{N,k}(s))^T}\\
		& = \bpsi_{N,k}(t)\cdot (\bpsi_{N,k}(s))^T
	\end{align*}
	which needs to be equal to the covariance of $\bdelta^N$, namely:
	\begin{equation}\label{eq:psichar}
		\bpsi_{N,k}(t)\cdot (\bpsi_{N,k}(s))^T = \bm{\mu}_{N,k}(t)\cdot \bm{\Sigma}_{N,k}\cdot (\bm{\mu}_{N,k}(s))^T.
	\end{equation} 
	Therefore, since $\bm{\mu}_{N,k}(m_{N,k})=\bm{I}_d$, we have:
	\[\bpsi_{N,k}(m_{N,k})\cdot (\bpsi_{N,k}(m_{N,k}))^T=\bm{\Sigma_{N,k}}\]
	meaning that $\bm{\sigma}_{N,k} \stackrel{def}{=} \bpsi_{N,k}(m_{N,k})$ is a square root of the symmetric positive matrix $\bm{\Sigma_{N,k}}$. Moreover, by fixing $s=m_{N,k}$ in equation \eqref{eq:psichar}, we get:
	\[\bpsi_{N,k}(t)\cdot \bm{\sigma}_{N,k}^T = \bm{\mu}(t) \cdot \bm{\sigma}_{N,k}\cdot \bm{\sigma}_{N,k}^T \]
	Eventually, since by assumption we have $\bm{\Sigma}_{N,k}$ invertible, so is $\bm{\sigma}_{N,k}$, and the functions $\bpsi_{N,k}$ can be written as:
	\begin{equation}\label{eq:candPsiBasis}
		\bm{\psi}_{N,k}(t) =  \bm{\mu}_{N,k}(t) \cdot \bm{\sigma}_{N,k} \,
	\end{equation}
	with $\bm{\sigma}_{N,k}$ a square root of $\bm{\Sigma}_{N,k}$. Square roots of positive symmetric matrices are uniquely defined up to an orthogonal transformation. Therefore, all square roots of $\bm{\Sigma}_{N,k}$ are related by orthogonal transformations $\bm{\sigma'}_{N,k} = \bm{\sigma}_{N,k}\cdot \bm{O}_{N,k}$ where $\bm{O}_{N,k}\cdot \bm{O}_{N,k}^T=\bm{I}_d$. This property immediately extends to the functions $\bpsi_{N,k}$ we are studying: two different functions $\bpsi_{N,k}$ and $\bpsi_{N,k}'$ satisfying the theorem differ from an orthogonal transformation $\bm{O}_{N,k}$. 
We proved that, for $\bpsi_{N,k}(t)\cdot \bXi_{N,k}$ to have the same law as $\bdelta^N(t)$ in the interval $S_{N,k}$, the function $\bpsi_{N,k}$ with support in $S_{N,k}$ are necessarily of the form $\bm{\mu}_{N,k}(t) \cdot \bm{\sigma}_{N,k}$. It is straightforward to show the sufficient condition that provided such a set of functions, the processes $\bdelta^N_t$ and $\bm{d}^{N}_t$ are equal in law, which ends the proof of the theorem.
\end{proof}

Using the expressions obtained in Proposition \ref{bridgeProp}, we can make completely explicit the form of the basis in terms of the functions $\bm{f}$, $\bm{g}$ and $\bm{h}$. 
\begin{equation}\label{eq:BasisExpressionFirst}
	\bm{\psi}_{n,k}(t) = \begin{cases}
	\bm{g}(t) \,  \bm{g}^{-1}(m_{n,k})  \, \bm{h}_{m_{n,k}}(l_{n,k},t) \, \big( \bm{h}_{m_{n,k}}((l_{n,k},m_{n,k}) \big)^{-1} \bm{\sigma}_{n,k}, \\ \qquad \qquad  \qquad \mathrm{for} \qquad l_{n,k}\leq t \leq m_{n,k} \, ,\\
	\\
	\bm{g}(t) \,  \bm{g}^{-1}(m_{n,k}) \, \bm{h}_{m_{n,k}}(t,r_{n,k}) \big( \bm{h}_{m_{n,k}}(m_{n,k},r_{n,k}) \big)^{-1} \bm{\sigma}_{n,k}, \\ \qquad \qquad  \qquad \mathrm{for} \qquad m_{n,k} \leq t \leq r_{n,k} \, ,\\
\end{cases}
 \end{equation}
and $\bm{\sigma}_{n,k}$ satisfies
\begin{equation*}
\bm{\sigma}_{n,k} \cdot \bm{\sigma}_{n,k}^{T}
=
 \bm{h}_{m_{n,k}}(l_{n,k},m_{n,k}) \left( \bm{h}_{m_{n,k}}(l_{n,k},r_{n,k}) \right)^{-1} \bm{h}_{m_{n,k}}(m_{n,k},r_{n,k}) \, .
\end{equation*}
Note that $\bm{\sigma}_{n,k}$ can be defined uniquely as the symmetric positive square root, or as the lower triangular matrix resulting from the Cholesky decomposition of $\bm{\Sigma}_{n,k}$.\\
Let us now define the function $\bpsi_{0,0} : [0,1]\mapsto \R^{d\times d}$ such that the process $\bpsi_{0,0}(t) \cdot \bXi_{0,0}$ has the same covariance as $\bZ^0_t$, which is computed using exactly the same technique as developed in the proof of Theorem \ref{theo:basis} and that has the expression
\begin{equation*}
\bpsi_{0,0}(t) = \bm{g}(t)  \,  \bm{h}_0(l_{0,0},t) \,  \left( \bm{h}_0(l_{0,0},r_{0,0})\right)^{-1}   \, \bm{g}^{-1}(r_{0,0}) \, \bm{\sigma}_{0,0} \, ,
\end{equation*}
for $\bm{\sigma}_{0,0}$ a square root of $\bm{C}_{r_{0,0}}$ the covariance matrix of $\bX_{r_{0,0}}$ which from equation \eqref{eq:covExp} reads: \[ \bF(0,1) \bm{h}_{0}(1,1)  \bF(0,1)^{T}=\bm{g}(1)\bm{h}_0(1,1)(\bm{g}(1))^T. \]

We are now in position to show the following corollary of Theorem \ref{theo:basis}
\begin{corollary}\label{cor:basis}
	The Gauss-Markov process $\bZ^N_t$ is equal in law to the process 
	\[\bX^N_t = \sum_{n=0}^{N-1} \sum_{k\in \I_n} \bpsi_{n,k}(t)\cdot \bXi_{n,k} \] 
	where $\bXi_{n,k}$ are independent standard normal random variables $\mathcal{N}(0,\bm{I}_d)$. 
\end{corollary}

\begin{proof}
	We have:
	 	\begin{align*}
			\bZ^N_t &= (\bZ^{N}_t - \bZ^{N-1}_t) + (\bZ^{N-1}_t - \bZ^{N-2}_t) + \ldots + (\bZ^{2}_t-\bZ^1_t) + \bZ^1_t\\
			&= \sum_{n=1}^{N-1} \bdelta^{n}_t + \bZ^1_t\\
			&= \sum_{n=1}^{N-1} \sum_{k\in \I_n} \bpsi_{n,k}(t) \cdot \bXi_{n,k} + \bpsi_{0,0}(t) \cdot \bXi_{0,0}\\
			&= \sum_{n=0}^{N-1} \sum_{k\in \I_n} \bpsi_{n,k}(t) \cdot \bXi_{n,k}
		\end{align*}
\end{proof}

We therefore identified a collection of functions $\lbrace  \bpsi_{n,k} \rbrace_{(n,k) \in \I}$ that allows a simple construction of the Gauss-Markov process iteratively conditioned on increasingly finer partitions of the interval $[0,1]$. We will show that this sequence $\bZ^N_t$ converges almost surely towards the Gauss-Markov process $\bX_t$ used to construct the basis, proving that these finite-dimensional continuous functions $\bZ^N_t$ form an asymptotically accurate description of the initial process. 
Beforehand, we rigorously study the Hilbertian properties of the collection of functions we just defined.



\section{The Multi-Resolution Schauder Basis Framework}

The above analysis motivates the introduction of a set of functions $\lbrace  \bpsi_{n,k} \rbrace_{(n,k) \in \I}$ we now study in details.
In particular, we enlighten the structure of the collection of functions $\bpsi_{n,k}$ as a Schauder basis in a certain space $\mathcal{X}$ of continuous functions from $[0,1]$ to $\R^{d}$. The Schauder structure was defined in \cite{Schauder:27, Schauder:28}, and its essential characterization is the unique decomposition property: namely that every elements $x$ in $\mathcal{X}$ can be written as a well-formed linear combination
\begin{equation*}
x = \sum_{(n,k) \in \I} \bpsi_{n,k} \cdot \bm{\xi}_{n,k} \, ,
\end{equation*}
and that  the coefficients satisfying the previous relation are unique.


\subsection{System of Dual Bases}

To complete this program, we need to introduce some quantities that will play a crucial role in expressing the family $\bpsi_{n,k}$ as a Schauder basis for some given space. In equation \eqref{eq:BasisExpressionFirst}, two constant matrices  $\R^{d \times d}$ appear, that will have a particular importance in the sequel for $(n,k)$ in $\I$ with $n \neq 0$:
\begin{eqnarray*}
\bm{L}_{n,k} &=&\bm{g}^{T}(m_{n,k}) \big( \bm{h}_{m_{n,k}}(l_{n,k},m_{n,k}) \big)^{-1}  \bm{\sigma}_{n,k} \\
&=& \big( \bm{h}(l_{n,k},m_{n,k}) \big)^{-1} \, \bm{g}^{-1}(m_{n,k})\, \bm{\sigma}_{n,k} \, , \\
\bm{R}_{n,k} &=& \bm{g}^{T}(m_{n,k}) \big( \bm{h}_{m_{n,k}}(m_{n,k},r_{n,k}) \big)^{-1}  \bm{\sigma}_{n,k}  \\
&=& \big( \bm{h}(m_{n,k},r_{n,k}) \big)^{-1}\, \bm{g}^{-1}(m_{n,k})\,\bm{\sigma}_{n,k} \,  , 
\end{eqnarray*}
where $\bm{h}$ stands for $\bm{h}_0$. We further define the matrix
\[\bm{M}_{n,k} = \bm{g}^{T}(m_{n,k}) \,{\bm{\sigma}^{-1}_{n,k}}^{T} \] 
and we recall that $\bm{\sigma}_{n,k}$ is a square root of $\bm{\Sigma}_{n,k}$, the covariance matrix of $\bX_{m_{n,k}}$ knowing $\bX_{l_{n,k}}$ and $\bX_{r_{n,k}}$ given in equation \eqref{eq:Sigmank}. We stress that the matrices $\bm{L}_{n,k}$, $\bm{R}_{n,k}$, $\bm{M}_{n,k}$ and $\bm{\Sigma}_{n,k}$ are all invertible and satisfy the important following properties:

\begin{proposition}\label{computProp}
For all $(n,k)$ in $\I$, $n \neq 0$, we have:
\renewcommand{\theenumi}{\roman{enumi}} 
\begin{enumerate}
	\item $\bm{M}_{n,k} = \bm{L}_{n,k} + \bm{R}_{n,k}$ and 
	\item $\bm{\Sigma}^{-1}_{n,k} = \big( \bm{h}_{m_{n,k}}(l_{n,k},m_{n,k}) \big)^{-1} + \big( \bm{h}_{m_{n,k}}(m_{n,k},r_{n,k}) \big)^{-1}$ .
\end{enumerate}
\end{proposition}

\noindent To prove this proposition, we first establish the following simple lemma of linear algebra:
\begin{lemma}\label{lemAlg}
Given two invertible matrices ${A}$ and ${B}$ in $GL_{n}(\R)$ such that ${C} = {A} + {B}$ is also invertible, if we posit ${D} = {A} {C}^{-1} {B}$, we have the following properties: 
\renewcommand{\theenumi}{\roman{enumi}} 
\begin{enumerate}
	\item ${D} ={A} {C}^{-1} {B} = {B} {C}^{-1} {A} $ and 
	\item ${D}^{-1} ={A}^{-1} + {B}^{-1}$
\end{enumerate}
\end{lemma}

\begin{proof}
i) ${D} = {A} {C}^{-1} {B} = ({C}- {B}){C}^{-1} {B}=  {B}- {B}{C}^{-1} {B} = {B}( {I} -{C}^{-1} {B}) =  {B} {C}^{-1}({C} -{B})= {B} {C}^{-1} {A}$.\\
ii) $({A}^{-1} + {B}^{-1}){D} = {A}^{-1} {D} + {B}^{-1} {D} = {A}^{-1}  {A} {C}^{-1} {B}+ {B}^{-1}  {B} {C}^{-1} {A} = {C}^{-1}( {B} + {A}) = {C}^{-1}{C} = {I}$.
\end{proof}

\begin{proof}[Proof of Proposition \ref{computProp}] $ \; $\\
	\noindent (ii) directly stems from Lemma \ref{lemAlg}, item (ii) by posing $A = \bm{h}_{m_{n,k}}(l_{n,k},m_{n,k})$, $B= \bm{h}_{m_{n,k}}(m_{n,k},r_{n,k})$ and $C= A+B= \bm{h}_{m_{n,k}}(l_{n,k},r_{n,k})$. Indeed, the lemma implies that 
		\begin{eqnarray*}
		D^{-1} &=& {A}^{-1} {C} {B}^{-1} \\
		&=& \bm{h}_{m_{n,k}}(l_{n,k},m_{n,k})^{-1} \bm{h}_{m_{n,k}}(l_{n,k},r_{n,k}) \bm{h}_{m_{n,k}}(l_{n,k},m_{n,k})^{-1} \\
		&= &\bm{\Sigma}_{n,k}^{-1}
		\end{eqnarray*}
	(i) We have: 
	\begin{align*}
		\bm{L}_{n,k} + \bm{R}_{n,k} &= \bm{g}(m_{n,k})^T\big( \bm{h}(l_{n,k},m_{n,k})^{-1} + \bm{h}(m_{n,k},r_{n,k})^{-1}\big)  \bm{\sigma}_{n,k}\\
		&= \bm{g}(m_{n,k})^T \bm{\Sigma}_{n,k}^{-1} \bm{\sigma}^{n,k}\\
		&=\bm{g}(m_{n,k})^T \left(\bm{\sigma}_{n,k}^{-1}\right)^{T}
	\end{align*}
	which ends the demonstration of the proposition.
\end{proof}

Let us define $\bm{L}_{0,0} =  \left( \bm{h}(l_{0,0},r_{0,0})\right)^{-1}   \, \bm{g}^{-1}(r_{0,0}) \, \bm{\sigma}_{0,0}$. With this notations we define the functions in a compact form as:

\begin{definition}\label{def:psiDef}
For every $(n,k)$ in $\I$ with $n \neq 0$, the continuous functions $\bpsi_{n,k}$ are defined on their support $S_{n,k}$ as
\begin{equation}\label{eq:psiDef}
	\bm{\psi}_{n,k}(t) = \begin{cases}
	\bm{g}(t)   \, \bm{h}(l_{n,k},t)  \cdot \bm{L}_{n,k}\, ,& l_{n,k}\leq t \leq m_{n,k} \, ,\\
	\bm{g}(t) \,   \bm{h}(t,r_{n,k}) \cdot \bm{R}_{n,k}  \, ,& m_{n,k} \leq t \leq r_{n,k} \, ,\\
\end{cases}
 \end{equation}
and the basis element $\bpsi_{0,0}$ is given on $[0,1]$ by 
\begin{equation*}
\bpsi_{0,0}(t) = \bm{g}(t)  \,  \bm{h}(l_{0,0},t) \cdot \bm{L}_{0,0} \, .
\end{equation*}
\end{definition}

The definition implies that the $\bpsi_{n,k}$ are continuous functions in the space of piecewise derivable functions with piecewise continuous derivative which takes value zero at zero.
We denote such a space $C^{1}_{0}\big([0,1],\R^{d \times d}\big)$.

Before studying the property of the functions $\bpsi_{n,k}$, it is worth remembering that their definitions include the choice of a square root $\bm{\sigma}_{n,k}$ of $\bm{\Sigma}_{n,k}$.
Properly speaking, there is thus a class of bases $\bpsi_{n,k}$ and all the points we develop in the sequel are valid for this class.
However, for the sake of simplicity, we consider from now on that the basis under scrutiny results from choosing the unique square root $\bm{\sigma}_{n,k}$ that is lower triangular with positive diagonal entries (Cholesky decomposition).


\subsubsection{Underlying System of Orthonormal Functions}\label{ssec:Phink}

We first introduce a family of functions $\bphi_{n,k}$ and show that it constitutes an orthogonal basis on a certain Hilbert space. The choice of this basis can seem arbitrary at first sight, but the definition of these function will appear natural for its relationship with the functions $\bpsi_{n,k}$ and $\bm{\Phi}_{n,k}$ that is made explicit in the sequel, and the mathematical rigor of the argument leads us to choose this apparently artificial introduction. 

\begin{definition}\label{def:phiDef}
For every $(n,k)$ in $\I$ with $n \neq 0$, we define a continuous function $\bphi_{n,k}:[0,1] \to \R^{m \times d}$ which is zero outside its support $S_{n,k}$ and has the expressions:
\begin{equation} \label{eq:phiDef}
\bphi_{n,k}(t)= \begin{cases}
										\bm{f}(t)^{T} \cdot \bm{L}_{n,k}\, , & \textrm{if $l_{n,k} \! \leq \! t \! < m_{n,k} $} \: ,\\
										\bm{f}(t)^{T} \cdot \bm{R}_{n,k}\, , & \textrm{if $m_{n,k} \! \leq \! t \! < \! r_{n,k}$} \: .
\end{cases}
\end{equation}
The basis element $\bphi_{0,0}$ is defined on $[0,1]$ by 
\begin{equation}\label{defPhi00}
\bphi_{0,0}(t) =  \bm{f}(t)^T \cdot \bm{L}_{0,0}\, .
\end{equation}
\end{definition}

Remark that the definitions of make apparent that fact that these two families of functions are linked  for all $(n,k)$ in $\I$ through the simple relation
\begin{equation}\label{eq:simplePsiPhi}
\bpsi_{n,k}' = \bm{\alpha} \cdot  \bpsi_{n,k} + \sqrt{\bm{\Gamma}} \cdot \bphi_{n,k} \, .
\end{equation}
Moreover, this collection of functions $\bphi_{n,k}$ constitutes an orthogonal basis of functions, in the following sense:

\begin{proposition}\label{orthoProp}
Let $L^{2}_{\bm{f}}$ be the closure of 
\begin{equation*}
\lbrace \bm{u}:[0,1] \to \R^{m} \, \vert \, \exists \; \bm{v} \in L^{2}\big([0,1],\R^{d}\big)  \, ,  \bm{u} = \bm{f}^{T} \cdot \bm{v} \rbrace \, ,
\end{equation*}
equipped with the natural norm of $L^{2}\big([0,1],\R^{m}\big)$. It is a Hilbert space, and moreover for all $0 \leq j <d$,
the family of functions $c_{j}\left(\bphi_{n,k}\right)$ defined as the columns of $\bphi_{n,k}$, namely 
\[c_{j}\left(\bphi_{n,k}\right) =\left[ (\bphi_{n,k})_{i,j} \right]_{0 \leq i < m},\] 
forms a complete orthonormal basis of $L^{2}_{\bm{f}}$. 
\end{proposition}

\begin{proof}
The space $L^{2}_{\bm{f}}$ is clearly a Hilbert space as a closed subspace of the larger Hilbert space $L^{2}\big([0,1],\R^{m}\big)$ equipped with the standard scalar product:
\begin{equation*}
\forall \; \bm{u},\bm{v} \in L^{2}\big([0,1],\R^{d}\big) \, , \quad (\bm{u},\bm{v}) = \int_{0}^{1} \bm{u}(t)^{T} \cdot {\bm{v}(t)} \, dt \, . 
\end{equation*}
We now proceed to demonstrate that the columns of $\bphi_{n,k}$ form an orthonormal family which generates a dense subspace of $L^{2}_{\bm{f}}$.
To this end, we define $M\big([0,1],\R^{m \times d}\big)$ the space of functions 
\begin{equation*}
\lbrace \bm{A}:[0,1] \to \R^{m \times d} \, \vert \, \forall \; j : \, 0 \leq j < d \, ,  t \mapsto \left[ \bm{A}_{i,j}(t) \right]_{0 \leq i < m}  \in L^{2}\big([0,1],\R^{m }\big) \rbrace \, ,
\end{equation*}
that is, the space of functions which take values in the set of $m \times d$-matrices  whose columns are in $L^{2}\big([0,1],\R^{m }\big)$.
This definition allows us to define the bilinear function $\mathcal{P}:  M\big([0,1],\R^{m \times d}\big) \times M\big([0,1],\R^{m \times d}\big) \to \R^{d \times d}$ as
\begin{equation*}
\mathcal{P}(\bm{A},\bm{B}) =  \int_{0}^{1} \bm{A}(t)^{T} \cdot\bm{B}(t) \, dt  \quad \mathrm{satisfying} \quad \mathcal{P}(\bm{B},\bm{A})  = \mathcal{P}(\bm{A},\bm{B})^{T} \, ,
\end{equation*}
and we observe that the columns of $\bphi_{n,k}$ form an orthonormal system if and only if
\begin{eqnarray*}
\forall \; \big((p,q), (n,k)\big) \in \I \times \I \, , \quad \mathcal{P}(\bphi_{n,k}, \bphi_{p,q} ) = \int_{0}^{1} \bphi_{n,k}(t)^{T} \cdot \bphi_{p,q}(t) \, dt = \delta^{n,k}_{p,q} \, \bm{I}_d \, ,
\end{eqnarray*}
where $\delta^{n,k}_{p,q}$ is the Kronecker delta function, whose value is $1$ if  $n=p$ and $k=q$, and $0$ otherwise.\\
First of all, since the functions $\bphi_{n,k}$ are zero outside the interval $S_{n,k}$, the matrix  $\mathcal{P}(\bphi_{n,k}, \bphi_{p,q} ) $ are non-zero only if  $S_{n,k} \cap S_{p,q} \neq \emptyset$.
In such cases, assuming that $n \neq p$ and for example that $n<p$, we necessarily have $S_{n,k}$ strictly included in $S_{p,q}$:
more precisely, $S_{n,k}$ is either included in the left-child support $S_{p+1,2q}$ or in the  right-child support $S_{p+1,2q+1}$ of $S_{p,q}$.
In both cases, writing the matrix $\mathcal{P}(\bphi_{n,k}(t),\bphi_{p,q} )$ shows that it is expressed as a matrix product whose factors include $\mathcal{P}(\bphi_{n,k}, \bm{f}^{T} )$.
We then show that:
\begin{eqnarray*}
	\mathcal{P}(\bphi_{n,k},\bm{f}^{T})
 & = &
\int_{0}^{1} \bphi_{n,k}(t)^{T} \cdot \bm{f}(t)^{T} \\
 & = &
\bm{L}_{n,k}^{T} \cdot \int_{l_{n,k}}^{m_{n,k}} \bm{f}(u) \cdot \bm{f}^{T}(u) \,du   - \bm{R}_{n,k}^{T} \cdot \int_{m_{n,k}}^{r_{n,k}} \bm{f}(u) \cdot \bm{f}^{T}(u) \, du \, , \\
 & = &
\bm{L}_{n,k}^{T} \cdot \bm{h}(l_{n,k},m_{n,k})   - \bm{R}_{n,k}^{T} \cdot \bm{h}(m_{n,k},r_{n,k})    \\
 & = & 
\bm{\sigma}_{n,k}^{T} \,\bm{g}^{-1}(m_{n,k})^{T}
- 
\bm{\sigma}_{n,k}^{T} \, \bm{g}^{-1}(m_{n,k})^{T},
\end{eqnarray*}
which entails that $\mathcal{P}(\bphi_{n,k},\bm{f}^{T})=0$ if $n < p$.
If $n > p$, we use the fact that $\mathcal{P}(\bphi_{n,k},\bphi_{p,q} ) = {\mathcal{P}(\bphi_{p,q},\bphi_{n,k} )}^{T}$, and we conclude that $\mathcal{P}(\bphi_{n,k},\bphi_{p,q} ) = \bm{0}$ from the preceding case.
Fot $n=p$, we directly compute for $n>0$ the only non-zero term
\begin{eqnarray*}
\mathcal{P}(\bphi_{n,k},\bphi_{n,k} )
 & = &
\bm{L}_{n,k}^{T} \cdot \int_{l_{n,k}}^{m_{n,k}} \bm{f}(u) \cdot \bm{f}^{T}(u) \,du \cdot \bm{L}_{n,k} \\
&& \qquad + \;  \bm{R}_{n,k}^{T} \cdot \int_{m_{n,k}}^{r_{n,k}} \bm{f}(u) \cdot \bm{f}^{T}(u) \, du \cdot \bm{R}_{n,k} \, , \\
 & = & 
\bm{\sigma}_{n,k}^{T} \, \bm{g}^{-1}(m_{n,k})^{T} \, \big( \bm{h}(l_{n,k},m_{n,k}) \big)^{-1} \, \bm{g}^{-1}(m_{n,k})\, \bm{\sigma}_{n,k}  \\
&& \qquad + \; \bm{\sigma}_{n,k}^{T} \,\bm{g}^{-1}(m_{n,k})^{T}\, \big( \bm{h}(m_{n,k},r_{n,k}) \big)^{-1} \, \bm{g}^{-1}(m_{n,k})\, \bm{\sigma}_{n,k}
 \, .
\end{eqnarray*}
Using the passage relationship between the symmetric functions $\bm{h}$ and $\bm{h}_{m_{n,k}}$ given in equation \eqref{eq:hChainRule}, we can then write
 \begin{eqnarray*}
\mathcal{P}(\bphi_{n,k},\bphi_{n,k} )
 & = &
\bm{\sigma}_{n,k}^{T} \, \big( \bm{h}_{m_{n,k}}(l_{n,k},m_{n,k}) \big)^{-1} \, \bm{\sigma}_{n,k} \\
&& \qquad \qquad + \bm{\sigma}_{n,k}^{T} \, \big( \bm{h}_{m_{n,k}}(m_{n,k},r_{n,k}) \big)^{-1} \, \bm{\sigma}_{n,k} \, . 
\end{eqnarray*}
Proposition \ref{computProp} implies that $ \bm{h}_{m_{n,k}}(l_{n,k},m_{n,k})^{-1} + \bm{h}_{m_{n,k}}(m_{n,k},r_{n,k})^{-1} = \bm{\Sigma}^{-1}_{n,k} = ( \bm{\sigma}^{-1}_{n,k})^{T} \bm{\sigma}_{n,k}^{-1}$
which directly implies that  $\mathcal{P}(\bphi_{n,k},\bphi_{n,k}^{T}) = \bm{I}_d$.
For $n=0$, a computation of the exact same flavor yields that $\mathcal{P}(\bphi_{0,0},\bphi_{0,0}) = \bm{I}_d$.
Hence, we have proved that the collection of columns of $ \bphi_{n,k}$ forms an orthonormal family of functions in $L^{2}_{\bm{f}}$
(the definition of $ \bphi_{n,k}$ clearly states that its columns can be written in the form of elements of $L^{2}_{\bm{f}}$).

The proof now amounts showing the density of the family of functions we consider. Before showing this density property, we introduce for all $(n,k)$ in $\I$ the functions $\bm{P}_{n,k}: [0,1] \to \R^{d \times d}$ with support on $S_{n,k}$ defined by:
\begin{equation*}
\bm{P}_{n,k}(t)
=
\left\{
\begin{array}{cccc}
  & \bm{L}_{n,k}  &\mathrm{if} & l_{n,k} \leq t < m_{n,k}  \\
 - &\bm{R}_{n,k}   &\mathrm{if} &    m_{n,k} \leq t < r_{n,k}
\end{array}
\right.
\quad \mathrm{n \neq 0} \quad \mathrm{and} \quad
\bm{P}_{0,0}(t) = \bm{L}_{0,0} 
\end{equation*}
Showing that the family of columns of $\bphi_{n,k}$ is dense in $L^{2}_{\bm{f}}$ is equivalent to show that the column vectors of the matrices $\bm{P}_{n,k}$ seen as a function of $t$, are dense in $L^{2}\big([0,1], \R^{d} \big)$.
It is enough to show that the span of such functions contains the family of piecewise continuous $\R^{d}$-valued functions that are constant on $S_{n,k}$, $(n,k)$ in $\I$
(the density of the endpoints of the partition $\cup_{N \in \mathbb{N}}D_{N}$ entails that the latter family generates $L^{2}\big([0,1],\R^{d}\big)$). \\
In fact, we show  that the span of functions 
\begin{equation*}
V_{N} = \mathrm{span} \left\{ t \mapsto  c_{j}(\bm{P}_{n,k})(t) \, \Big \vert \, 0 \leq j < d  \, , (n,k) \in \I_{N} \right\}
\end{equation*}
is exactly equal to the space $K_{N}$ of piecewise continuous functions from $[0,1]$ to $\R^{d}$ that are constant on the supports $S_{N+1,k}$, for any $(N+1,k)$ in $\I$.
The fact that $V_{N}$ is included in $K_{N}$ is clear from the fact that the matrix-valued functions $\bm{P}_{N,k}$ are defined constant on the support $S_{N+1,k}$, for $(N,k)$
in $I$. \\

We prove that $K_{N}$ is included in $V_{N}$ by induction on $N \leq 0$. The property is clearly true at rank $N = 0$ since  $\bm{P}_{0,0}$ is then equal to the constant invertible matrix $\bm{L}_{0,0}$.
Assuming the proposition true at rank $N-1$ for a given $N>0$, let us consider a piecewise continuous function $\bm{c}:[0,1] \to \R^{d}$ in $K_{N-1}$.
Remark that for every $(N,k)$ in $\I$, the function $\bm{c}$ can only take two values on $S_{N,k}$ and can have discontinuity jump in $m_{N,k}$: let us denote these jumps as
\begin{equation*}
\bm{d}_{N,k} = \bm{c}(m_{N,k}^{+}) - \bm{c}(m_{N,k}^{-}) \, .
\end{equation*}
Now, remark that for every $(N,k)$ in $\I$, the matrix-valued functions $\bm{P}_{N,k}$ takes only two matrix values on $S_{N,k}$, namely $\bm{L}_{N,k}$ and $-\bm{R}_{N,k}$.
From Proposition \ref{computProp}, we know that $\bm{L}_{N,k}+\bm{R}_{N,k} = \bm{M}_{N,k}$ is invertible.
This fact directly entails that there exists  vectors $\bm{a}_{N,k}$, for any $(N,k)$ in $\I$, such that $\bm{d}_{N,k} = \big(\bm{L}_{N,k}+\bm{R}_{N,k}\big)(-\bm{a}_{N,k})$.
We then necessarily have that the function $\bm{c}' = \bm{c} + \bm{P}_{n,k} \cdot \bm{a}_{n,k}$ is piecewise constant on the supports $S_{N,k}$, $(N,k)$ in $\I$.
By recurrence hypothesis,  $\bm{c}' $ belongs to $V_{N-1}$, so that $\bm{c}$ belongs to $V_{N}$, and we have proved that $K_{N} \subset V_{N}$. \\ 
Therefore, the space generated by  the column vectors $P_{n,k}$ is dense in $L^{2}[0,1]$, which completes the proof that the functions $t \mapsto \left[ (\bphi_{n,k}(t))_{i,j} \right]_{0 \leq i < m}$ form a complete orthonormal family of $L^{2}[0,1]$.
\end{proof}

The fact that the column functions of $\bphi_{n,k}$ form a complete orthonormal system of $L^{2}_{\bm{f}}$ directly entails the following decomposition of the identity on $L^{2}_{\bm{f}}$:

\begin{corollary}\label{cor:identDecomp}
If  $\delta$ is the real delta Dirac function, we have
\begin{eqnarray}
 \sum_{(n,k) \in \I} \bphi_{n,k}(t) \cdot \bphi^{T}_{n,k}(s) = \delta(t-s) {Id}_{L^{2}_{\bm{f}}} \, ,
\end{eqnarray} 
\end{corollary}

\begin{proof}
Indeed it easy to verify that for all $\bm{v}$ in $L^{2}_{\bm{f}}$,   we have for all $N>0$
\begin{eqnarray*}
\int_{U} \sum_{(n,k) \in \I_{N}} \left( \bphi_{n,k}(t) \cdot \bphi^{T}_{n,k}(s) \right) \, \bm{v}(s) \, ds 
& = &
 \sum_{(n,k) \in \I_{N}}  \bphi_{n,k}(t) \cdot \mathcal{P} \left( \bphi_{n,k}   , \bm{v} \right) \, ,\\
 & = &
 \sum_{(n,k) \in \I_{N}}  \sum_{p=0}^{d-1}  c_{p}(\bphi_{n,k}) \big( c_{p}(\bphi_{n,k}) , \bm{v} \big) \, ,
\end{eqnarray*}
where $\big( c_{p}(\bphi_{n,k}) , \bm{v} \big)$ denotes the inner product in $L^{2}_{\bm{f}}$ between $\bm{v}$ and the $p$-column of $\bpsi_{n,k}$.
Therefore, by the Parseval identity, we have in the $L^{2}_{\bm{f}}$ sense
\begin{equation*}
\int_{U} \sum_{(n,k) \in \I} \left( \bphi_{n,k}(t) \cdot \bphi^{T}_{n,k}(s) \right) \, \bm{v}(s) \, ds  = \bm{v}(t) \, .
\end{equation*}
\end{proof}

From now on, abusing language, we will say that the family of $\R^{m \times d}$-valued functions $\bphi_{n,k}$ is an orthonormal family of functions to refer to the fact that the columns of such matrices form orthonormal set of $L^{2}_{\bm{f}}$. 
We now make explicit the relationship between this orthonormal basis and our functions $(\bpsi_{n,k})$ derived in our analysis of multi-dimensional Gauss-Markov processes.


\subsubsection{Generalized Dual Operators}\label{ssec:DualOps}

\paragraph{The integral operator $\mathcal{K}$ }
The basis $\bphi_{n,k}$  is of great interest in this article for its relationship to the functions $\bpsi_{n,k}$ that naturally arise in the decomposition of Gauss-Markov processes. Indeed, the collection $\bpsi_{n,k}$ can be generated from the orthonormal basis $\bphi_{n,k}$ through the action of the integral operator $\mathcal{K}$ defined on $L^{2}\big([0,1],\R^{m} \big)$ into $L^{2}\big([0,1],\R^{d} \big)$ by:
\begin{equation}\label{eq:KDef}
\bm{u} \mapsto \mathcal{K}[\bm{u}] 
=
\left\{ t \mapsto  \bm{g}(t)  \cdot \int_{U} \mathbbm{1}_{[0,t]}(s) \bm{f}(s) \, \bm{u}(s) \, ds \right\} 
\end{equation}
where $U\supset [0,1]$ is an open set and where, for any set $E\subset U$, $\mathbbm{1}_{E} (\cdot)$ denotes the indicator function of $E$.
Indeed, realizing that $ \mathcal{K}$ acts on $M\big([0,1],\R^{m \times d}\big)$ into $M\big([0,1],\R^{d \times d}\big)$ through 
\begin{equation*}
\forall \bm{A} \in M\big([0,1],\R^{m \times d}\big) \, , \quad \mathcal{K}[\bm{A}] = \Big[ \mathcal{K}[c_{0}( \bm{A})], \ldots ,\mathcal{K}[c_{d-1}(\bm{A})]\Big] \, ,
\end{equation*}
where $c_{j}(\bm{A})$ denote the $j$-th  $\R^{m}$-valued column function of $\bm{A}$, we easily see that for all $(n,k)$ in $\I$, $0 \leq t \leq1$:
\begin{equation}\label{eq:psinkKPhi}
\bpsi_{n,k}(t) = \bm{g}(t) \cdot \int_{0}^{t} \bm{f}(s) \cdot \bphi_{n,k}(s) \, ds = \mathcal{K}[ \bphi_{n,k}](t) \, .
\end{equation}
It is worth noticing that the introduction of the operator $\mathcal{K}$ can be considered natural since it characterizes the centered Gauss-Markov process $\bX$ through formally  writing $\bX = \mathcal{K}[\dot{\bm{W}}]$ .\\
In order to exhibit a dual family of functions to the basis $\bpsi_{n,k}$, we further investigate the property of the integral operator $\mathcal{K}$.
In particular, we study the existence of an inverse operator $\mathcal{D}$, whose action on the orthonormal basis $\bm{\phi}_{n,k}$ will conveniently provide us with a dual basis to $\bpsi_{n,k}$.
Such an operator does not always exist, nevertheless, under special assumptions,  it can be straightforwardly expressed as a generalized differential operator.

\paragraph{The differential operator $\mathcal{D}$ }

Here, we make the assumptions that $m=d$, that for all $t$, $\bm{f}(t)$ is invertible in $\R^{d \times d}$ and that $\bm{f}$ and $\bm{f}^{-1}$ have continuous derivatives, which especially implies that $L^{2}_{\bm{f}}=L^{2}\big( \R^{d} \big)$. 
In this setting, we define the space $D_{0}\big(U,\R^{d}\big)$ of functions in $C_{0}^{\infty}\big(U,\R^{d}\big)$ that are zero at zero, and denote by $D_{0}'\big(U, \R^{d} \big)$ its dual in the space of distributions (or generalized functions). Under the assumptions just made, the operator $ \mathcal{K}:D_{0}\big(U,\R^{d}\big) \mapsto D_{0}\big(U,\R^{d}\big)$ admits the differential operator $\mathcal{D}:D_{0}\big(U,\R^{d}\big) \mapsto D_{0}\big(U,\R^{d}\big)$  defined  by
\begin{equation*}\label{eq:DDef}
	\bm{u} \in D_{0}\big(U, \R^{d} \big) \mapsto \mathcal{D}[\bm{u}] 
	=
	\left\{ t \mapsto \bm{f}^{-1}(t) \frac{d}{dt} \! \Big( \bm{g}^{-1}(t) \bm{u}(t) \Big) \right\}
 	\, .
\end{equation*}
as its inverse, that is, when restricted to $D_{0}\big(U,\R^{d}\big)$, we have $\mathcal{D} \circ \mathcal{K} = \mathcal{K} \circ \mathcal{D}=  Id$ on $D_{0}(U,\R^{d})$. The dual operators of $\mathcal{K}$ and $\mathcal{D}$ are expressed, for any $\bm{u}$ in $D_{0}\big(U,\R^{d} \big)$, as 
\begin{eqnarray*}
 \mathcal{D}^{*}[\bm{u}] &=& \left\{  t \mapsto   -{\left( \bm{g}^{-1}(t) \right)}^{T} \displaystyle \frac{d}{dt} \left( \left(\bm{f}^{-1}(t)\right)^{T} \, \bm{u}(t) \right) \right\} \, ,\\
 \mathcal{K}^{*}[\bm{u}] &=&  \left\{  t \mapsto   -\bm{f}(t)^{T} \displaystyle  \int_{U} \mathbbm{1}_{[0,t]}(s) \, \bm{g}^{T}(s) \, \bm{u}(s) \, ds \right\} \, .
\end{eqnarray*}
They satisfy (from the properties of $\mathcal{K}$ and $\mathcal{D}$) $\mathcal{D}^{*} \circ \mathcal{K}^{*} = \mathcal{K}^{*} \circ \mathcal{D}^{*}= Id$ on $D_{0}(U,\R^{d})$.
By dual pairing, we extend the definition of the operators $\mathcal{K}$, $\mathcal{D}$ as well as their dual operators, to the space of generalized function $D_0'(U,\R^{d})$. 
In details, for any distribution $T$ in $D_{0}'(U,\R^{d})$ and test function $\bm{u}$ in $D_{0}(U,\R^{d})$, define $ \mathcal{K}$ and $\mathcal{K}^{*}$ by
\begin{eqnarray*}
\left(  \mathcal{D} [T] , \bm{u} \right) =  \left( T , \mathcal{D}^{*}[\bm{u}] \right)  \quad \mathrm{and} \quad
\left(  \mathcal{K} [T] , \bm{u} \right) = \left( T , \mathcal{K}^{*}[\bm{u}] \right) \, ,
\end{eqnarray*}
and reciprocally for the dual operators $\mathcal{D}^{*}$ and $ \mathcal{K}^{*}$. \\

\paragraph{Candidate Dual Basis }
We are now in a position to use the orthonormality of the $\bphi_{n,k}$ to infer a dual family to the basis of elements $\bpsi_{n,k}$.
For  any function $\bm{u}$ in $L^{2}\big(U, \R^{d} \big)$, the generalized function $\mathcal{K}[\bm{u}]$ belongs to $C_{0}\big(U, \R^{d} \big)$, the space of continuous function which are zero at zero. 
We equip this space with the uniform norm and denote its topological dual $R_{0}\big(U, \R^{d} \big)$, the set of  $d$-dimensional Radon measures with $R_{0}\big(U, \R^{d} \big) \subset D'_{0}\big(U, \R^{d} \big)$.
Consequently, operating in the Gelfand triple
\begin{equation}\label{eq:Gelf}
C_{0}\big(U, \R^{d} \big) \subset  L^{2}\big(U, \R^{d} \big) \subset R_{0}\big(U, \R^{d} \big) \, ,
\end{equation}
we can write, for any function $\bm{u}$, $\bm{v}$ in $L^{2}\big(U, \R^{d} \big) \subset R_{0}\big(U,\R^{d}\big)$ 
\begin{equation*}
\left(\bm{u},\bm{v} \right) = \Big( \big( \mathcal{D} \circ \mathcal{K} \big) [\bm{u}], \bm{v} \Big)  = \left( \mathcal{K}  [\bm{u}], \mathcal{D}^{*} [\bm{v}] \right)
\end{equation*}
The first equality stems from the fact that, when  $\mathcal{K}$ and $\mathcal{D}$ are seen as generalized functions, they are still inverse of each other, so that in particular $\mathcal{D} \circ \mathcal{K} = Id$ on $L^{2}\big(U, \R^{d} \big)$.
The  dual pairing associated with the Gelfand triple \eqref{eq:Gelf} entails the second equality where $\mathcal{D}^{*} $ is the generalized operator defined on $D'_{0}\big(U,\R^{d}\big)$ and where $\mathcal{D}^{*} [\bm{v}]$ is in $R_{0}\big(U,\R^{d}\big)$.\\
As a consequence, defining the functions $\bdelta_{n,k}$ in  $R_{0}\big(U,\R^{d \times d}\big)$,  the $d \times d$-dimensional space of  Radon measures, by 
\begin{equation*}
\bdelta_{n,k} = \mathcal{D}^{*}(\bphi_{n,k}) = \left[ \mathcal{D}^{*}\big[ c_{1}(\bphi_{n,k}) \big] , \ldots, \mathcal{D}^{*}\big[ c_{d}((\bphi_{n,k})] \big] \right]
\end{equation*}
provides us with a family of $d \times d$-generalized functions which are dual to the family $\bpsi_{n,k}$ in the sense that, for all $\big((n,k),(p,q)\big)$ in $\I \times \I$, we have
\begin{equation*}
\mathcal{P}\big( \bdelta_{n,k} , \bpsi_{p,q} \big) = \delta^{n,k}_{p,q} \, \bm{I}_d \, ,
\end{equation*}
where the definition of $\mathcal{P}$ has been extended through dual pairing:
given any $\bm{A}$ in  $R_{0}\big(U,\R^{m \times d}\big)$ and any $\bm{B}$ in $ C_{0}\big(U,\R^{m \times d}\big)$, we have
\begin{equation*}
\mathcal{P}(\bm{A} , \bm{B} ) = \Big[ \big(  c_{i}(\bm{A}) , c_{j}(\bm{B}) \big) \Big]_{0\leq i,j <d}
\end{equation*}
with $\big(  c_{i}(\bm{A}) ,  c_{j}(\bm{B}) \big)$ denoting the dual pairing between the $i$-th column of $\bm{A}$ taking value in $R_{0}\big(U,\R^{d}\big)$ and the $j$-th column of $\bm{B}$ taking value in $C_{0}\big(U,\R^{d}\big)$.
Under the favorable hypothesis of this section, the $d \times d$-generalized functions $\bdelta_{n,k}$ can actually be easily computed, since considering the definition of $\bphi_{n,k}$ shows that the functions $(\bm{f}^{-1})^{T} \cdot \bphi_{n,k}$ have support $S_{n,k}$ and are constant on $S_{n+1,2k}$ and $S_{n+1,2k+1}$ in $\R^{d \times d}$.
Only the discontinuous jumps in $l_{n,k}$, $m_{n,k}$ and $r_{n,k}$ intervene, leading to express for $(n,k)$ in $\I$, $n \neq 0$
\begin{equation*}
\bdelta_{n,k}(t) = (\bm{g}(t)^{-1})^{T}   \cdot \left( \bm{M}_{n,k} \, \delta(t-{m_{n,k}})  - \left( \bm{L}_{n,k} \, \delta(t-{l_{n,k}}) + \bm{R}_{n,k} \, \delta(t-{r_{n,k}})\right) \right)
\end{equation*}
and $\bdelta_{0,0}(t) = (\bm{g}(t)^{-1})^{T} \cdot \bm{L}_{0,0}$, where $\delta(\cdot)$ denotes the standard delta Dirac function (centered in $0$). These functions can be extended to the general setting of the article since its expressions do not involve the assumptions made on the invertibility and smoothness of $\bm{f}(t)$. We now show that these functions, when defined in the general setting still provide a dual basis of the functions $\bpsi_{n,k}$.


\subsubsection{Dual Basis of Generalized Functions}

The expression of the basis $\bdelta_{n,k}$ that has been found under favorable assumptions makes no explicit reference to these assumptions.
It suggest to define a  functions $\bdelta_{n,k}$  formally as linear combination of Dirac functions acting by duality on $C_{0}\big(U,\R^{d \times d}\big)$: 

\begin{definition}\label{deltaDef}
	For $(n,k)$ in $\I$, the family of generalized functions $\bdelta_{n,k}$ in $R_{0}\big(U,\R^{d \times d}\big)$ is given $n \neq 0$ by 
	\begin{equation*}  
	 \bdelta_{n,k}(t)
	=
	(\bm{g}(t)^{-1})^{T}  \cdot \left( \bm{M}_{n,k} \, \delta(t-{m_{n,k}}) - \left( \bm{L}_{n,k} \, \delta(t-{l_{n,k}}) + \bm{R}_{n,k} \, \delta(t-{r_{n,k}})\right) \right)  ,
	 \, ,
	\end{equation*}
	and $\bdelta_{0,0}(t) = (\bm{g}(t)^{-1})^{T} \cdot \bm{L}_{0,0}$, where $\delta$ is the standard Dirac distribution.
\end{definition}

Notice that the basis $\bdelta_{n,k}$ is defined for the open set $U$.
For the sake of consistency, we extend the definition of the families $\bpsi_{n,k}$ and $\bphi_{n,k}$ on $U$ by setting them to zero on $U \setminus [0,1]$, except for $\bpsi_{0,0}$ which is continued for  $t>1$ by a  continuous function $\bm{c}$ that is compactly supported in $[1,a)$ for a given $a$ in $U$, $a>1$ and satisfy $\bm{c}(1)= \bpsi_{0,0}(1)$.\\
We can now formulate:

\begin{proposition}
\label{dualProp}
Given the dual pairing in $C_{0}(U) \subset L^{2}(U) \subset R(U)$ where $U$ is a bounded open set of $\mathbb{R}$ containing $[0,1]$, the family of continuous functions $\bpsi_{n,k}$ in $C_{0}(U)$ admits for dual family in $R(U)$, the set of distributions $\bdelta_{n,k}$. \\
\end{proposition}

\begin{proof}
We have to demonstrate that, for all $\big( (n,k),(p,q)\big)$ in $\I \times \I$,
\begin{equation*}
\mathcal{P}\big( \bdelta_{p,q} , \bpsi_{n,k} \big) = \delta^{n,k}_{p,q} \, \bm{I}_d \, .
\end{equation*}
Suppose first, $n,p>0$.
If $p < n$, $\mathcal{P}\big( \bdelta_{n,k} , \bpsi_{p,q} \big)$ can only be non-zero if the support $S_{p,q}$ is strictly included in $S_{n,k}$.
We then have
\begin{eqnarray*}
\mathcal{P}\big( \bdelta_{p,q} , \bpsi_{n,k} \big) 
&=&
\bm{M}^{T}_{p,q} \, \bm{g}^{-1}(m_{p,q})  \bpsi_{n,k}(m_{p,q})  \\
&& \qquad - \left( \bm{L}^{T}_{p,q} \, \bm{g}^{-1}(l_{p,q})  \bpsi_{n,k}(l_{p,q})  + \bm{R}^{T}_{p,q} \, \bm{g}^{-1}(r_{p,q})  \bpsi_{n,k}(r_{p,q}) \right) \, .
\end{eqnarray*}
Assume $S_{p,q}$ is to the left of $m_{n,k}$, that is, $S_{p,q}$ is a left-child of $S_{n,k}$ in the nested binary tree of supports and write
\begin{eqnarray*}
\mathcal{P}\big( \bdelta_{p,q} , \bpsi_{n,k} \big) 
= 
\Big( \bM^{T}_{p,q} \, \bh(l_{n,k},m_{p,q})   -   \bL^{T}_{p,q} \, \bh(l_{n,k},l_{p,q})\big) -  \bR^{T}_{p,q}  \bh(l_{n,k},r_{p,q})\big) \Big) \bL_{n,k}  \, .
\end{eqnarray*}
Using the fact that $\bM_{p,q} = \bL_{p,q}+ \bR_{p,q}$ and that the function $h(x,y)$, as any integral between $x$ and $y$, satisfies the chain rule $\bh(x,y)=\bh(x,z)+\bh(z,y)$ for all $(x,y,z)$, we obtain:
\begin{eqnarray*}
\mathcal{P}\big( \bdelta_{p,q} , \bpsi_{n,k} \big) 
&=& 
 \Big(- \bL^{T}_{p,q}\,\big( \bh(l_{n,k},m_{p,q}) - \bh(l_{n,k},l_{p,q})\big)   \\
 && \quad  + \bR^{T}_{p,q} \, \big( h(l_{n,k},m_{p,q}) - h(l_{n,k},r_{p,q})\big) \Big) \bL_{n,k}\\
&=&
\Big(- \bL^{T}_{p,q}\, \bh(l_{p,q},m_{p,q}) + \bR^{T}_{p,q} \, h(r_{p,q},m_{p,q})  \Big) \bL_{n,k}\\
&=& \Big(- \bm{\sigma}_{p,q}^T (\bm{g}^{-1}(m_{p,q}))^T (\bh(l_{p,q},m_{p,q})^{-1})^T\cdot \bh(l_{p,q},m_{p,q}) \\
&& \quad + \bm{\sigma}_{p,q}^T (\bm{g}^{-1}(m_{p,q}))^T (\bh(m_{p,q},r_{p,q})^{-1})^T\cdot \bh(m_{p,q},r_{p,q})\Big)\cdot \bL_{n,k}\\
&=&\bm{0}
\end{eqnarray*}
The same result is true if $S_{p,q}$ is a right-child of $S_{n,k}$ in the nested binary tree of supports.
If $p=n$, necessarily the only non-zero term is for $q=p$, i.e. 
\begin{eqnarray*}
\mathcal{P}\big( \bdelta_{p,q} , \bpsi_{n,k} \big) 
&=&
 \bM_{n,k}^{T} \, \bgg^{-1}(m_{n,k}) \, \bpsi(m_{n,k}) \\
&=&
 \bM_{n,k}^{T} \,  \bh(l_{n,k},m_{n,k}) \, \bL_{n,k}\\
&=&
\bm{\sigma}_{p,q}^{-1} \, \bgg(m_{n,k}) \,\bh(l_{n,k},m_{n,k})\, \bh(l_{n,k},m_{n,k})^{-1}\,\bgg^{-1}(m_{n,k}) \, \bm{\sigma}_{p,q} \\
& = & \bm{I}_{d}
 \, .
\end{eqnarray*}
If $p > n$, $\mathcal{P}\big( \bdelta_{n,k} , \bpsi_{p,q} \big)$ can only be non-zero if the support $S_{n,k}$ is included in $S_{p,q}$, but then $\bpsi_{n,k}$ is zero in $l_{p,q}$, $m_{p,q}$, $r_{p,q}$ so that $\mathcal{P}\big( \bdelta_{n,k} , \bpsi_{p,q} \big) = \bm{0}$.\\
Otherwise, if $n=0$ and $p>0$, we directly have 
\begin{eqnarray*}
\mathcal{P}\big( \bdelta_{p,q} , \bpsi_{0,0} \big) 
&=&
\bm{M}^{T}_{p,q} \, \bm{g}^{-1}(m_{p,q})  \bpsi_{0,0}(m_{p,q}) \\  
&& \quad   - \left( \bm{L}^{T}_{p,q} \, \bm{g}^{-1}(l_{p,q})  \bpsi_{0,0}(l_{p,q})  + \bm{R}^{T}_{p,q} \, \bm{g}^{-1}(r_{p,q})  \bpsi_{0,0}(r_{p,q}) \right) \, , \\
&=& 
 \Big(- \bL^{T}_{p,q} \, \bh(l_{p,q},m_{p,q}) + \bR^{T}_{p,q} \, \bh(m_{p,q},r_{p,q}) \Big) \bL_{0,0}\, , \\
&=&
 \Big( - \bm{\sigma}_{p,q}^{T} \, {\bm{g}^{-1}(m_{p,q})}^{T}+ \bm{\sigma}_{p,q}^{T} \, {\bm{g}^{-1}(m_{p,q})}^{T} \big) \Big) \bL_{n,k} \, ,\\
&=& \bm{0} \, .
 \nonumber
\end{eqnarray*}
Finally, if $p=0$, given the simple form of $\bdelta_{0,0}$ with a single Dirac function centered in $r_{0,0}$, we clearly have $\mathcal{P}\big( \bdelta_{0,0} , \bpsi_{n,k} \big) = 0$, and if $n>0$
\begin{eqnarray*}
\mathcal{P}\big( \bdelta_{0,0} , \bpsi_{0,0} \big) \, ,
&=& \bL^{T}_{0,0}\, \bh(l_{0,0},r_{0,0})\, \bL_{0,0} \, ,\\
&=& \bm{\sigma}_{0,0}^{T} \,  \big(\bm{g}^{-1}(r_{0,0})\big)^{T} \, \bL_{0,0} \, ,\\
&=& \bm{\sigma}_{0,0}^{T} \,  \big(\bm{g}^{-1}(r_{0,0})\big)^{T} \,  \big(\bh(l_{0,0},r_{0,0})\, \bL_{0,0}) \big)^{-1} \,  \bm{g}^{-1}(r_{0,0}) \, \bm{\sigma}_{0,0} \, , \\
&=& \bm{\sigma}_{0,0}^{T} \, \bm{C}_{r_{0,0}}^{-1} \,  \bm{\sigma}_{0,0} \, ,
\end{eqnarray*}
and using the fact that (by definition of $\bm{\sigma}_{0,0}$) we have $\bm{\sigma}_{0,0} \cdot \bm{\sigma}_{0,0}^{T} = \bm{C}_{r_{0,0}}$, this last expression is equal to:
\[\mathcal{P}\big( \bdelta_{0,0} , \bpsi_{0,0} \big)=\bm{\sigma}_{0,0}^{T} \, \big(\bm{\sigma}_{0,0}^T\big)^{-1} \cdot \big(\bm{\sigma}_{0,0}\big)^{-1}  \,  \bm{\sigma}_{0,0} = I_d,\]
which completes the proof. 
\end{proof}

This proposition directly implies the main result of the section:
\begin{theorem}\label{thm:schauderbasis}
	The collection of functions $(\bpsi_{n,k}\;;\;(n,k)\in \I)$ constitute a Schauder basis of functions on $C_0(U,\R^d).$
\end{theorem}
 
This theorem provides us with a complementary view of stochastic processes: in addition to the standard sample paths view, this structure allows to see Gauss-Markov processes as coefficients on the computed basis. This duality is developed in the sequel.


\subsection{The Sample Paths Space} 


\subsubsection{The Construction Application} \label{app:ConsAppl}
The Schauder basis of functions with compact supports constructed allows to define functions by considering the coefficients on this basis, which constitute sequences of real numbers in the space: 
\begin{equation} 
\uxi \Omega = \big \lbrace \bm{\xi}={\lbrace \bm{\xi}_{n,k} \rbrace}_{\I} \, ; \, \forall \, (n,k) \in \I, \bm{\xi}_{n,k} \in \mathbb{R}^{d} \big \rbrace =  \left( \mathbb{R}^{d} \right)^{\I}
\nonumber.
\end{equation}
We equip $\uxi \Omega$ with the uniform norm $\Vert \bm{\xi} \Vert_{\infty} = \sup_{(n,k) \in \I}\vert  \bm{\xi}_{n,k} \vert$, where we write $ \vert  \bm{\xi}_{n,k} \vert = \sup_{0 \leq i <d} \vert (\bm{\xi}_{n,k})_{i}\vert$. We denote by $\mathcal{B}\left(\uxi \Omega\right)$ the Borelian sets of the topology induced by the uniform norm and we recall that $C(\uxi \Omega)$, the cylinder sets of $\uxi \Omega$, form a generative family of Borelian sets. Remark that not any sequence of coefficients provide a continuous function, and one needs to assume a certain decrease in the coefficients to get convergence. A sufficient condition to obtain convergent sequences is to consider coefficients in the space:
\[\uxi \Omega' = \Big\{ \bxxi \in \uxi \Omega \; \vert \; \exists \delta \in (0,1), \; \exists N \in \N,\; \forall (n,k) \in \I \setminus \I_{N} ,\, \vert \bxxi_{n,k} \vert < 2^{\frac{n\delta}{2}}\Big\}\] 
This set is clearly a Borelian set of $\uxi \Omega$, since it can be written as a countable intersection and union of cylinder, namely, by denoting $\mathcal{J}$ the set of finite subset of $\mathbb{N}$ and $\delta_{p} = 1 \!-\!1/p$, $p>1$,
\begin{equation}
\uxi \Omega' = \bigcup_{p >1} \bigcup_{J \in \mathcal{J}} \bigcap_{ n \in \mathbb{N} \setminus J} \Big\{ \bxxi  \in \uxi \Omega \, \Big \vert \, \max_{0 \leq k < 2^{n-1}} \vert \bxxi_{n,k} \vert < 2^{\frac{n\delta_{p}}{2}} \Big\} \, .
\nonumber
\end{equation}
It is also easy to verify that it forms a vectorial subspace of $\uxi \Omega$.

After these definitions, we are in position to introduce the following useful function:
\begin{definition}
 We denote by $\bPsi^{N}$ the \emph{partial construction application}:
\begin{equation*}
\bPsi^{N} = \begin{cases} 
	\uxi \Omega & \longrightarrow  C_{0}\big([0,1],\R^{d}\big)\\
  \bxxi  & \longmapsto \sum_{(n,k) \in \I_{N}} \bpsi_{n,k}(t) \cdot \bxi_{n,k} \, .
\end{cases}
\end{equation*}
where the $C_{0}\big([0,1],\R^{d}\big)$ is the $d$-dimensional Wiener space, which is complete under the uniform norm $\Vert \bm{x}\Vert_{\infty} = \sup_{0 \leq t \leq 1} \vert \bm{x}(t) \vert$.
\end{definition}

This sequence of partial construction applications is shown to converge to the construction application in the following:

\begin{proposition}
\label{lem:convergence}
For every $\bxi$ in $\uxi \Omega'$, $ \bPsi^{N}(\bxi)$ converges uniformly toward a continuous function in $C_{0}\big([0,1],\R^{d}\big)$. We will denote this function $\bPsi(\bxi)$, defined as:
\[
\bPsi: \begin{cases}
\uxi \Omega' &\longrightarrow C_0([0,1],\R^d)\\
\bxi &\longmapsto \sum_{(n,k)\in \I}\psi_{n,k}(t) \cdot  \bxi_{n,k}
\end{cases}\]
and this application will be referred to as the construction application. 
\end{proposition}

This proposition is proved in appendix \ref{append:ConstructCoeff}. The image of this function constitutes a subset of the Wiener space continuous functions $C_0([0,1],\R^d)$. Let us now define the vectorial subspace $\ux \Omega' = \bPsi( \uxi \Omega')$ of $C_{0}\big([0,1],\R^{d}\big)$ so that $\bPsi$ appears as a bijection.

It is important to realize that, in the multidimensional case, the space ${}_{ \scriptscriptstyle x \displaystyle}\Omega'$  depends on $\bm{\Gamma}$ and $\bm{\alpha}$  in a non-trivial way.
For instance, assuming $\balpha = 0$, the space ${}_{ \scriptscriptstyle x \displaystyle}\Omega'$ depends obviously crucially on the rank of $\bm{\Gamma}$.
To fix idea, for a given constant  $\bm{\sqrt{\bm{\Gamma}(t)}} = \left[ 0,0 \ldots 1\right]^{T}$ in $\R^{d \times 1}$, we expect the space ${}_{ \scriptscriptstyle x \displaystyle}\Omega'$ to only include sample paths of $C_{0}\big([0,1],\R^{d}\big)$ for which the $n\!-\!1$ first components are constant.
Obviously, a process with such sample paths is degenerate in the sense that its covariance matrix is not invertible.\\
Yet, if we additionally relax the hypothesis that $\balpha \neq 0$,  the space ${}_{ \scriptscriptstyle x \displaystyle}\Omega'$ can be dramatically altered:
if we take 
\[
 \bm{\alpha}(t) =  \left[ \begin{array}{ccccc}
		0&1\\
		&   \ddots & \ddots& \\
		&&\ddots&1\\
		&&&0\\
\end{array}
\right]
\]
the space ${}_{ \scriptscriptstyle x \displaystyle}\Omega'$ will represent the sample space of the $d\!-\!1$-integrated Wiener process, a non-degenerate d-dimensional process we fully develop in the example section.

However, the situation is much simpler in the one-dimensional case: because the uniform convergence of the sample paths is preserved as long as $\alpha$ is continuous and $\Gamma$ is non-zero through \eqref{eq:UnifBound}, the definition $\ux \Omega'$ does not depend on $\alpha$ or $\Gamma$.
Moreover, in this case, the space $\ux \Omega'$ is large enough to contain reasonably regular functions as proved in appendix \ref{append:ConstructCoeff}, Proposition \ref{lem:Hold}. 

In the case of the $d\!-\!1$-integrated Wiener process, the space ${}_{ \scriptscriptstyle x \displaystyle}\Omega'$ clearly contains the functions $\left\{ \bm{f} = \left( f_{d-1}, \ldots ,f_{0}\right) \, \vert \, f_{0} \in H , f'_{i} = f_{i-1} \, , 0 < i < d \right\} $.

This remark does not holds  that the space ${}_{ \scriptscriptstyle x \displaystyle}\Omega'$ does not depend on $\alpha$  as long as $\alpha$ is continuous, because the uniform convergence of the sample paths is preserved through the change of basis of expansion $\psi_{n,k}$ through \eqref{eq:UnifBound}.

We equip the space $\ux \Omega'$ with the topology induced by the uniform norm on $C_{0}\big([0,1],\R^{d}\big)$.
As usual, we denote $\mathcal{B}(\ux \Omega')$ the corresponding Borelian sets.
We prove in Appendix \ref{append:ConstructCoeff} that:

\begin{proposition}\label{lem:bijection}
The function $\bPsi: \left( \uxi \Omega',\mathcal{B}\left(\uxi \Omega' \right)\right)     \rightarrow    \left(\ux \Omega', \mathcal{B}(\ux \Omega')  \right)$ is a bounded continuous bijection.
\end{proposition}

We therefore conclude that we dispose of a continuous bijection mapping the coefficients onto the sample paths, $\bPsi$. We now turn to study its inverse, the coefficient application, mapping sample paths on coefficients over the Schauder basis. 


\subsubsection{The Coefficient Application} \label{app:CoefAppl}

In this section, we introduce and study the properties of the following function:
\begin{definition}\label{def:coeffAppli} 
	We call \emph{coefficient application} and denote by $\bXi$ the function defined by:
\begin{equation}\label{eq:CoeffAppliDef}
\bXi : \begin{cases}
C_{0}\big([0,1],\R^{d}\big) & \longrightarrow \uxi \Omega = \left( \R^{d} \right)^{\I}\\
\qquad \quad \bm{x} & \longmapsto \bDelta(\bm{x}) = \large\{\bDelta(\bm{x}) \large\}_{(n,k) \in \I}
\quad \mathrm{with} \quad 
{\lbrace \bDelta(\bm{x}) \rbrace}_{n,k} = \mathcal{P}\left( \bdelta_{n,k}, \bm{x} \right) \, .
\end{cases}
\end{equation}
\end{definition}

Should a function $x$ admits a uniformly convergent  decomposition in terms on the basis of elements $\bpsi_{n,k}$, the function $\bDelta$ gives its coefficients in such a representation. 
More precisely, we have:

\begin{theorem} \label{lemInv}
The function $\bDelta:\left(\ux \Omega' ,\mathcal{B}\left(  \ux \Omega' \right)  \right) \rightarrow \left(  \uxi \Omega',\mathcal{B}\left(\uxi \Omega' \right)\right) $ is a measurable linear bijection whose inverse is $\bPsi = \bDelta^{-1}$. 
\end{theorem}

The proof of this theorem is provided in Appendix \ref{append:ConstructCoeff}


\section{Representation of Gauss-Markov Processes}


\subsection{Inductive  Construction of Gauss-Markov Processes}

Up to this point, we have rigorously defined the dual spaces of sample paths $\ux \Omega'$ and coefficients $\uxi \Omega'$.
Through the use of the Schauder basis $\bpsi_{n,k}$ and its dual family of generalized functions $\bdelta_{n,k}$, we have defined inverse measurable bijections $\bPsi$ and $\bDelta$ transforming one space into the other.
In doing so, we have unraveled the fundamental role played by the underlying orthonormal basis $\bphi_{n,k}$.
We now turn to use this framework to formulate a path-wise construction of the Gauss-Markov processes in the exact same flavor as the Levy-Cesielski construction of the Wiener process.


\subsubsection{Finite-Dimensional Approximations}

Considering the infinite dimensional subspace $\ux \Omega'$ of $C_{0}\big([0,1],\R^{d}\big)$, let us introduce the equivalence relation $\sim_{N}$ as
\begin{equation}
\bx \sim_{N} \by \iff  \; \forall \; t \in D_{N} , \quad \bx(t) = \by(t) \, .
\nonumber
\end{equation}
We can use the functions $ \bPsi$ to carry the structure of $\sim_{N}$ on the infinite-dimensional space of coefficients $\uxi \Omega'$:
\begin{equation}
\bxi \sim_{N} \bm{\eta} \iff \; \bPsi(\xi) \sim_{N} \bPsi(\bm{\eta}) \iff \; \forall \; (n,k) \in \I_{N} \, , \quad \bxi_{n,k} = \bm{\eta}_{n,k} \, , \nonumber
\end{equation}
which clearly entails that $\bx \sim_{N} \by$  if and only if $\bDelta(\bx) \sim_{N} \bDelta(\by)$.
We denote the sets of equivalence classes of $\ux \Omega' / \sim_{N} = \ux \Omega_{N}$ and $\uxi \Omega' / \sim_{N} = \uxi \Omega_{N}$, which are  isomorphic  $ \ux \Omega_{N} = \left( \R^{d} \right)^{\I} = \uxi \Omega_{N}$.
For every $N>0$, we define the finite-dimensional operators $\bPsi_{N} = \ux \bm{i}_{N} \circ \bPsi \circ \uxi \bm{p}_{N}$ and $\bDelta_{N} = \uxi \bm{i}_{N} \circ \bDelta \circ \ux \bm{p}_{N}$, with the help of the canonical projections $\uxi \bm{p}_{N}:\uxi \Omega'  \to \uxi \Omega _{N}$, $\ux \bm{p}_{N}:\ux \Omega'  \to \ux \Omega _{N}$ and the inclusion map $\uxi \bm{i}_{N}:\uxi \Omega_{N}  \to \uxi \Omega' $, $\ux \bm{i}_{N}:\ux \Omega_{N}  \to \ux \Omega' $.

The results of the preceding sections straightforwardly extend on the equivalence classes, and in particular we see that the function $\bPsi_{N}:\uxi \Omega_{N} \rightarrow \ux \Omega_{N}$ and $\bDelta_{N}:\ux \Omega_{N} \rightarrow \uxi \Omega_{N}$ are linear finite-dimensional bijections satisfying $\bPsi_{N} = {\bDelta_{N}}^{-1}$. 
We write $\bm{e}=\lbrace \bm{e}_{p,q} \rbrace_{(p,q) \in \I}$ (resp. $ \bm{f} = \lbrace \bm{f}_{p,q} \rbrace_{(p,q) \in \I}$) the canonical basis of $\uxi \Omega _{N}$ (resp. $\ux \Omega _{N}$) when listed in the recursive dyadic order.
In these bases, the matrices $\bPsi_{N}$ and $\bDelta_{N}$ are lower block-triangular. Indeed, denoting $\bPsi_{N}$ in the natural basis $\bm{e}=\lbrace \bm{e}_{p,q} \rbrace_{(p,q) \in \I}$ and $ \bm{f} = \lbrace \bm{f}_{p,q} \rbrace_{(p,q) \in \I}$ by
\begin{equation}
\bPsi_{N}
=
\big[ \bpsi_{n,k}(m_{i,j}) \big]
=
\Big[ \bPsi^{i,j}_{n,k} \Big] \, ,
\nonumber
\end{equation}
where $\bPsi^{i,j}_{n,k}$ is a $d \times d$ matrix, the structure of the nested-support $S_{n,k}$ entails the  block-triangular structure (where only possibly non-zero coefficients are written):
\begin{displaymath}
 \ag \Psi_{N}
=
\left[
\begin{array}{c|c|cc|cccc|c}
  \bpsi_{0,0}^{0,0} &   &  &  &  &   &  &  &  \\
  \hline
  \bpsi_{0,0}^{1,0}&  \bpsi_{1,0}^{1,0} &  &   &  &   &  &  & \\
   \hline
  \bpsi_{0,0}^{2,0}&  \bpsi_{1,0}^{2,0} &  \bpsi_{2,0}^{2,0}  &   &   &   &  &  & \\
   \hline
   \bpsi_{0,0}^{2,1}&   \bpsi_{1,0}^{2,1} &  &  \bpsi_{2,1}^{2,1}  &   &   &  &  & \\
  \hline
   \bpsi_{0,0}^{3,0}&   \bpsi_{1,0}^{3,0} &  \bpsi_{2,0}^{3,0}  &  &  \bpsi_{3,0}^{3,0}  &   &  &  & \\
   \hline
   \bpsi_{0,0}^{3,1}&  \bpsi_{1,0}^{3,1}  &  \bpsi_{2,0}^{3,1}  &   &  &  \bpsi_{3,1}^{3,1}  &  &  & \\
  \hline
   \bpsi_{0,0}^{3,2}&  \bpsi_{1,0}^{3,2}  &  &  \bpsi_{2,1}^{3,2} &   &   & \bpsi_{3,2}^{3,2} & & \\
  \hline
   \bpsi_{0,0}^{3,3}&  \bpsi_{1,0}^{3,3}  &  &  \bpsi_{2,1}^{3,3} &   &   &  &  \bpsi_{3,3}^{3,3} &\\
   \hline
  \vdots &   &  &   &   &   &  &  &  \ddots
\end{array}
\right] \, .
\end{displaymath}
Similarly, for the matrix representation of $\bDelta_{N}$ in the natural basis $e_{n,k}$ and $f_{i,j}$
\begin{equation}
\bDelta_{N} 
=
\Big[ \bDelta^{n,k}_{i,j} \Big]
\nonumber
\end{equation}
proves to have the following triangular form:
\begin{displaymath}
\bDelta_{N} 
=
\left[
\begin{array}{c|c|c|c|c}
 { \bgg^{-1}(t_{0,0})}^{T} \bM_{0,0}   &   &  &   &   \\
\hline
 -{\bgg^{-1}(t_{0,0})}^{T} \bR_{1,0} &   {\bgg^{-1}(t_{1,0})}^{T} \bM_{1,0}  &  &   &    \\
\hline
&  -{\bgg^{-1}(t_{0,0})}^{T} \bR_{2,0} &  {\bgg^{-1}(t_{1,0})}^{T} \bM_{2,0} &   &   \\
 -{\bgg^{-1}(t_{0,0})}^{T} \bR_{2,1} &   -{\bgg^{-1}(t_{1,0})}^{T} \bL_{2,1}  &  &   {\bgg^{-1}(t_{2,1})}^{T} \bM_{2,1}  &   \\
\hline
  \vdots &   &  &   &     \ddots
\end{array} 
\right] \, .
\end{displaymath}
The duality property \ref{dualProp} simply reads for all $0 \leq n <N$ and $0 \leq k < 2^{n-1}$, $0 \leq p <N$ and $0 \leq p < 2^{q-1}$ 
\begin{eqnarray}
\mathcal{P}(\bdelta_{p,q},\bpsi_{n,k}) 
= 
\sum_{(n,k) \in \I_{N} } \bDelta^{p,q}_{i,j} \cdot \bPsi^{i,j}_{n,k}   = \delta^{p,q}_{n,k} \bm{I}_{d}\, .
\nonumber
\end{eqnarray}
that is, $\bDelta_{N} \cdot  \bPsi_{N}= Id_{\uxi \Omega_{N}}$.
But because, we are now in a finite-dimensional setting, we also have $\bPsi_{N} \cdot \bDelta_{N} = Id_{\ux \Omega_{N}}$:
\begin{equation}
\delta^{i,j}_{k,l} \bm{I}_{d} = \sum_{(p,q) \in \I_{N} } \bPsi^{i,j}_{p,q} \cdot \bDelta^{p,q}_{k,l} \, .
\nonumber
\end{equation}
Realizing that $\delta^{i,j}_{k,l} \bm{I}_{d}$ represents the class of functions $\bm{x}$ in $\ux \Omega'$ whose value are zero on every dyadic points of $D_{N}$ except for $\bm{x}(l2^{k}) = \bm{I}_{d}$, $\left\{ \bDelta^{p,q}_{k,l} \right\}_{(p,q) \in \I_{N}}$ clearly appear as the coefficients of the decomposition of such functions in the bases $\bpsi_{p,q}$ for $(p,q)$ in $\I_N$.

Denoting $\bXi = \lbrace \bXi_{n,k} \rbrace_{(n,k) \in I}$, a set  of independent Gaussian variables of law $\mathcal{N}(\bm{0},\bm{I}_d)$ on $\left( \Omega, \mathcal{F} , \mathrm{\bold{P}}\right)$, and for all $N>0$, we form the finite dimensional Gauss-Markov vector  $\left[ \bX^{N}_{i,j} \right]_{(i,j) \in \I_{N}}$ as
\begin{eqnarray*}
 \bX^{N}_{i,j} = \sum_{(n,k) \in \I_{N} } \bpsi_{n,k}(m_{i,j}) \cdot \bXi_{n,k} \, ,
\end{eqnarray*}
which,  from Corollary \ref{cor:basis}, has the same law as $\left[ \bX_t \right]_{t \in D_{N}}$, the finite-dimensional random vector obtained from sampling $\bX$ on $D_{N}$ (modulo a permutation on the indices). 
We then prove the following lemma that sheds light on the meaning of the construction:

\begin{lemma}
The Cholesky decomposition of the finite-dimensional covariance block matrix $ \bm{ \Sigma}_{N}$  is given  by $ \bm{ \Sigma}_{N} =  \bPsi_{N}\cdot { \bPsi_{N}}^{T}$.
\end{lemma}

\begin{proof}
For every $0 \leq t , s\leq 1$, we compute the covariance of the finite-dimensional process $\bX^{N}$ as
\begin{eqnarray} \label{eq:CovRepeat}
\bm{C}_{N}(t,s) = \Exp{ \bX^{N}_{t} \cdot (\bX^{N}_{s})^{T}} =   \sum_{n = 0}^{N} \sum_{ \hspace{5pt} 0 \leq k < 2^{n\!-\!1} }   \bpsi_{n,k}(t) \cdot  \left( \bpsi_{n,k}(s) \right)^{T} \, , \nonumber
\end{eqnarray}
From there,  we write the finite-dimensional covariance block matrix $\bm{ \Sigma}_{N}$ in the recursively ordered basis $\bm{f}_{i,j}$ for $0 \leq i \leq N$, $0 \leq j < 2^{i-1}$, as
\begin{eqnarray}
{\big[  \bm{\Sigma}_{N} \big]}^{i,j}_{k,l} = \bm{C}_{N}(m_{i,j},m_{k,l}) = \sum_{n = 0}^{N} \sum_{ \hspace{5pt} 0 \leq k < 2^{n\!-\!1} }   \bPsi_{n,k}^{i,j} \cdot  \bPsi_{n,k}^{k,l} \, .
\nonumber
\end{eqnarray}
We already established that the matrix $\bPsi_{N}$ was triangular with positive diagonal coefficient, which entails that the preceding equality provides us with the Cholesky decomposition of $\bm{\Sigma}$.
\end{proof}

In the finite-dimensional case, the inverse covariance or potential matrix is a well-defined quantity and we straightforwardly have the following corollary:

\begin{corollary}
The Cholesky decomposition of the finite-dimensional inverse covariance matrix $\bm{ \Sigma}^{-1}_N$  is given  by $\bm{ \Sigma}^{-1}_N = {\bDelta_{N}}^{T} \cdot \bDelta_{N}$.
\end{corollary}

\begin{proof}
The result stems for the equalities: $\bm{ \Sigma}^{-1}_N = {\left( \bPsi_{N} \cdot { \bPsi_{N}}^{T}   \right)}^{-1}= {\left(  \bPsi^{-1}_{N} \right)}^{T} \cdot \bPsi^{-1}_{N} = {\bDelta_{N}}^{T} \cdot \bDelta_{N}$.
\end{proof}


\subsubsection{The L\'{e}vy-Cesielski Expansion}\label{ssec:LevyConstruct}

We now show that asymptotically, the basis $\bpsi_{n,k}$ allows us to faithfully build the Gauss-Markov process from which we have derived its expression.
In this perspective we consider $\bXi = \lbrace \bXi_{n,k} \rbrace_{(n,k) \in \I}$, a set  of independent Gaussian variables of law $\mathcal{N}(\bm{0},\bm{I}_d)$ on $\left( \Omega, \mathcal{F} , \mathrm{\bold{P}}\right)$, and for all $N>0$, we form the finite dimensional continuous Gaussian process $\bZ^{N}$, defined for $0 \leq t \leq1$ by
\begin{eqnarray*}
\bX^{N}_{t} = \sum_{(n,k) \in \I_{N} } \bpsi_{n,k}(t) \cdot \bXi_{n,k} \, ,
\end{eqnarray*}
which is, from the result of Theorem \ref{theo:basis}, has the same law $\bZ^N_t=\Exp{X_t\vert \mathcal{F}_N}$. We prove the following lemma:

\begin{lemma} \label{propX} 
The sequence of processes $\bX^N$ almost surely converges towards a continuous Gaussian process denoted $\bX^{\infty}$ .
\end{lemma}

\begin{proof}
For all fixed $N>0$ and for any $\omega$ in $\Omega$, we know that $t \mapsto \bX^{N}_{t}(\omega)$ is continuous.
Moreover, we have establish, that for every $\bxi$ in $\uxi \Omega'$, $\bX^{N}(\bxi)$ converges uniformly in $t$ toward a continuous limit denoted $\bX^{N}(\bxi)$.
Therefore, in order to prove that $\lim_{N \to \infty} \bX^{N}$ defines almost surely a process $\bX$ with continuous paths, it is sufficient to show that $\Probxi (\uxi \Omega') = 1$, where $\Probxi =\ProbXi$ is the $\Xi$-induced measure on $\uxi \Omega$, which stems from a classical Borel-Cantelli argument.
For $\xi$ a random variable of normal law $\mathcal{N}(0,1)$, and $a>0$, we have
\begin{equation}
\mathrm{\bold{P}}(\vert \xi \vert>a) = \sqrt{\frac{2}{\pi}} \int_{a}^{\infty} e^{-u^{2}/2} \, du \leq  \sqrt{\frac{2}{\pi}} \int_{a}^{\infty} \frac{u}{a} e^{-u^{2}/2} \, du = \sqrt{\frac{2}{\pi}} \frac{e^{-a^{2}/2}}{a} \, .
\nonumber
\end{equation}
Then, for any $\delta>0$
\begin{equation}
\Probxi \big(\max_{0 \leq k < 2^{n-1}} \vert \bxi_{n,k} \vert_{\infty} > 2^{\frac{n\delta}{2}} \big) \leq d2^{n} \mathrm{\bold{P}}(\vert \xi \vert> d2^{\frac{n\delta}{2}}) = \sqrt{\frac{2}{\pi}} 2^{(1-\delta/2)n} \exp{\left( - 2^{n\delta -1} \right)} \, .
\nonumber
\end{equation}
Since the series 
\begin{equation}
\sum_{n=0}^{\infty}\sqrt{\frac{2}{\pi}} 2^{(1-\delta/2)n} \exp{\left( - 2^{n\delta -1} \right)}
\end{equation}
is convergent, the Borel-Cantelli argument implies that $\Probxi ({}_{ \scriptscriptstyle \xi \displaystyle}\Omega') = 1$.
Eventually, the continuous almost-sure limit process $\bX^{\infty}_t$ is Gaussian as a countable sum of Gaussian processes.
 \end{proof}
 
Now that these preliminary remarks have been made, we can evaluate, for any $t$ and $s$ in $[0,1]$,  the covariance of $\bX$ as the limit of the covariance of $\bX^{N}$.
 
\begin{lemma} \label{lem:covariance}
For any $0 \leq t,s  \leq 1$, the covariance of $\bX^{\infty} = \lbrace \bX^{\infty}_{t} = \bPsi_{t} \circ \bXi ; 0 \! \leq \! t \! \leq \! 1 \rbrace$ is
\begin{eqnarray}
\bm{C}(t,s) = \Exp{ \bX^{\infty}_{t} \cdot \left( \bX^{\infty}_{s} \right)^{T}} = \bgg(t) \,  \bh(t \wedge s) \, \bgg(s)^{T} \, .
\end{eqnarray}
\end{lemma}

\begin{proof}
As $\bXi_{n,k}$ are independent Gaussian random variables of normal law $\mathcal{N}(\bm{0},\bm{I}_d)$, we see that the covariance of $\bX^{N}$ is given by
\begin{eqnarray*}
\bm{C}^{N}(t,s) = \Exp{ \bX_{t}^{N} \cdot \left(\bX_{s}^{N}\right)^{T} } = 
\sum_{ (n,k) \in \I_{N} } \; \bpsi_{n,k}(t) \cdot \left(\bpsi_{n,k}(s)\right)^{T} \, .
\end{eqnarray*}
To compute the limit of the right-hand side, we recall that the element of the basis $\bpsi_{n,k}$ and the functions  $\bphi_{n,k}$ are linked by the following relation
\begin{eqnarray*}\label{eq:IntPsi}
\bpsi_{n,k}(t) =  \mathcal{K}[\bphi_{n,k}] = \bm{g}(t)  \int_{U} \mathbbm{1}_{[0,t]}(s) \bm{f}(s) \, \bphi_{n,k}(s) \, ds 
\end{eqnarray*}
from which we deduce
\begin{eqnarray*}
\bm{C}^{N}(t,s) &=& \bgg(t) \left( \sum_{ (n,k) \in \I_{N} }  \left( \int_{U} \mathbbm{1}_{[0,t]}(u) \bm{f}(u) \, \bphi_{n,k}(u) \, du \right) \right. \\
&& \qquad  \qquad \qquad \qquad \left.  \left(\int_{U} \mathbbm{1}_{[0,s]}(v) \bm{f}(v) \, \bphi_{n,k}(v) \, dv \right)^{T} \right)  \bgg(s)^{T} \, .
\end{eqnarray*}
Defining the auxiliary $\R^{d \times d}$-valued function
\begin{equation*}
\bm{\kappa}_{n,k}(t) = \int_{U} \mathbbm{1}_{[0,t]}(u) \bm{f}(u) \, \bphi_{n,k}(u) \, du \, ,
\end{equation*}
we observe that $(i,j)$-coefficient function reads
\begin{eqnarray*}
(\bm{\kappa}_{n,k})_{i,j}(t) 
&=& 
\int_{U} \mathbbm{1}_{[0,t]}(u) \left( l_{i}\big(\bff(u)\big)^{T} \cdot c_{j}\big(\bphi_{n,k}(u)\big) \right) \, du \, , \\
&=&
\int_{U}  \left( \mathbbm{1}_{[0,t]}(u) \,c_{i}\big( \bff^{T}(u)\big)^{T} \cdot c_{j}\big(\bphi_{n,k}(u)\big) \right) \, du \, , 
\end{eqnarray*}
where $\mathbbm{1}_{[0,t]}$ is the real function that is one if $0 \leq u \leq t$ and zero otherwise.
As we can write 
\begin{equation*}
\mathbbm{1}_{[0,t]}(u) \,c_{i}\big( \bff^{T}(u)\big) = \bff^{T}(u) \cdot 
\left[
\begin{array}{c}
 0 \\
  \vdots\\
 \mathbbm{1}_{[0,t]}(u) \\
   \vdots\\
  0 \\    
\end{array}
\right] \leftarrow i \, ,
\end{equation*}
we see that the function $\bm{f}_{i,t} = \mathbbm{1}_{[0,t]} \,c_{i}( \bff^{T})$ belongs to $L^{2}_{\bm{f}}$, so that we can write $(\bm{\kappa}_{n,k})_{i,j}(t) $ as a scalar product in the Hilbert space $L^{2}_{\bm{f}}$:
\begin{eqnarray*}
(\bm{\kappa}_{n,k})_{i,j}(t) 
=
\int_{U}  \bm{f}_{i,t}^{T}(u) \cdot c_{j}\big(\bphi_{n,k}(u)\big) \, du 
=
\Big( \bm{f}_{i,t} , c_{j}\big(\bphi_{n,k}\big) \Big)\, .
\end{eqnarray*}
We then specify the $(i,j)$-coefficient of $\bgg^{-1}(t) \, \bm{C}^{N}(t,s) \, \big( \bgg^{-1}(s) \big)^{T}$ writing
\begin{eqnarray*}
\sum_{ (n,k) \in \I_{N} } \Big(\bm{\kappa}(t) \cdot \bm{\kappa}(s)^{T}\Big)_{i,j}
=
\sum_{ (n,k) \in \I_{N} }
\sum_{p=0}^{d-1} 
\left(   \bm{f}_{i,t} , c_{j}\big(\bphi_{n,k} \big)  \right)
\left(   \bm{f}_{j,s} , c_{j}\big(\bphi_{n,k} \big)  \right)
\end{eqnarray*}
and, remembering that the family of functions $c_j \big(\bphi_{n,k}\big)$ forms a complete orthonormal system of $L^{2}_{\bm{f}}$, we can  use the Parseval identity, which reads
\begin{eqnarray*}
\sum_{ (n,k) \in \I } \Big(\bm{\kappa}(t) \cdot \bm{\kappa}(s)^{T}\Big)_{i,j}
&=&
\Big( \bm{f}_{i,t}  , \bm{f}_{j,s} \Big) \\
&=&
\int_{U} \mathbbm{1}_{[0,t]}(u) \, c_{i}\big( \bff^{T}(u)\big)^{T}  \cdot \mathbbm{1}_{[0,s]}(u) \, c_{j}\big( \bff^{T}(u)\big)  \, du \, , \\
&=&
\int_{0}^{t \wedge s} \left( \bm{f} \cdot \bm{f}^{T} \right)_{i,j}(u) \, du \, .
\end{eqnarray*}
Thanks to this relation, we can conclude the evaluation of the covariance since   
\begin{eqnarray*} 
\lim_{N \to \infty} \bm{C}^{N}(t,s)
=
\bgg(t)  \left( \int_{0}^{t \wedge s} \left( \bm{f} \cdot \bm{f}^{T} \right)(u) \, du \right)  \bgg(s)^{T} 
=
\bgg(t)  \bh(t \wedge s) \,  \bgg(s)^{T} 
\, , 
\end{eqnarray*}
\end{proof}

We stress the fact that the relation
\begin{eqnarray*}
\bm{C}(t,s) = 
\sum_{ (n,k) \in \I } \; \bpsi_{n,k}(t) \cdot \left(\bpsi_{n,k}(s)\right)^{T} \,
= \bPsi(t) \circ \bPsi^{T}(s) .
\end{eqnarray*}
provides us with a continuous version of  the Cholesky decomposition of the covariance kernel $\bm{C}$.
Indeed, if we chose $\bm{\sigma}_{n,k}$ as the Cholesky square root of $\bm{\Sigma}_{n,k}$, we remark that the operators  $\bPsi$ are triangular in the following sense:
consider the chain of nested vectorial spaces $\left\{ F_{n,k} \right\}_{(n,k) \in \I}$
\begin{eqnarray*}
F_{0,0} \subset F_{1,0} \subset \left( F_{2,0}\subset F_{2,1} \right) \ldots \subset \left(  F_{n,0} \subset \ldots \subset F_{n,2^{n}-1} \right) \ldots \subset \uxi \Omega'
\end{eqnarray*}
with $F_{n, k} = \mathrm{span} \left\{ \bm{f}_{i,j} \, \vert \, 0 \leq i  \leq n , 0 \leq j \leq k \right\}$, then  for every $(n,k)$ in $\I$, the operator $\bPsi$ transforms  the chain $\left\{ F_{n,k} \right\}_{(n,k) \in \I}$ into the chain 
\begin{eqnarray*}
\mathcal{E}_{0,0} \subset \mathcal{E}_{1,0} \subset \left( \mathcal{E}_{2,0}\subset \mathcal{E}_{2,1} \right) \ldots \subset \left(  \mathcal{E}_{n,0} \subset \ldots \subset \mathcal{E}_{n,2^{n}-1} \right) \ldots \subset \ux \Omega'
\end{eqnarray*}
with $\mathcal{E}_{n, k} = \mathrm{span} \left\{ \bm{\bPsi}_{i,j} \, \vert \, 0 \leq i  \leq n , 0 \leq j \leq k \right\}$.

The fact that this covariance is equal to the covariance of the process $\bm{X}$ solution of equation \eqref{eq:eqStoch} implies that we have the following fundamental result:

\begin{theorem}\label{theo:PrincipTheo}
	The process $\bm{X}^{\infty}$ is equal in law to the initial Gauss-Markov process $\bX$ used to construct the basis of functions. 
\end{theorem}


\begin{remark}
Our  multi-resolution representation of Gauss-Markov processes appears to be the direct consequence of the fact that, because of the Markov property, the Cholesky decomposition of the finite-dimensional covariance admit a simple inductive continuous limit.
More generally, triangularization of kernel operators have been studied in depth~\cite{Kailath:1978,Brodski:1968,Brodski:1971,Gohberg:1970} and it would be interesting to investigate if these results make possible a similar multi-resolution approach for non-Markov Gaussian processes.
In this regard, we naturally expect to lose the compactness of the supports of a putative basis.
\end{remark}


\subsection{Optimality Criterion of the Decomposition}\label{sec:Optimality}

In the following, we draw from the theory interpolating spline to further characterize the nature of our proposed basis for the construction of Gauss-Markov processes.
We show, following \cite{Dolph,Kimel3}, that the finite-dimensional sample paths of our construction induce a nested sequence $\mathcal{E}_{N}$ of reproducing Hilbert kernel space (RKHS). 
In turn, the finite-dimensional process $\bX^N$ naturally appears as the orthogonal projection of the infinite-dimensional process $\bX$ onto $\mathcal{E}_{N}$.
We then show that such a RKHS structure allows us to define a unicity criterion for the finite-dimensional sample path as the only functions of  $\mathcal{E}$ that minimize a functional, called Dirichlet energy, under constraint of interpolation on $D_N$ (equivalent to conditioning on the times $D_{N}$), thus extending well-known results to multidimensional kernel~\cite{Smola:2001}. 
In this respect, we point out that the close relation between Markov processes and Dirichlet forms is the subject of a vast literature, largely beyond the scope of the present paper (see e.g. \cite{Fukushima:1994}).


\subsubsection{Sample Paths Space as a Reproducing Hilbert Kernel Space}

In order to define the finite-dimensional sample paths as a nested sequence of RKHS, let us first define the infinite-dimensional operator
\begin{equation*}
\bPhi: \begin{cases} l^{2}(\uxi \Omega) &\mapsto L^{2}_{\bm{f}}\\
\bxi &\mapsto \bPhi[\bxi] = \Big \lbrace t \mapsto \sum_{(n,k) \in \I} \bphi_{n,k}(t) \cdot \bxi_{n,k} \Big \rbrace 
\end{cases} \, .
\end{equation*}
Since we know that the column functions of $\bphi_{n,k}$ form a complete orthonormal system of $L^{2}_{\bm{f}}$, the operator $\bPhi$ is an isometry and its inverse satisfies $\bPhi^{-1}=\bPhi^{T}$, which reads for all $\bm{v}$ in $L^{2}_{\bm{f}}$
\begin{eqnarray*}
{\left[\bPhi^{-1}[\bm{v}]\right]}_{n,k} = \int_{U} \bphi_{n,k}^{T}(t) \cdot \bm{v}(t) \, dt = \mathcal{P}(\bphi_{n,k},\bm{v}) \, .
\end{eqnarray*}
Equipped with this infinite-dimensional isometry, we then consider the linear operator $\mathcal{L} = \bPhi \circ \bDelta$ siutably defined on the set
\begin{eqnarray*}
\mathcal{E} = \big \lbrace  \bm{u} \in C_{0}\big(U,\R^{d}\big)\, \vert \, \mathcal{L}[\bm{u}] \in L^{2}_{\bm{f}} \big \rbrace 
=
 \big \lbrace  \bm{u} \in C_{0}\big(U,\R^{d}\big)\, \vert \,   \bDelta[\bm{u}] \in l^{2}(\uxi \Omega) \big \rbrace 
 \, .
\end{eqnarray*}
with ${\Vert \bxi \Vert_{2}}^{2} = \sum_{n,k \in \I} {\vert \bxi_{n,k} \vert_{2}}^{2} $, the $l^{2}$ norm of $\uxi \Omega$.
The set $\mathcal{E}$ form an infinite-dimensional vectorial space that is naturally equipped with the inner product
\begin{equation*}
\forall \: (\bm{u}, \bm{v}) \, \in \mathcal{E}^{2}\, , \quad \langle \bm{u}, \bm{v} \rangle = \int_{U} \mathcal{L}[\bm{u}](t)^{T} \cdot \mathcal{L}[\bm{v}](t) \, dt = \big( \mathcal{L}[\bm{u}] , \mathcal{L}[\bm{v}]\big) \, ,
\end{equation*}
Moreover  since $\bm{u}(0) = \bm{v}(0) = \bm{0}$, such an inner product is definite positive and consequently, $\mathcal{E}$ forms an Hilbert space.

\begin{remark}
Two straightforward remarks are worth making.
	 First, the space $\mathcal{E}$ is strictly included in the infinite-dimensional sample paths space $\ux \Omega'$. 
	Second notice that, in the favorable case $m=d$, if $\bm{f}$ is everywhere invertible with continuously differentiable inverse, we have $\mathcal{L} = \mathcal{D} = \mathcal{K}^{-1}$.
	More relevantly, the operator $\mathcal{L}$   can actually considered a first-order differential operator  from $\mathcal{E}$ to $L^{2}_{\bm{f}}$ as a general left-inverse of the integral operator $\mathcal{K}$.
	Indeed, realizing that on $L^{2}_{\bm{f}}$, $\mathcal{K}$ can be expressed as $\mathcal{K} = \bPsi \circ \bPhi^{-1}$, we clearly have
	\begin{equation*}
	\mathcal{L} \circ \mathcal{K} 
	=
	 \bPhi \circ \bDelta \circ \bPsi \circ \bPhi^{-1}
	 =
	Id_{L^{2}_{\bm{f}}} \, .
	\end{equation*}
\end{remark}

We know motivate the introduction of the Hilbert space $\mathcal{E}$ by the following claim:

\begin{proposition}
The Hilbert space $\left( \mathcal{E}, \langle , \rangle \right)$ is a reproducing kernel Hilbert space (RKHS) with $\R^{d \times d}$-valued reproducing kernel $\bm{C}$, the covariance function of the process $\bX$.
\end{proposition}

\begin{proof}
Consider the problem of finding all elements $\bm{u}$ of $\mathcal{E}$ solution of the equation $\mathcal{L}[\bm{u}] = \bm{v}$ for $ \bm{v}$   in $L^{2}_{\bm{f}}$.
The operator $\mathcal{K}$ provides us with a continuous  $\R^{d \times m}$-valued kernel function $\bm{k}$ 
\begin{eqnarray*}
\forall \: (t,s) \, \in \, U^{2}\, , \quad \bm{k}(t,s) = \mathbbm{1}_{[0,t]}(s) \; \bm{g}(t) \cdot \bm{f}(s) \, , 
\end{eqnarray*}
which is clearly the Green function for our differential equation.
This entails that the following equalitiy holds for every $\bm{u}$ in $\mathcal{E}$
\begin{eqnarray*}
\bm{u}(t)
=
\int_{U}  \bm{k}(t,s) \, \mathcal{L}[\bm{u}](s) ds \, . 
\end{eqnarray*}
Moreover, we can decompose the kernel $\bm{k}$ in the $L^{2}_{\bm{f}}$ sense as
\begin{equation*}
\bm{k}(t,s)= \sum_{(n,k) \in \I} \bpsi_{n,k}(t) \cdot \bphi^{T}_{n,k}(s)
\end{equation*}
since we have
\begin{eqnarray*}
 \bm{k}(t,s) &=& \mathcal{K} \left[ \delta_{s} \,  {Id}_{L^{2}_{\bm{f}}} \right](t) \\
 &=& \mathcal{K} \left[ \sum_{(n,k) \in \I} \bphi_{n,k} \cdot \bphi^{T}_{n,k}(s)\right] \\
 & =&  \sum_{(n,k) \in \I}  \mathcal{K} \left[ \bphi_{n,k}\right](t) \cdot \bphi^{T}_{n,k}(s) \, ,
\end{eqnarray*}
withs $\delta_{s}= \delta(\cdotp-s)$.
Then, we clearly have 
\begin{eqnarray*}
\bm{C}(t,s)
=
\int_{U} \bm{k}(t,u) \cdot \bm{k}(s,u)^{T} \, du
=
\sum_{(n,k) \in \I} \bpsi_{n,k}(t) \cdot \bpsi^{T}_{n,k}(s)
\end{eqnarray*}
where we recognize the covariance function of $\bX$, which implies 
\begin{eqnarray*}
\bm{k}(t,s) = 
 \sum_{(n,k) \in \I} \bpsi_{n,k}(t) \cdot \mathcal{L}\left[ \bphi_{n,k}\right]^{T}(s)
 =
 \mathcal{L}\left[ \bm{C}(t, \cdotp)\right] \, .
\end{eqnarray*}
Eventually, for all $\bm{u}$ in $L^{2}_{\bm{f}}$, we have:
\begin{eqnarray*}
\bm{u}(t)
=
\int_{U}   \mathcal{L}\left[ \bm{C}(t, \cdotp)\right](s) \cdot \mathcal{L}[\bm{u}](s) ds 
=
\mathcal{P} \langle \bm{C}(t, \cdotp) , \bm{u} \rangle
\, ,
\end{eqnarray*}
where we have introduced the $\mathcal{P}$-operator associated with the inner product $\langle , \rangle$:
for all $\R^{d \times d}$-valued functions $\bm{A}$ and $\bm{B}$ defined on $U$ such that the columns $c_{i}(\bm{A})$ and $c_{i}(\bm{B})$, $0 \leq i < d$, are in $\mathcal{E}$, we define the matrix $ \mathcal{P}\langle\bm{A},\bm{B}\rangle$ in $\R^{d \times d}$ by
\begin{equation*}
\forall \; 0 \leq i,j < d \, , \quad \mathcal{P}\langle\bm{A},\bm{B}\rangle_{i,j} =  \big \langle c_{i}(\bm{A}), c_{j}(\bm{B})\big \rangle \, .
\end{equation*}
By the Moore-Aronszajn theorem \cite{aronszajn:1950}, we deduce that there is a unique reproducing kernel Hilbert space associated with a given covariance kernel.
Thus, $\mathcal{E}$ is the reproducing subspace of $C_{0}\big( U, \R^{d} \big)$ corresponding to the kernel $\bm{C}$, with respect to the inner product $\langle , \rangle$.
\end{proof}

\begin{remark}
From a more abstract point of view, it is well-know that the covariance operator of a Gaussian measure defines an associated Hilbert structure~\cite{Kuo:1975}.
\end{remark}

In the sequel, we will use the space $\mathcal{E}$ as the ambient Hilbert space to define the finite-dimensional sample-paths spaces as  a nested sequence of RKHS.
More precisely, let us write for $\mathcal{E}_{N}$ the finite-dimensional subspace of $\mathcal{E}$
\begin{eqnarray*}
\mathcal{E}_{N}
=
\big \lbrace  \bm{u} \in C_{0}\big(U,\R^{d}\big)\, \vert \, \mathcal{L}[\bm{u}] \in L^{2}_{\bm{f},N} \big \rbrace 
 \, ,
\end{eqnarray*}
with the space $L^{2}_{\bm{f},N}$ being defined as
\begin{eqnarray*}
L^{2}_{\bm{f},N}
=
\mathrm{span} \left[ \left\{ c_{i}(\bphi_{n,k}) \right\}_{n,k \in \I_{N}, 0 \leq i < d} \right] \, .
\end{eqnarray*}
We refer to such spaces as finite-dimensional approximation spaces, since we remark that
\begin{eqnarray*}
\mathcal{E}_{N}
=
\mathrm{span} \left[ \left\{ c_{i}(\bpsi_{n,k}) \right\}_{n,k \in \I_{N}, 0 \leq i < d} \right] 
=
\bPsi_{N}\left[ \uxi \Omega_{N} \right] \, ,
\end{eqnarray*}
which means the space $\mathcal{E}_{N}$ is made of the sample space of the finite dimensional process $\bX_{N}$.
The previous definition makes obvious the nested structure $\mathcal{E}_{0} \subset \mathcal{E}_{1} \subset \dots \subset \mathcal{E}$, and it is easy to characterize each space $\mathcal{E}_{N}$ as reproducing Hilbert kernel space:

\begin{proposition}
The Hilbert spaces $\left( \mathcal{E}_{N}, \langle , \rangle \right)$ are reproducing kernel Hilbert space (RKHS) with $\R^{d \times d}$-valued reproducing kernel $\bm{C}_{N}$, the covariance function of the process $\bX_{N}$.
\end{proposition}

\begin{proof}
The proof this proposition follows the exact same argument as in the case of $\mathcal{E}$, but with the introduction of finite-dimensional kernels $\bm{k}_{N}$ 
\begin{eqnarray*}
\forall \: (t,s) \, \in \, U^{2}\, , \quad \bm{k}_{N}(t,s) =  \sum_{(n,k) \in \I_{N}}  \bpsi_{n,k}(t) \cdot \bphi^{T}_{n,k}(s) \, ,
\end{eqnarray*}
and the corresponding covariance function
\begin{equation*}
\forall \: (t,s) \, \in \, [0,1]^{2}\, , \quad \bm{C}_{N}(t,s) =  \sum_{(n,k) \in \I_{N}}  \bpsi_{n,k}(t) \cdot \bpsi^{T}_{n,k}(s) \, .
\end{equation*}
\end{proof}


\subsubsection{Finite-Dimensional Processes as Orthogonal Projections}

The framework set in the previous section offers a new interpretation of our construction.
Indeed, for all $N>0$, the column of $\lbrace \bpsi_{n,k} \rbrace_{(n,k) \in \I_{N}}$ form an orthonormal basis of the space $\mathcal{E}_{N}$:
\begin{eqnarray*}
\mathcal{P} \langle \bpsi_{n,k} , \bpsi_{p,q} \rangle = \mathcal{P}( \mathcal{L}[\bpsi_{n,k}] ,  \mathcal{L}[\bpsi_{p,q}]) =  \mathcal{P}( \bphi_{n,k} ,  \bphi_{p,q}) = \delta^{n,k}_{p,q} \, .
\end{eqnarray*}
This leads to define the finite-dimensional approximation $\bm{x}_{N}$ of an sample path $\bm{x}$ of $\mathcal{E}$ as the orthogonal projection of $\bm{x}$ on $\mathcal{E}_{N}$ with respect to the inner product $\langle, \rangle$.
At this point, it is worth remembering that the space $\mathcal{E}$ is strictly contained in $\ux \Omega'$ and does not coincide with  $\ux \Omega'$: actually one can easily show that $\mathrm{\bold{P}}(\mathcal{E}) = \bm{0}$.
We devote the rest of this section to define the finite-dimensional processes $\bZ^{N} = \ExpN{X}$ resulting from the conditioning on $D_{N}$, as path-wise orthogonal projection of the original process $\bX$ on the sample space $\mathcal{E}_{N}$.

\begin{proposition}\label{prop:orthProj}
For any $N>0$, the conditioned processes $\ExpN{\bX}$ can be written as the orthogonal projection of $\bX$ on $\mathcal{E}_{N}$ with respect to $\langle, \rangle$
\begin{equation*}
\ExpN{\bX} = \sum_{(n,k) \in \I_{N}} \bpsi_{n,k} \cdot \mathcal{P} \langle \bpsi_{n,k} ,  \bX \rangle \, .
\end{equation*}
\end{proposition}

The only hurdle to prove Proposition \ref{prop:orthProj} is purely technical in the sense that the process $\bX$ lives in a larger space than $\mathcal{E}$: 
we need to find a way to extend the definition of $\langle , \rangle$ so that the expression bears a meaning.
Before answering this point quite straighforwardly, we need to establish the following lemma:

\begin{lemma}\label{lem:finiteStoch}
Writing the Gauss-Markov process $\bX_{t} = \int_{0}^{1} \bm{k}(t,s) \, d\bm{W}_{s}$, for all $N>0$, the conditioned process $\bZ^{N} = \ExpN{\bX}$ is expressed as the stochastic integral
\begin{eqnarray*}
\bZ^{N}= \int_{0}^{1} \bm{k}_{N}(t,s) \, d\bm{W}_{s} \quad \mathrm{with} \quad   \bm{k}_{N}(t,s) = \sum_{(n,k) \in \I_{N}} \bpsi_{n,k}(t) \cdot \bphi^{T}_{n,k}(s)\, .
\end{eqnarray*}
\end{lemma}

\begin{proof}
In the previous section, we have noticed  that the kernel $\bm{k}_{N}$ converges toward the kernel  $\bm{k}$ (Green function) in the $L^{2}_{\bm{f}}$ sense:
\begin{eqnarray*}
\bm{k}(t,s) &=&  \sum_{(n,k) \in \I} \bpsi_{n,k}(t) \cdot \bphi^{T}_{n,k}(s)\\ 
&=&  \lim_{N \to \infty} \sum_{(n,k) \in \I_{N}} \bpsi_{n,k}(t) \cdot \bphi^{T}_{n,k}(s) \bm{k}_{N} \\
& =& \lim_{N \to \infty} \bm{k}_{N}(t,s) \, .
\end{eqnarray*}
This implies that the process $\bX$ as the stochastic integral, can also be written as 
\begin{equation*}
\bX_{t} =  \bgg(t) \int_{U}  \mathbbm{1}_{[0,t]}(s) \bff(s) \, d\bm{W}_{s} = \int_{0}^{1} \bm{k}(t,s) \, d\bm{W}_{s}  = \lim_{N \to \infty} \int_{0}^{1} \bm{k}_{N}(t,s) \, d\bm{W}_{s}  \, .
\end{equation*}
Specifying the decomposition of $\bm{k}_{N}$, we can then naturally express $\bX$ as the convergent sum 
\begin{eqnarray*}
\bX_{t}   = \sum_{(n,k) \in \I} \bpsi_{n,k} \cdot \bXi_{n,k} \quad \mathrm{with} \quad \bXi_{n,k} = \int_{0}^{1} \bphi_{n,k}^{T}(s) \, d\bm{W}_{s} \, ,
\end{eqnarray*}
where the orthonormality property of the $\bphi_{n,k}$ with respect to $( , )$, makes the vectors $\bXi_{n,k}$ appears as independent  $d$-dimensional Gaussian variables of law $\mathcal{N}(0,\bm{I}_{d})$.
It is then easy to see that by definition of the elements $\bpsi_{n,k}$, for almost every $\omega$ in $\Omega$,  we then have 
\begin{eqnarray*}
\forall \: N>0 \, , \: 0 \leq t \leq 1 \, , \quad \bZ^{N}(\omega) = \ExpN{\bX}(\omega) = \sum_{(n,k) \in \I_{N}} \bpsi_{n,k} \cdot \bXi_{n,k}(\omega)  \, ,
\end{eqnarray*}
and we finally  recognize in the previous expression that for all $0 \leq t \leq 1$
\begin{eqnarray*}
\bZ_{t}^{N}  = \sum_{(n,k) \in \I_{N}} \bpsi_{n,k} \cdot \bXi_{n,k}=  \sum_{(n,k) \in \I_{N}} \bpsi_{n,k}(t) \cdot \int_{U} \bphi^{T}_{n,k}(s) \, d\bm{W}_{s} = \int_{0}^{1} \bm{k}_{N}(t,s) \, d\bm{W}_{s}  \, .
\end{eqnarray*}
\end{proof}

We can now proceed to justify the main result of Propositon \ref{prop:orthProj}:

\begin{proof}
The finite-dimensional processes $\bZ^{N}$ defined through Lemma \ref{lem:finiteStoch} have sample-paths $t \mapsto \bZ^{N}_{t}(\omega)$ belonging to $\mathcal{E}_{N}$.
Moreover, for almost every $\omega$ in $\Omega$, and for all $n,k$ in $\I_{N}$, 
\begin{eqnarray*}
\mathcal{P} \Big \langle \bpsi_{n,k} , \bZ^{N}(\omega)\Big \rangle
&=&
\mathcal{P} \bigg \langle  \bpsi_{n,k},  \int_{0}^{1} \bm{k}_{N}(t,s) \, d\bm{W}_{s}(\omega) \bigg \rangle \, ,\\
&=&
\mathcal{P} \bigg \langle  \bpsi_{n,k},  \sum_{(p,q) \in I_{N}} \bpsi_{p,q}(\omega) \cdot \int_{0}^{1} \bphi^{T}_{p,q}(s) \, d\bm{W}_{s}(\omega) \bigg \rangle \, ,\\
&=& 
\int_{0}^{1} \bphi_{n,k}^{T}(s) \, d\bm{W}_{s}(\omega) \, ,
\end{eqnarray*}
because of the orthonormality property of $\bpsi_{n,k}$ with respect to $\langle  , \rangle$.
As the previous equalities holds for every $N>0$, the applications $\bx \mapsto \mathcal{P} \big \langle \bpsi_{n,k} , \bx \rangle $ can naturally be extended on $\ux \Omega'$ by continuity,
Therfore, it makes sense to write for all $(n,k)$ in $\I_{N}$, $ \mathcal{P} \langle \bpsi_{n,k} , \bZ^{N} \rangle = \lim_{N \to \infty}  \mathcal{P} \langle \bpsi_{n,k} , \bZ^{N} \rangle \stackrel{def}{=}  \mathcal{P} \langle \bpsi_{n,k} , \bX \rangle$ even if the $\bX$ is defined into a larger sample space than $\mathcal{E}$.
In other words,  we have
\begin{eqnarray*}
\mathcal{P} \big \langle \bpsi_{n,k} , \bX \big \rangle = \int_{0}^{1} \bphi^{T}_{n,k}(s) \, d\bm{W}_{s}  = \bXi_{n,k} \, ,
\end{eqnarray*}
and we can thus express the conditioned process $\bZ^{N} = \ExpN{\bX}$ as the orthogonal projection of $\bX$ onto the finite sample-path $\mathcal{E}_{N}$ by writing
\begin{equation*}
\bZ^{N} = \sum_{(n,k) \in \I_{N}} \bpsi_{n,k} \cdot \mathcal{P} \langle \bpsi_{n,k} , \bX \rangle \, . 
\end{equation*}
\end{proof}


\subsubsection{Optimality Criterion of the Sample Paths}

Proposition \ref{prop:orthProj} elucidates the structure of the conditioned processes $\bZ_{N}$ as path-wise orthogonal projections of $\bX$ on the finite-dimensional RKHS $\mathcal{E}_{N}$.
It  allows us to cast the finite sample-paths in a geometric setting and incidentally, to give a characterization of them as the minimizer of some functionals.
In doing so, we shed a new light on well-known results of interpolation theory~\cite{Wahba:1990,Rako:2005,Hsing:2009} and extend them to the multidimensional case.\\
The central point of this section reads as follows:

\begin{proposition}\label{prop:minCriterion}
Given a function $\bm{x}$ in $\mathcal{E}$, the function $\bm{x}_{N} = \left( \bPsi \circ \bDelta_{N} \right)[\bm{x}]$ belongs to $\mathcal{E}_{N}$ and is defined by the following optimal criterion:
 $\bm{x}_{N} $ is the only function in $\mathcal{E}$ interpolating $\bm{x}$ on $D_{N}$ such that the functional 
 \begin{eqnarray}\label{eq:energy}
 \langle \bm{y}, \bm{y} \rangle = {\Vert \mathcal{L}[\bm{y}](t) \Vert_{2}}^{2} = \int_{0}^{1} {\vert \mathcal{L}[\bm{y}](t) \vert_{2}}^{2} \, dt \, ,
\end{eqnarray}
takes its unique minimal value over $\mathcal{E}$ in $\bx_{N}$.
\end{proposition}

\begin{proof}
The space $\mathcal{E}_{N}$ has been defined as $\mathcal{E}_{N} = \bPsi_{N}\left[ \uxi \Omega_{N} \right] = \bPsi \circ \bDelta_{N} \left[ \mathcal{E} \right]$, so that for all $\bx$ in $\mathcal{E}$, $\bx_{N}$ clearly belongs to $\mathcal{E}_{N}$.
Moreover, $\bm{x}_{N}$ interpolates $\bx$ on $D_{N}$:
indeed, we know that the finite-dimensional operator $\bDelta_{N}$ and $\bPsi_{N}^{-1}$ are inverse of each other $\bDelta_{N} = \bPsi_{N}^{-1}$, which entails that for all $t$ in $D_{N}$
\begin{eqnarray*}
\bx_{N}(t) = \left( \bPsi \circ \bDelta_{N} \right)[\bm{x}](t) =  \left( \bPsi_{N} \circ \bDelta_{N} \right)[\bm{x}](t) = \bx(t) \, ,
\end{eqnarray*}
where we use the fact that for any $\bxi$ in $\uxi \Omega'$, and for all $t$ in $D_{N}$, $\bPsi_{N}[\bxi](t) = \bPsi[\bxi](t)$ (recall that $\bpsi_{n,k}(t) = \bm{0}$ if $n > N$ and $t$ belongs to $D_{N}$). \\
Let us now show that $\bx_{N}$  is determined in $\mathcal{E}$ by the announced optimal criterion.
Suppose $\by$ belongs to $\mathcal{E}$ and interpolates $\bx$ on $D_{N}$ and remark that we can write
\begin{eqnarray*}
 \langle \by, \by \rangle =  {\Vert \mathcal{L}[\by] \Vert_{2}}^{2}  = {\Vert (\bPhi \circ \bDelta)[\by](t) \Vert_{2}}^{2}   = {\Vert \bDelta[\by] \Vert_{2}}^{2}  \, ,
\end{eqnarray*}
since $\bPhi$ is an isometry.
Then, consider $\bDelta[\by]$ in  $l^{2}(\uxi \Omega)$ and remark that, since for all $(n,k)$ in $\I_{N}$,  $\bdelta_{n,k}$ are Dirac measures supported by $D_{N}$, we  have 
\begin{eqnarray*}
\forall \; (n,k) \in \I_{N} \, , \quad \bDelta_{n,k}[\by] = \mathcal{P}(\bdelta_{n,k}, \by) = \mathcal{P}(\bdelta_{n,k}, \bx) =  \bDelta_{n,k}[\bx] = \bDelta_{n,k}[\bx_{N}]\, .
\end{eqnarray*}
This entails
\begin{eqnarray*}
 {\Vert \bDelta[\by] \Vert_{2}}^{2} \, dt
=
\sum_{(n,k) \in \I}  {\vert \bDelta_{n,k}[\by] \vert_{2}}^{2}
\geq
\sum_{(n,k) \in \I_{N}}  {\vert \bDelta_{n,k}[\by] \vert_{2}}^{2}
=
 {\Vert \bDelta[\bx_{N}] \Vert_{2}}^{2} \, dt \, .
\end{eqnarray*}
since by definition of $\bx_{N}$, $\bdelta_{n,k}[\bx_{N}] = \bm{0}$ if $n>N$.
Moreover,  the minimum $ \langle \bx_{N}, \bx_{N} \rangle$ is only attained for $\by$ such that $\bdelta_{n,k}[\by] =   \bm{0}$ if $n>N$ and $\bdelta_{n,k}[\by] = \bdelta_{n,k}[\bx] $ if $n \leq N$, which defines univocally $\bx_{N}$.
This establishes that for all $\by$ in $\mathcal{E}$ such that for all $t$ in $D_{N}$, $\by(t) = \bx(t)$, we have $ \langle \bm{x}_{N}, \bm{x}_{N} \rangle \leq  \langle \by, \by \rangle$ and the equality case holds if and only if $\by = \bx_{N}$.
\end{proof}

\begin{remark}
When $\mathcal{L}$ represents a regular differential operator of order $d$, $\sum_{i=1}^{d} a_{i}(t) D^{i}$ where $D = \frac{d}{dt}$, that is for
\[
d\bm{X}_{t}
=
 \bm{\alpha}(t) \cdot \bm{X}_t + \sqrt{\bm{\Gamma}(t)} \cdot dW_t \, ,
\]
with
\[
 \bm{\alpha}(t) =  \left[ \begin{array}{ccccc}
		0&1\\
		&   \ddots & \ddots& \\
		&&\ddots&1\\
		a_{d}&a_{d-1}&\ldots&a_{1}\\
\end{array}
\right]
, \quad 
\bm{\sqrt{\bm{\Gamma}(t)}} = 
\left[
\begin{array}{c}
0\\
0\\
\vdots\\
1
\end{array}
\right] \, .
\]
the finite-dimensional sample paths coincide exactly the spline interpolation of order $2d+1$, which are well-known to satisfy the previous criterion~\cite{Kimel3}. This example will be further explored in the example section. 
\end{remark}

The Dirichlet energy simply appears as the squared norm induced on $\mathcal{E}$ by the inner product $\langle , \rangle$, which in turn can be characterized as a Dirichlet quadratic form on $\mathcal{E}$.
Actually, such a Dirichlet form can be used to define the Gauss-Markov process, extending  the Gauss-Markov property to processes indexed on multidimensional spaces  parameter~\cite{Pitt:1971kl}.
In particular, for a $n$-dimensional parameter space, we can condition such Gauss-Markov processes on a smooth $n\!-\!1$-dimensional boundary. 
Within the boundary, the sample paths of the resulting conditioned process (the solution to the prediction problem in~\cite{Pitt:1971kl}) are the solutions to the corresponding Dirichlet problems for the elliptic operator associated with the Dirichlet form.

The characterization of the basis as the minimizer of such a  Dirichlet energy \eqref{eq:energy} gives rise to an alternative method to compute the basis as the solution of a Dirichlet boundary value problem for an elliptic differential operator:

\begin{proposition}\label{prop:diffeq}
	Let us assume that $\balpha$ and $\sqrt{\bm{\Gamma}}$ are continuously differentiable and that $\sqrt{\bm{\Gamma}}$ is invertible. Then the functions $\bm{\mu}_{n,k}$ are defined as:
	\[ \bm{\mu}_{n,k}(t)=\begin{cases}
	 \bm{\mu}^l(t) & t\in [l_{n,k},m_{n,k}] \\
	  \bm{\mu}^r(t) & t\in[m_{n,k},r_{n,k}]\\
	\bm{0} & \text{else},
	\end{cases}
	\] where $\bm{\mu}^l$ and $\bm{\mu}^r$ are the unique solutions of the second order $d$-dimensional linear differential equation
	\begin{equation}\label{eq:bvp}
			\bm{u}'' + \left( \bm{\Gamma}^{-1} \left( \balpha^{T}   \bm{\Gamma} -  \bm{\Gamma}' \right) - \balpha \right) \bm{u}'- \left( \bm{\Gamma}^{-1} \left( \balpha^{T}   \bm{\Gamma} -  \bm{\Gamma}' \right)  \balpha + \balpha' \right) \bm{u} = \bm{0}\\
	\end{equation}
 with the following boundary value conditions:\\

\centering{
\begin{minipage}{.4\textwidth}
	\[\begin{cases}
                \bm{\mu}^l(l_{n,k}) =  \bm{0} \\
	 \bm{\mu}^l(m_{n,k}) = \bm{I}_{d} 
	\end{cases}
	\]
\end{minipage}
\begin{minipage}{.4\textwidth}
	\[\begin{cases}
		 \bm{\mu}^r(m_{n,k}) = \bm{I}_{d} \\
	 	 \bm{\mu}^r(r_{n,k}) = \bm{0}   
	\end{cases}
	\]
\end{minipage}
}
\end{proposition}

\begin{proof}
	By Proposition \ref{prop:minCriterion}, we know that $\bm{\mu}_{n,k}(t)$ minimizes  the convex functional
	\[\int_0^1 {\vert \mathcal{L} [\bm{u}](s)\vert_{2}}^2\,ds \]
	over  $\mathcal{E}$, being equal to zero outside the interval $[l_{n,k},r_{n,k}]$ and equal to one at the point $t=m_{n,k}$. 
	Because of the hypotheses on $\balpha$ and $\sqrt{\Gamma}$, we have $\mathcal{L} = \mathcal{D}$ and we can additionally restrain our search to functions that are twice continuously differentiable.
	Incidentally, we only need to minimize separately the contributions on the interval $[l_{n,k},m_{n,k}]$ and $[m_{n,k},r_{n,k}]$.
	On both intervals, this problem is a classical Euler-Lagrange problem (see e.g.~\cite{Allaire:2007}) and is solved using basic principles of calculus of variations. 
	We easily identify the Lagrangian of our problem as
	\begin{eqnarray*}
	L(t,\bm{u},\bm{u}') &=& {{\left \vert \Big(\bm{u}' - \balpha(t)\,\bm{u}(t) \Big)\left( \sqrt{\bm{\Gamma(t)}} \right)^{-1}\right \vert}_{2}}^2 \\
	&=&  \Big(\bm{u}'(t) - \balpha(t)\,\bm{u}(t) \Big)^{T} {\Big(\bm{\Gamma}(t)	\Big)}^{-1}\Big(\bm{u}'(t) - \balpha(t)\,\bm{u}(t) \Big).
	\end{eqnarray*}
	From there, after some simple matrix calculations, the Euler-Lagrange equations
	\[\frac{\partial L(t,\bm{u},\bm{u}')}{\partial u_i} - \frac{d}{dt} \left (\frac{\partial L(t,\bm{u},\bm{u}')}{\partial u'_i} \right)=0, \qquad i=1,\,\ldots,\,d.\]
	can be expressed under the form:
\[ \bm{u}'' + \left( \bm{\Gamma}^{-1} \left( \balpha^{T}   \bm{\Gamma} -  \bm{\Gamma}' \right) - \balpha \right) \bm{u}'- \left( \bm{\Gamma}^{-1} \left( \balpha^{T}   \bm{\Gamma} -  \bm{\Gamma}' \right)  \balpha + \balpha' \right) \bm{u} = \bm{0} \]
which ends the proof. 
\end{proof}

\begin{remark}
It is a simple matter of calculus to check that the expression of $\bm{\mu}$ given in Proposition  \ref{bridgeProp} satisfies equation \eqref{eq:bvp}.
Notice also that in the case $\bm{\Gamma} = \bm{I}_{d}$, the differential equation becomes
\[ \bm{u}'' + \left(   \balpha^{T}  - \balpha \right) \bm{u}'- \left(  \balpha^{T}     \balpha + \balpha' \right) \bm{u} = \bm{0} \, ,\] 
which is further simplified for constant or symmetrical $\balpha$.
\end{remark}

Under the hypotheses of Proposition \ref{prop:diffeq}, we can thus define $\bm{\mu}_{n,k}$ as the unique solution to the second-order linear differential equation \eqref{eq:bvp} with the appropriate boundary values conditions.
From this definition, it is then easy to derive the basis $\bpsi_{n,k}$ by completing the following program:
\begin{enumerate}
\item Compute the $t \mapsto \bm{\mu}_{n,k}(t)$ by solving the linear ordinary differential problem.
\item Apply the differential operator $\mathcal{D}$ to get the functions $\mathcal{D}[\bm{\mu}_{n,k}]$.
\item Orthonormalize the column functions $t \mapsto c_{j}(\mathcal{D}[\bm{\mu}_{n,k}(t)])$  by Gram-Schmidt process.
\item Apply the integral operator $\mathcal{K}$  to get the desired functions $\bpsi_{n,k}$ (or equivalently multiply the original function $t \mapsto \bm{\mu}_{n,k}(t)$ by the corresponding Gram-Schmidt triangular matrix).
\end{enumerate}
Notice finally that each of these points are easily implemented numerically.


\section{Exemples: Derivation of the Bases for Some Classical  Processes}


 \subsection{One-Dimensional Case}\label{ssect:OneDimCalculus}

 In the one-dimensional case, the construction of the Gauss-Markov process is considerably simplified since we do not have to consider the potential degeneracy of matrix-valued functions.
 Indeed, in this situation, the centered Gauss-Markov process $X$ is solution of the one-dimensional stochastic equation
 \begin{equation*}
 dX_{t} = \alpha(t) \, X_{t} \, dt + \sqrt{\Gamma(t)} \, dW_{t} \, ,
 \end{equation*}
 with $\alpha$ homogeneously H\"{o}lder continuous and $\Gamma$ positive continuous function.
 We then have the  Doob representation  
 \begin{equation*}
 X_t = g(t) \int_{0}^{t} f(s) \, dW_s \, , \quad \mathrm{with}\quad g(t) = e^{\int_{0}^{t} \alpha(v) \, dv} \, , \, f(t) = \sqrt{\Gamma(t)} \, e^{-\int_{0}^{t} \alpha(v) \, dv} \, .
 \end{equation*}
 Writing the function $h$ as
 \begin{equation*}
 h(t) = \int_{0}^{t} f^{2}(s) \, ds \, ,
 \end{equation*}
 the covariance of the process  reads for any $0 \leq t,s \leq 0$ as
 \begin{equation*}
 C(t,s) = g(t) g(s) h(t \wedge s) \, .
 \end{equation*}
 The variance of the Gauss-Markov bridge $B_{t}$ pinned in $t_{x}$ and $t_{z}$ yields
 \begin{eqnarray*}
 \left(\sigma_{t_{x},t_{z}}(t)\right)^{2}= g(t)^{2} \frac{\big( h(t)-h(t_{x}) \big) \big( h(t_{z})-h(t) \big)}{h(t_{z})-h(t_{x})} \, .
 \end{eqnarray*}
 These simple relations entails that the functions $\psi_{n,k}$ are defined on their supports $S_{n,k}$ by $\psi_{n,k}(t)^2 = \Exp{\left(\delta^{n}(t)\right)^{2}}$ with
 \begin{eqnarray*}
 \lefteqn{
 \Exp{\left(\delta^{n}(t)\right)^{2}}
 =
 \left(\sigma_{l_{n,k},r_{n,k}}(t)\right)^{2}
  }
 \\
 &&
 - \left(
\mathbbm{1}_{S_{n+1,2k}}(t) \left(\sigma_{l_{n,k},m_{n,k}}(t)\right)^{2}
 +
\mathbbm{1}_{S_{n+1,2k+1}}(t) \left(\sigma_{l_{n,k},m_{n,k}}(t)\right)^{2}
 \right) \, .
 \end{eqnarray*}
 This reads on $S_{n+1,2k}$ as
 \begin{eqnarray*}
 \lefteqn{ \psi_{n,k}(t)^2 = g(t)^2 \left[ \frac{\big( h(t)-h(l_{n,k}) \big) \big( h(r_{n,k})-h(t) \big)}{ h(r_{n,k})-h(l_{n,k}) } \right. } 
	\\
	&& 
	\qquad \qquad \qquad \qquad \qquad - \left. \frac{(h(t)-h(l_{n,k}))(h(m_{n,k})-h(t))}{h(m_{n,k})-h(l_{n,k})}\right]
 	 \, ,
 \end{eqnarray*}
 and on $S_{n+1,2k+1}$ as
 \begin{eqnarray*}
  \lefteqn{
 	\psi_{n,k}(t)^2 = g(t)^2 \left[ \frac{(h(t)-h(l_{n,k}))(h(r_{n,k})-h(t))}{h(r_{n,k})-h(l_{n,k})}  \right. }
	\\
	&&
	\qquad \qquad \qquad \qquad \qquad - \left. \frac{(h(t)-h(m_{n,k}))(h(r_{n,k})-h(t))}{h(r_{n,k})-h(m_{n,k})}\right]  \, , 
 \end{eqnarray*}
 and therefore we have:
 \begin{equation*}
 	\psi_{n,k}(t) = \begin{cases}
 	\displaystyle  \frac{\sigma_{n,k} \, g(t) (h(t)-h(l_{n,k}))}{g(m_{n,k}) (h(m_{n,k})-h(l_{n,k}))}  \, , 
	&   l_{n,k}\leq t \leq m_{n,k} \, ,\\
	 \\
 	\displaystyle  \frac{\sigma_{n,k} \,  g(t) (h(r_{n,k})-h(t))}{g(m_{n,k})(h(r_{n,k})-h(m_{n,k}))}  \, , 
	 & m_{n,k} \leq t \leq r_{n,k} \, ,\\
 \end{cases} 
 \end{equation*}
 with
 \[
 \sigma_{n,k} = \sqrt{\frac{\big(h(r_{n,k}) - h(m_{n,k}) \big) \big(h(m_{n,k})-h(l_{n,k})\big)}{h(r_{n,k})-h(l_{n,k})}} \, .
 \]
 As for the first element, it simply results from the conditional expectation of the one-dimensional bridge pinned in $l_{0,0}= 0$ and $r_{0,0}=1$:
 \begin{equation*}
 \psi_{0,0}(t) =
 \frac{g(t) (h(t) - h(l_{0,0}))}{\sqrt{h(r_{0,0})-h(l_{0,0})}} \, .
 \end{equation*}
 In this class of processes, two paradigmatic processes are the Wiener process and the Ornstein-Uhlenbeck processes with constant coefficients. 
 In the case of the Wiener process, $h(t) = t $ and $g(t)=1$, which yields the classical triangular-shaped functions Schauder functions used by L\'{e}vy~\cite{Levy}.
 As for the Ornstein-Uhlenbeck process with constant coefficients $\alpha$ and $\sqrt{\Gamma}$, we have $g(t) = \exp(\alpha \, t)$, $f(t) = \sqrt{\Gamma} \exp(-\alpha\,t)$ and $h(t)=\frac{\Gamma}{2\alpha}(1-e^{-2\alpha\,t})$, yielding for the construction basis the expression:
 \begin{equation*}
 	\psi_{n,k}(t) = \begin{cases}
 	\displaystyle \sqrt{\frac{\Gamma}{\alpha}} \frac{\text{sinh}(\alpha(t-l_{n,k}))}{\sqrt{\text{sinh}(\alpha(m_{n,k}-l_{n,k}))}} \, , & l_{n,k}\leq t \leq m_{n,k} \, ,\\
 	\displaystyle \sqrt{\frac{\Gamma}{\alpha}}	\frac{\text{sinh}(\alpha(r_{n,k}-t))}{\sqrt{\text{sinh}(\alpha(m_{n,k}-l_{n,k}))}} \, , & m_{n,k}\leq t \leq r_{n,k } \, , 
 \end{cases}
 \end{equation*}
 and
 \[\psi_{0,0}(t)= 	\sqrt{\frac{\Gamma}{\alpha}}\frac{e^{-\alpha/2}\text{sinh}(\alpha\,t)}{\sqrt{\text{sinh}(\alpha)}}\]
 which were already evidenced in \cite{ThibaudJOSP:2008}.


 \subsection{Multidimensional Case}\label{ssect:MultiDimCalculus}

In the multidimensional case, the explicit expressions for the basis functions $\bpsi_{n,k}$ makes fundamental use of the flow $\bm{F}$ of the underlying linear equation \eqref{eq:linDeter} for a given function $\bm{\alpha}$. 
For commutative forms of $\bm{\alpha}$ (i.e such that $\bm{\alpha}(t) \cdot \bm{\alpha}(s) = \bm{\alpha}(s) \cdot \bm{\alpha}(t)$ for all $t, s$), the flow can be formally expressed as an exponential operator.
It is however a notoriously difficult problem to find a tractable expression for general $\alpha$. 
As a consequence, it is only possible to provide closed-from formulae for our basis functions in very specific cases.


\subsubsection{Multi-Dimensional Gauss-Markov Rotations}

We consider in this section $\balpha$ antisymmetric and constant and $\sqrt{\bm{\Gamma}} \in \R^{d\times m}$ such that $\bm{\Gamma} = \sigma^2 \bm{I}_d$. 
Since $\balpha^T(t) = -\balpha(t)$, we have: 
\[ \bF(s,t)^T= \bF(s,t)^{-1},\]
i.e. the flow is unitary. This property implies that
\begin{align*}
	\bm{h}_u(s,t) &=\sigma^2\int_{s}^{t}  \bF(w,u)  \bF(w,u)^T\,dw = \sigma^2 (t-s) \bm{I}_d \, ,
\end{align*}
which yields by definition of $\bm{\sigma}_{n,k}$
\[\bm{\sigma}_{n,k} \cdot \bm{\sigma}_{n,k}^{T} = \sigma^2 \frac{(m_{n,k} - l_{n,k})(r_{n,k} - m_{n,k})}{r_{n,k} - l_{n,k}} \bm{I}_d\, .\] 
The square root $\bm{\sigma}_{n,k}$ is then uniquely defined (both by choosing Cholesky and symmetrical square root) by
\[\bm{\sigma}_{n,k} = \sigma \sqrt{\frac{(m_{n,k} - l_{n,k})(r_{n,k} - m_{n,k})}{(r_{n,k} - l_{n,k})}} \bm{I}_d\, , \] 
and $\bpsi_{n,k}(t)$ reads
\[\bpsi_{n,k}(t) = \begin{cases}
 \displaystyle \sigma \sqrt{\frac{r_{n,k} - m_{n,k}}{(m_{n,k} - l_{n,k})(r_{n,k} - l_{n,k})}} (t-l_{n,k}) \bF(m_{n,k},t) \, ,& l_{n,k}\leq t \leq m_{n,k} \, ,\\
 \displaystyle \sigma \sqrt{\frac{m_{n,k} - l_{n,k}}{(r_{n,k} - m_{n,k})(r_{n,k} - l_{n,k})}} (r_{n,k}-t) \bF(m_{n,k},t) \, ,& l_{n,k}\leq t \leq m_{n,k} \, .\\
\end{cases} \]
Recognizing the $(n,k)$ element of the Schauder basis for the construction of the one-dimensional Wiener process
\[s_{n,k}(t) = \begin{cases}
 \displaystyle \sqrt{\frac{r_{n,k} - m_{n,k}}{(r_{n,k} - l_{n,k})(m_{n,k} - l_{n,k})}} (t-l_{n,k}) \, ,& l_{n,k}\leq t \leq m_{n,k} \, ,\\
 \displaystyle \sqrt{\frac{m_{n,k} - l_{n,k}}{(r_{n,k} - l_{n,k})(r_{n,k} - m_{n,k})}} (r_{n,k}-t) \, ,& l_{n,k}\leq t \leq m_{n,k} \, ,\\
\end{cases}\]
we obtain the following formula:
\[\bpsi_{n,k}(t) = \sigma \, s_{n,k}(t) \,  \bF(t-m_{n,k}) \, .\] 
This form shows that the Schauder basis for multidimensional rotations results from the multiplication of the triangular-shaped elementary function used for the L\'{e}vy-Cesielski construction of the Wiener process with the flow of the equation, i.e. the elementary rotation. \\
The simplest example of this kind is the stochastic sine and cosine process corresponding to:
\begin{equation}\label{eq:Rot}
	\balpha = \left ( \begin{array}{cc} 0 & 1 \\ -1 & 0 \end{array}\right) \qquad \text{and}\qquad
\sqrt{\bm{\Gamma}} = \sigma^2 \bm{I}_2 \, .
\end{equation}
In that case, $\bpsi_{n,k}$ has the expression
\[\bpsi_{n,k}(t) = s_{n,k}(t) \left ( \begin{array}{cc}\cos(t-m_{n,k}) & -\sin(t-m_{n,k}) \\ \sin(t-m_{n,k}) & \cos(t-m_{n,k})\end{array} \right)\]
\begin{figure}
	\centering
		\includegraphics[width=.6	\textwidth]{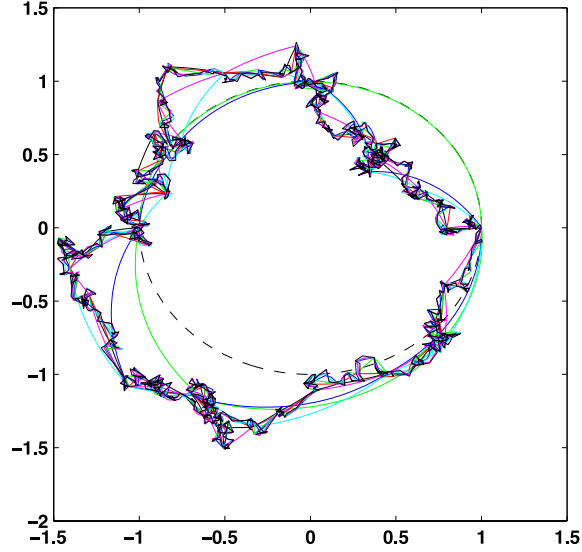}
	\caption{Construction of the stochastic sine and cos Ornstein-Uhelenbeck processes for the parameters given in equation \eqref{eq:Rot}: multi-resolution construction of the sample path. }
	\label{fig:Rot}
\end{figure}
Interestingly, the different basis functions have the structure of the solutions of the non-stochastic oscillator equation. One of the equations perturbs the trajectory in the radial component of the deterministic solution, and the other one in the tangential direction.
We represent such a construction in Figure \ref{fig:Rot} with the additional conditioning that $\bX_1=\bX_0$, i.e. imposing that the trajectory forms a loop between time $0$ and $1$.


\subsubsection{The Successive Primitives of the Wiener Process}

In applications, it often occurs that people use smooth stochastic processes to model the integration of noisy signals. 
This is for instance the case of a particule subject to a Brownian forcing or of the synaptic integration of noisy inputs~\cite{Touboul:2008}. 
Such smooth processes involves in general integrated martingales, and the simplest example of such processes are the successive primitives of a standard Wiener process. \\
Let $d>2$ and denote by $X^{d}_t$ the $d-1$th order primitive of the Wiener process. 
This process can be defined via the lower order primitives $X^{k}_t$ for $k<d$ via the relations: 
\[
\begin{cases}
	dX^{k+1}_t & =  X^{k}_t \, dt \qquad k<d\\
	dX^{1}_t &=	dW_t
\end{cases}
\]
where $W_t$ is a standard real Wiener process. 
These equations can be written in our formalism
\[
d\bm{X}_{t}
=
 \bm{\alpha}(t) \cdot \bm{X}_t + \sqrt{\bm{\Gamma}(t)} \cdot dW_t \, ,
\]
with
\[
 \bm{\alpha}(t) =  \left[ \begin{array}{ccccc}
		0&1\\
		&   \ddots & \ddots& \\
		&&\ddots&1\\
		&&&0\\
\end{array}
\right]
, \quad 
\bm{\sqrt{\bm{\Gamma}(t)}} = 
\left[
\begin{array}{c}
0\\
0\\
\vdots\\
1
\end{array}
\right] \, .
\]
In particular, though none of the integrated processes $X^k$ for $K>1$ is Markov by itself, the $d$-uplet $\bX = (X^d,\ldots,X^1)$ is a Gauss-Markov process. \\
Furthermore because of the simplicity and the sparcity of the matrices involved, we can identify in a compact form all the variables used in the computation of the construction basis for these processes. In particular, the flow $\bm{F}$ of the equation is the exponential of the matrix $\balpha$, and since $\alpha$ is nilpotent, it is easy to show that $\bm{F}$ has the expression:
\[\bm{F}(s,t) = \left [ \begin{array}{cccccc}
1  & (t-s)  & \frac{(t-s)^2}{2} & \ldots & \frac{(t-s)^{d-1}}{(d-1)!}\\
& \ddots & \ddots & \ddots & \vdots\\
& & \ddots& \ddots  & \frac{(t-s)^2}{2} \\
& & &\ddots & (t-s) \\
& & & & 1
\end{array}\right]\]
and the only non-zero entry of the $d \times d$ matrix $\bm{\Gamma}$ is one at position $(d-1,d-1)$. 
Using this expression and the highly simple expression of $\bm{\Gamma}$, we can compute the general element of the matrix $\bm{h}_u(t,s)$, which reads:
\[\left(\bm{h}_u(s,t)\right)_{i,j} = (-1)^{i+j}\frac{(t-u)^{2d-1 -(i+j)} - (s-u)^{2d-1 -(i+j)}}{\big(2d-1-(i+j)\big)(d-1-i)!(d-1-j)!} .\] 
Eventually, we observe that the functions $\bm{\psi}_{n,k}$ yielding the multi-resolution description of the integrated Wiener processes, are directly deduced from the matrix-valued function
\[\left(\bm{c}_{n,k}(t)\right)_{i,j} = \begin{cases}
 \displaystyle  \bm{\psi}_{n,k} \cdot \bm{L}^{-1}_{n,k} = \bm{g}(t) \bm{h}(l_{n,k},t)\, ,& l_{n,k}\leq t \leq m_{n,k} \, ,\\
 \displaystyle  \bm{\psi}_{n,k} \cdot \bm{R}^{-1}_{n,k} = \bm{g}(t) \bm{h}(t,r_{n,k})\, ,& m_{n,k}\leq t \leq r_{n,k} \, ,\\
\end{cases}\] 
whose components are further expressed as 
\[\left(\bm{c}_{n,k}(t)\right)_{i,j} = 
 \displaystyle  \sum_{p=i}^{d-1} (-1)^{p+j} \frac{t^{i-p}}{(i-p)!}  \frac{t^{2d-1 -(p+j)} - l_{n,k}^{2d-1 -(p+j)}}{\big(2d-1-(p+j)\big)(d-1-p)!(d-1-j)!} \, ,\\
\]
for $l_{n,k}\leq t \leq m_{n,k}$  and as
\[\left(\bm{c}_{n,k}(t)\right)_{i,j} =  \sum_{p=i}^{d-1} (-1)^{p+j} \frac{t^{i-p}}{(i-p)!}  \frac{m_{n,k}^{2d-1 -(p+j)} - t^{2d-1 -(p+j)}}{\big(2d-1-(p+j)\big)(d-1-p)!(d-1-j)!} \, ,\]
for $m_{n,k}\leq t \leq r_{n,k}$.
The final computation of the $\bpsi_{n,k}$ involves the computation of $\bm{L}_{n,k}$ and $\bm{R}_{n,k}$, which in the general case can become very complex. 
However, this expression is highly simplified when assuming that $m_{n,k}$ is the middle of the interval $[l_{n,k},r_{n,k}]$. 
Indeed, in that case, we observe that for any $(i,j)$ such that $i+j$ is odd, $(\bm{h}_m(l,r))_{i,j} = \bm{0}$ which induces the same property on the covariance matrix $\bm{\Sigma}_{n,k}$ and on the polynomials $(\bm{c}_{n,k}(t))_{i,j}$. 
This property gives therefore a preference to the dyadic partition that provides simple expressions for the basis elements in any dimensions, and allows simple computations of the basis.

\begin{remark}
Observe that for all $0 \leq i < d-1$, we have
\begin{eqnarray*}
\left(\bm{c}_{n,k}(t)\right)'_{i,j} &=& \left(\bm{c}_{n,k}(t)\right)_{i+1,j} \pm \sum_{p=i}^{d-1} (-1)^{p+j} \frac{t^{i-p}}{(i-p)!} \frac{t^{2d-1 -(p+j)} - l_{n,k}^{2d-2 -(p+j)}}{(d-1-p)!(d-1-j)!} \, , \\
&=& \left(\bm{c}_{n,k}(t)\right)_{i+1,j} \pm \frac{t^{d-j-1}}{(d-j-1)!} \underbrace{\sum_{q=0}^{d-1-i} \frac{(-t)^{q} t^{(d-1-i) -p}}{p ! \big( (d-1-i) -p \big)!}}_{\displaystyle 0} \, . 
\end{eqnarray*}
As $\bm{L}_{n,k}$ and $\bm{R}_{n,k}$ are constant, we immediately deduce the important relation that for all $0 \leq i \leq d-1$, $\left(\bpsi_{n,k}(t)\right)^{(i)}_{0,j} =  \left(\bpsi_{n,k}(t)\right)_{i,j}$.
This indicates that each finite-dimensional sample paths of our construction has components that satisfies the non-deterministic equation associated with the iteratively integrated Wiener process.
Actually, this fact is better stated remembering that the Schauder basis $\bpsi_{n,k}$ and the  corresponding orthonormal basis $\bphi_{n,k}:[0,1] \rightarrow \mathbb{R}^{1 \times d}$ are linked through the equation \eqref{eq:simplePsiPhi}, which reads
\begin{multline*}
\left[ \begin{array}{ccc}
		(\bpsi_{n,k})'_{0,0} & \ldots & (\bpsi_{n,k})'_{0,d-1}\\
		\vdots & & \vdots  \\
		(\bpsi_{n,k})'_{d-2,0} & \ldots & (\bpsi_{n,k})'_{d-2,d-1}\\
		(\bpsi_{n,k})'_{d-1,0} & \ldots & (\bpsi_{n,k})'_{d-1,d-1}\\
\end{array}
\right]
=
\left[ \begin{array}{ccc}
		(\bpsi_{n,k})_{1,0} & \ldots & (\bpsi_{n,k})_{1,d-1}\\
		\vdots & & \vdots  \\
		(\bpsi_{n,k})_{d-1,0} & \ldots & (\bpsi_{n,k})_{d-1,d-1}\\
		0 & \ldots & 0\\
\end{array}
\right]\\
 + 
 \left[ \begin{array}{ccc}
		0 & \ldots & 0\\
		\vdots & & \vdots   \\
		 0 & \ldots & 0\\
		(\bphi_{n,k})_{0,0} & \ldots & (\bphi_{n,k})_{0,d-1}\\
\end{array}
\right] \, .
\end{multline*}
Additionally, we realize that the orthonormal basis is entirely determined by the one-dimensional families $(\bphi_{n,k})_{0,j}$, which are mutually orthogonal functions satisfying $(\bphi_{n,k})_{0,j} = (\bpsi_{n,k})^{(d)}_{0,j}$.
\end{remark}

\begin{figure}
	\centering
		\includegraphics[width=.7\textwidth]{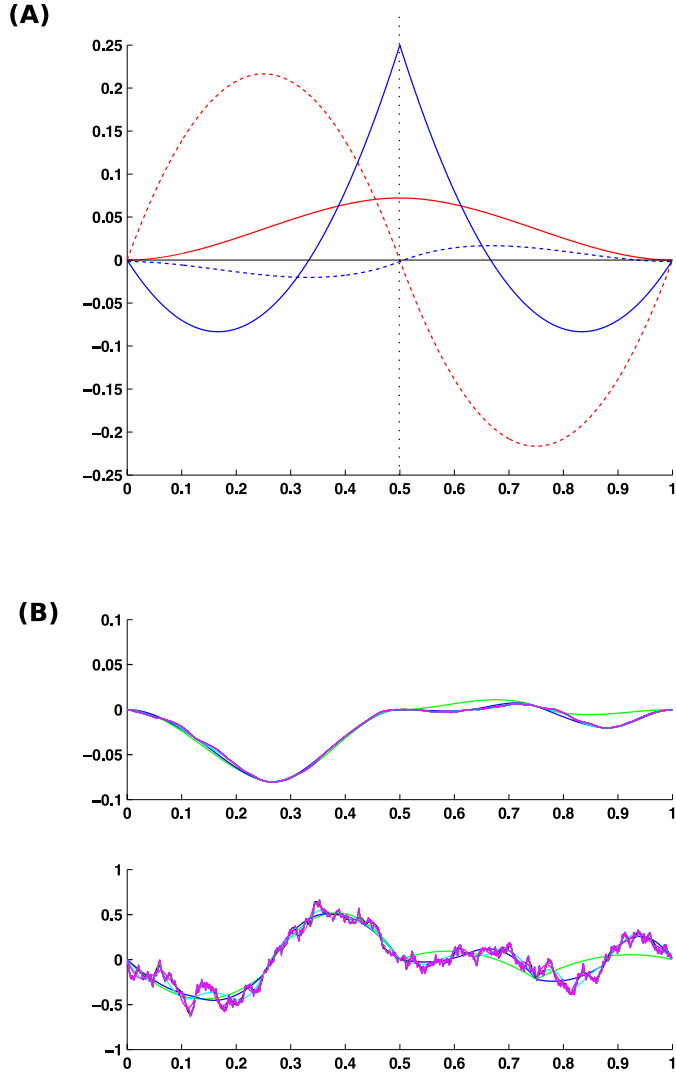}
	\caption{(A) Basis for the construction of the Integrated Wiener process ($d=2$). Plain red curve: $\psi_{1,1}$, dashed red: $\psi_{2,1}$, plain blue: $\psi_{1,2}$ and dashed blue: $\psi_{2,2}$. (B) $10$-steps construction of the process. We observe that each basis accounts separately for different aspects of sample path: $\psi_{2,1}$ fixes the value of the integrated process at the middle point $m_{n,k}$ and $\psi_{2,2}$ the value of the derivative of the process at the endpoints $\{l_{n,k}, r_{n,k}\}$ in relationship with the value of the Wiener process at the middle point, whose contributions are split between functions $\psi_{1,1}$ and $\psi_{1,2}$ (see Fig.~\ref{fig:IWPFig}). }
	\label{fig:IWPFig}
\end{figure}

We study in more details the case of the integrated and doubly-integrated Wiener process  ($d=2$ and $d=3$), for which closed-form expressions are provided in Appendices \ref{append:IP} and \ref{append:DIP}.
As expected, the first row of the basis functions for the integrated Wiener process turns out to be  the well-known cubic Hermite splines~\cite{Dahmen:2000}.
These functions have been widely used in numerical analysis and actually constitute the basis of lowest degree in a wider family of bases known as the natural basis of polynomial splines of interpolation~\cite{Kimel2}.
Such bases are used to interpolate data points with constraint of smoothness of different degree (for instance the cubic Hermite splines ensure that the resulting interpolation is in $C^{1}[0,1]$).
The next family of  splines of interpolation (corresponding to the $C^{2}$ constraint) is naturally retrieved by considering the construction of the doubly-integrated Wiener process:
we obtain a family of three $3$-dimensional functions, that constitutes the columns of a $3\times 3$ matrix that we denote $\bpsi$.
The top row is made of  polynomials of degree five, which have again simple expressions when $m_{n,k}$ is the middle of the interval $[l_{n,k},r_{n,k}]$.


\section{Stochastic Calculus from the Schauder point of view}\label{sec:CalculSto}

Thus far, all calculations, propositions and theorems are valid for any finite dimensional Gauss-Markov process and all the results are valid path-wise, i.e. for each sample path. The analysis provides a Schauder description of the processes as a series of standard Gaussian random variables multiplied by certain specific functions, that form a Schauder basis in the suitable spaces. This new description of Gauss-Markov processes provides a new way for treating problems arising in the study of stochastic processes. As examples of this, we derive It\^o formula and Girsanov theorem from the Schauder viewpoint. Note that these results are equalities in law, i.e. dealing with the distribution of stochastic processes, which is a weaker notion compared to the path-wise analysis. In this section, we restrict our analysis to the one-dimensional case for technical simplicity.

The closed-form expressions of the basis of functions $\psi_{n,k}$ in the one-dimensional case are given in section \ref{ssect:OneDimCalculus}. The differential and integral operators associated, introduced in section \ref{ssec:DualOps} are highly simplified in the one dimensional case. Let $U$ be a bounded open neighborhood of $[0,1]$, and denote by  $C(U)$ is the space of continuous real functions on $U$, $R(U)$ its topological dual, the space of Radon measures, $D_{0}(U)$ the space of test function in $C^{\infty}(U)$ which are zero at zero and it dual $D_{0}'(U)\subset R(U)$. We consider the Gelfand triple
\begin{equation}
 D_0(U) \subset C(U) \subset  L^{2}(U) \subset D_0'(U) \subset R(U) \nonumber
\end{equation}
The integral operator $\mathcal{K}$ is defined (and extended by dual pairing) by:
\[\mathcal{K}[\cdot](t) = \int_{U} \mathbbm{1}_{[0,t]}(s) g_{\alpha}(t) f_{\alpha}(s) \cdot \, ds\]
and the inverse differential operator $\mathcal{D}$ reads:
\[\mathcal{D}[\cdot](t) = \frac{1}{g_{\alpha}(t)} \frac{d}{dt} \! \left( \frac{\cdot}{f_{\alpha}(t)} \right).\]

Now that we dispose of all the explicit forms of the basis functions and related operators, we are in position to complete our program, and start by proving the very important  It\^o formula and its finite-dimensional counterpart before turning to Girsanov's theorem.


\subsection{It\^o's Formula}

A very useful theorem in the stochastic processes theory is the It\^o formula. We show here that this formula is consistent with the Schauder framework introduced. Most of the proofs can be found in the Appendix \ref{append:Ito}. The proof of It\^o formula is based on demonstrating the integration by parts property:

\begin{proposition}[Integration by parts]\label{prop:IPP}
	Let $(X_t)$ and $(Y_t)$ be two one-dimensional Gauss-Markov processes starting from zero. Then we have the following equality in law:
	\[X_t\,Y_t= \int_0^t X_s \circ dY_s + \int_0^t Y_s \circ dX_s\]
	where $\int_0^t A_s \circ dB_s$ for $A_t$ and $B_t$ two stochastic processes denotes the Stratonovich integral. In terms of It\^o's integral, this formula is written:
	\[X_t\,Y_t= \int_0^t X_s dY_s + \int_0^t Y_s dX_s + \langle X,Y \rangle_t\]
	where the brackets denote the mean quadratic variation. 
\end{proposition}

The proof of this proposition is quite technical and is provided in Appendix \ref{append:Ito}. It is based a thorough analysis of the finite-dimensional processes $X^N_t$ and $Y^N_t$. For this integration by parts formula and using a density argument, one can recover the more general It\^o formula:

\begin{theorem}[It\^o]\label{theo:Ito}
	Let $(X_t)_t$ be a Gauss-Markov process and $F \in C^2(\R)$. The process $f(X_t)$ is a Markov process and satisfies the relation:
	\begin{equation}\label{eq:Ito}
		f(X_t)=f(X_0)+\int_0^t f'(X_s)dX_s + \frac 1 2 \int_0^t f''(X_s) d\langle X\rangle_s
	\end{equation}
\end{theorem}

This theorem is proved in the Appendix \ref{append:Ito}.

It\^o's formula implies in particular that the multi-resolution description developed in the paper is valid for every smooth functional of a Gauss-Markov process. In particular, it allows a simple description of exponential functionals of Gaussian Markovian processes, which are of particular interest in mathematics and has many applications, in particular in economics (see e.g.~\cite{Yor:2001}).

Therefore, we observe that in the view of the article, It\^o's formula stems from the non-orthogonal projections of the basis elements the one on the other. For multidimensional processes, the proof of It\^o's formula is deduced from the one-dimensional proof and would involve study of the multidimensional bridge formula for ${\bm{X}}_t$ and ${\bm{Y}}_t$.

We eventually remark that this section provides us with a finite-dimensional counterpart if It\^o's formula for discretized processes, which has important potential applications, and further assesses the suitability of using the finite-resolution representation developed in this paper. Indeed, using the framework developed in the present paper allows considering finite-resolution processes and their transformation through nonlinear smooth transformation in a way that is consistent with the standard stochastic calculus processes, since the equation on the transformed process indeed converges towards its It\^o representation as the resolution increases.


\subsection{Girsanov Formula, a geometric viewpoint}

In the framework we developed, transforming a process $X$ into a process $Y$ is equivalent to substituting the Schauder construction basis related to $Y$ for the basis related to $X$. 
Such an operation provides a path-wise mapping for each sample path of $X$ onto a sample path of $Y$ having the same probability density in ${}_{\xi}\Omega'$. 
This fact shed a new light on the geometry of multi-dimensional Gauss-Markov processes, since the relationship between two processes is seen as a linear change of basis. In our framework, this relationship between processes is straightforwardly studied in the finite-rank approximations of the processes up the a certain resolution. Technical intricacy are nevertheless raised when dealing with the representation of the process itself in infinite dimensional Schauder spaces. We solve these technical issues here, and show that in the limit $N\to \infty$ one recovers Girsanov's theorem as a limit of the linear transformations between Gauss-Markov processes.\\
The general problem consists therefore in studying the relationship between two real Gauss-Markov $X$ and $Y$ that are defined by:
\[
\begin{cases}
dX_t &= \alpha_\dX(t) X_t  \, dt+ \sqrt{\Gamma_\dX(t)} \, dW_t\\
dY_t &= \alpha_\dY(t) X_t  \, dt+ \sqrt{\Gamma_\dY(t)} \, dW_t
\end{cases}
\]
We have noticed that the spaces ${}_{x}\Omega'$ are the same in the one-dimensional case as long as both $\Gamma_\dX$ and $\Gamma_\dY$ never vanish, and therefore make this assumption here. In order to further simplify the problem, we assume that $\gamma_{\dX,\dY}= \Gamma_\dX / \Gamma_Y$ is continuously differentiable. This assumption allows us to introduce the process $Z_t = \gamma_{\dX,\dY}(t) Y_t$ that satisfies the stochastic differential equation:
\begin{align*}
	dZ_t &= \frac{d}{dt}\left(\gamma_{\dX,\dY}(t) \right) Y_t \, dt + \gamma_{\dX,\dY}(t) dY_t\\
	&= \left(\frac{d}{dt}\left(\gamma_{\dX,\dY}(t) \right) + \gamma_{\dX,\dY}(t) \alpha_\dY(t)\right)\,Y_t\,dt + \sqrt{\Gamma_\dX(t)}\,dW_t\\
	&= \alpha_\dZ(t) \, Z_t \,dt + \sqrt{\Gamma_\dX(t)}\,dW_t
\end{align*}
with $\alpha_\dZ(t)=\frac{d}{dt}\left(\gamma_{\dX,\dY}(t)\right) \gamma_{\dX,\dY}(t)^{-1}+ \alpha_\dY(t) $. Moreover, if $\uZ \psi_{n,k}$ and $\uY \psi_{n,k}$ are the basis of functions that describe the process $Z$ and $Y$ respectively, we have $\uZ \psi_{n,k} = \gamma_{\dX,\dY} \cdot \uY \psi_{n,k}$. \\
The previous remarks allow us to restrict without loss of generality our study to processes defined for a same function $\sqrt{\Gamma}$, thus reducing the parametrization of Gauss-Markov processes to the linear coefficient $\alpha$. 
Observe that in the classical stochastic calculus theory, it is well-known that such hypothesis are necessary for the process $X$ to be absolutely continuous with respect to  $Y$ (through the use of Girsanov's theorem). \\
Let us now consider $\alpha, \beta$ and $\sqrt{\Gamma}$ three real H\"older continuous real functions, and introduce $\ag X$ and $\bg X$ solutions of the equations:
\[\begin{cases}
d \left (\ag X_t\right) &= \alpha(t) \left (\ag X_t\right) \, dt+ \sqrt{\Gamma(t)} \, dW_t\\
d \left (\bg X_t \right) &= \beta(t) \left (\bg X_t \right) \, dt+ \sqrt{\Gamma(t)} \, dW_t
\end{cases}\]
All the functions and tools related to the process $\ag X$ (resp $\bg X$) will be indexed by $\ag$ ($\bg$) in the sequel.

\subsubsection{Lift Operators}
Depending on the space we we are considering (either coefficients or trajectories), we define the two following operators mapping the process $\ag X_t$ on $\bg X_t$
	\begin{enumerate}
		\item The \emph{coefficients lift} operator ${}_{ \alpha, \beta}G$ is the linear operator mapping in $\uxi \Omega'$ the process $\ag X$ on the process $\bg X$:
	\[{}_{ \alpha, \beta}G= {}_{ \beta}\Delta \circ {}_{ \alpha}\Psi:\left(\uxi \Omega' ,\mathcal{B}\left(   \uxi \Omega' \right)  \right) \rightarrow \left(  \uxi \Omega', \mathcal{B}\left(   \uxi \Omega' \right) \right) .\]
	For any $\xi \in \uxi \Omega'$, the operator ${}_{ \alpha, \beta}G$ maps a sample path of $\ag X$ on a sample path of $\bg X$. 
		\item The \emph{process lift} operator ${}_{\alpha, \beta}F$ is the linear operator mapping in $\ux \Omega$ the process $\ag X$ on the process $\bg X$:
	\[{}_{ \alpha, \beta}H= {}_{ \alpha}\Psi \circ {}_{ \alpha}\Delta :\left(\ux \Omega ,\mathcal{B}\left(   \ux \Omega \right)  \right) \rightarrow \left(  \ux \Omega, \mathcal{B}\left(   \ux \Omega \right) \right) \]
\end{enumerate}

We summarize the properties of these operators now.
\begin{proposition} \label{prop:LiftOperator}
The operators ${}_{ \alpha, \beta}G$ and ${}_{ \alpha, \beta}H$ satisfies the following properties:
\renewcommand{\theenumi}{\roman{enumi}} 
\begin{enumerate}
	\item They are linear measurable bijections,
	\item For every $N>0$, the function  ${}_{ \alpha, \beta}G_{N} = P_{N} \circ {}_{ \alpha, \beta}G \circ I_{N}: \uxi \Omega' _{N} \to \uxi \Omega _{N}$ (resp. ${}_{ \alpha, \beta}H_{N} = P_{N} \circ {}_{ \alpha, \beta}H \circ I_{N}: \uxi \Omega' _{N} \to \uxi \Omega _{N}$) is a finite-dimensional linear operator, whose matrix representation is triangular in the natural basis of $\uxi \Omega_{N}$ (resp. $\uxi \Omega' _{N}$) and whose eigenvalues ${}_{ \alpha, \beta}\nu_{n,k}$ are given by
	\begin{equation}
	{}_{ \alpha, \beta}\nu_{n,k} 
	=
	\frac{g_{\alpha}(m_{n,k})}{g_{\beta}(m_{n,k})} \frac{\bg M_{n,k}}{\ag M_{n,k}}
	, \quad 0 \leq n \leq N, \quad 0\leq k < 2^{N-1} \, .
	\nonumber
	\end{equation}
	(resp.  ${}_{ \beta, \alpha}\nu_{n,k} =  ({}_{ \alpha, \beta}\nu_{n,k})^{-1}$).
	\item $\ab G$ and $\ab H$ are bounded operators for the spectral norm with
	\begin{equation*}
	\big \Vert \ab G  \big \Vert_{2} = \sup_{n} \sup_{k} {}_{ \alpha, \beta}\nu_{n,k} \leq \frac{\sup g_{\alpha}}{\inf g_{\beta}}  \frac{\sup f^{2}_{\alpha}}{\inf f^{2}_{\beta}} < \infty \, ,
	\end{equation*}
	and $\Vert \ab H  \big \Vert_{2} = \big \Vert {}_{\beta,\alpha} G  \big \Vert_{2}<\infty$. 
	 
	\item the determinants of ${}_{ \alpha, \beta}G_{N}$ (denoted $ \ab J_{N}$) and ${}_{ \alpha, \beta}H_{N}$ admit a limit when $N$ tends to infinity:
	\begin{eqnarray*}
	{}_{ \alpha, \beta}J\ =  \lim_{N \to \infty} \ab J_{N} = \exp{\left(\frac{1}{2}\left( \int_{0}^{1} (\alpha(t) - \beta(t)) \, dt \right) \right)} , 
	\end{eqnarray*}
	and	
	\[\lim_{N \to \infty}  \det\left( {}_{ \alpha, \beta}H_{N} \right) = \exp{\left(\frac{1}{2}\left( \int_{0}^{1} (\beta(t) - \alpha(t)) \, dt \right) \right)} ={}_{ \beta, \alpha}J\, .\]
\end{enumerate}
\end{proposition}

The proof of these properties elementary stem from the analysis done on the functions ${\bm{\Psi}}$ and ${\bm{\Delta}}$ that were previously performed, and these are detailed in Appendix \ref{append:LiftOperator}.


\subsubsection{Radon Nikodym Derivatives}

From the properties proved on the lift operators, we are in position to further analyze the relationship between the probability distributions of $\ag X$ and $\bg X$. 
We first consider the finite-dimensional processes $\ag X^N$ and $\bg X^N$. 
We emphasize that throughout this section, all equalities are true path-wise.

\begin{lemma}
Given the finite-dimensional measure $P^{N}_{\alpha}$ and $P^{N}_{\beta}$, the Radon-Nikodym derivative of $P^{N}_{\beta}$ with respect to $P^{N}_{\alpha}$ satisfies
\begin{eqnarray*}
\frac{d P^{N}_{\beta}}{d P^{N}_{\alpha}}(\omega) = \ab J_{N} \cdot \exp{ \left(- \frac{1}{2} \left( {\Xi_{N}(\omega)}^{T} \left( \ab S_{N}- Id_{\uxi \Omega_{N}}\right) \Xi_{N}(\omega)\right) \right)}
\end{eqnarray*}
with $\ab S_{N} = {\ab G_{N}}^{T} \cdot \ab G_{N}$ and the equality is true path-wise.
\end{lemma}

\begin{proof}
In the finite-dimensional case,  for all $N>0$, the probability measures $P^{N}_{\alpha}$, $P^{N}_{\beta}$  [$p^{N}_{\alpha}$] and the Lebesgue measure on $\ux \Omega^{N}$ are mutually absolutely continuous. We denote  $p^{N}_{\alpha}$ and $p^{N}_{\beta}$ the Gaussian density of $P^{N}_{\alpha}$ and $P^{N}_{\beta}$  with respect to the Lebesgue measure on $\ux \Omega^{N}$.
The Radon-Nikodym derivative of $P^{N}_{\beta}$ with respect to $P^{N}_{\alpha}$ is defined path-wise and is simply given by the quotient of the density of the vector $\lbrace \bg X^{N}(m_{i,j}) \rbrace$ with the density of the vector $\lbrace \ag X^{N}(m_{i,j}) \rbrace$ for $0 \leq i \leq N$, $0 \leq j < 2^{i-1}$, that is
\begin{eqnarray}
\lefteqn{
\frac{d P^{N}_{\beta}}{d P^{N}_{\alpha}}(\omega) = \frac{p^{N}_{\beta}(\ag X^{N}(\omega))}{p^{N}_{\alpha}(\ag X^{N}(\omega))} = }
\nonumber\\
&& \sqrt{\frac{\det{\left( \ag \Sigma_{N} \right)}}{\det{\left( \bg \Sigma_{N} \right)}}} \cdot \exp{\left( -\frac{1}{2}\left( {\ag X^{N}(\omega)}^{T} \left( \bg \Sigma^{-1}_{N} - \ag \Sigma^{-1}_{N} \right)    \ag X^{N}(\omega)  \right)\right)} \, .
\label{eq:finiteRadon}
\end{eqnarray}
We first make explicit
\begin{eqnarray}
\frac{\det{\left( \ag \Sigma_{N} \right)}}{\det{\left( \bg \Sigma_{N} \right)}} 
&=& 
\det{\left(\ag \Sigma_{N} \cdot \bg \Sigma^{-1}_{N} \right)} \, , \nonumber\\
&=&
 \det{\left( \ag \Psi_{N} \cdot {\ag \Psi_{N}}^{T} \cdot {\bg \Delta_{N}}^{T} \cdot \bg \Delta_{N} \right)} \, , \nonumber\\
 &=&
  \det{\left( \bg \Delta_{N} \cdot \ag \Psi_{N} \cdot {\ag \Psi_{N}}^{T} \cdot {\bg \Delta_{N}}^{T}   \right)} \, , \nonumber\\
  &=&
  \det{\left( \ab G_{N} \cdot {\ab G_{N}}^{T}   \right)} \, ,  \nonumber\\
  &=&
  \det{\left( \ab G_{N} \right)}^{2} \nonumber \, .
\end{eqnarray}
Then, we rearrange the exponent using the Cholesky decomposition 
\begin{eqnarray}
\bg \Sigma^{-1}_{N} - \ag \Sigma^{-1}_{N} 
=
{\bg \Delta_{N}}^{T} \cdot   \bg \Delta_{N}  - {\ag \Delta_{N}}^{T} \cdot \ag \Delta_{N}  \, , \nonumber
\end{eqnarray}
so that we write the exponent of \eqref{eq:finiteRadon} as
\begin{eqnarray}
&&{\ag X_{N}}^{T} \left( {\bg \Delta_{N}}^{T} \cdot \bg \Delta_{N} - {\ag \Delta_{N}}^{T} \cdot \ag \Delta_{N} \right)    \ag X_{N}   \, , \nonumber\\
&& \qquad =
 {\Xi_{N}}^{T} \cdot {\ag \Psi_{N}}^{T} \left( {\bg \Delta_{N}}^{T} \cdot \bg \Delta_{N} - {\ag \Delta_{N}}^{T} \cdot \ag \Delta_{N} \right)    \ag \Psi_{N}(\omega) \cdot \Xi_{N} \, , \nonumber\\
&& \qquad =
 {\Xi_{N}}^{T}  \left( {\ab G_{N}}^{T} \cdot \ab G_{N} - Id_{\uxi \Omega_{N}} \right)   \Xi_{N} \, . \nonumber
\end{eqnarray}
We finally reformulate \eqref{eq:finiteRadon} as
\begin{eqnarray}
\frac{d P^{N}_{\beta}}{d P^{N}_{\alpha}}(\omega)
=
\ab J_{N} \cdot \exp{\left( -\frac{1}{2}\left( {\Xi_{N}(\omega)}^{T}  \left( {\ab G_{N}}^{T} \cdot \ab G_{N} - Id_{\uxi \Omega_{N}} \right)   \Xi_{N}(\omega) \right)\right)} \, . \nonumber
\end{eqnarray}

\end{proof}

Let us now justify from a geometrical point of view why this formula is a direct consequence of the finite-dimensional change of variable formula on the model space $\uxi \Omega_N$.
If  we introduce $\ag \Delta_{N}$ the coefficient application related to $\ag X^{N}$, we know that $\Xi_{N} = \ag \Delta_{N}(\ag X^{N})$ follows a normal law $\mathcal{N}(\bm{0}, \bm{I}_{\uxi \Omega_{N}})$.
We denote $p_{\xi}^{N}$ its standard Gaussian density with respect to the Lebesgue measure on $\uxi \Omega^{N}$.
We also know that 
\[ \ag \Delta_{N}(\bg X^{N}) = (\ag \Delta_{N} \circ \ab F_N) (\ag X^{N})=  \ba G_N (\ag \Delta_{N}(\ag X^{N})) = \ba G_N (\Xi_{N}) \, . \]
Since $\ba G_N$ is linear,  the change of variable formula directly entails that $\ba G_N (\Xi_{N})$ admits on $\uxi \Omega_{N}$
\begin{eqnarray*} 
p^N_{\beta, \alpha}(\xi_N) = \vert \det(\ab G) \vert \: p^N_{\xi} \big [ {\xi_N}^T ({\ab G_{N}}^{T} \cdot \ab G_{N}) \, \xi_N \big ] 
\end{eqnarray*}
as density with respect to the Lebesgue measure.
Consider now $B$ a measurable set of $(\ux \Omega_{N}, \mathcal{B}(\ux \Omega_N))$, then we have
\begin{eqnarray*}
P^{N}_{\beta}(B) &=&  \int_{\ag \Delta_{N} (B)} p^N_{\beta, \alpha}(\xi_{N}) \, d\xi_N \, , \\
&=&  \int_{\ag \Delta_{N} (B)} \frac{p^N_{\beta, \alpha}(\xi_{N})}{p^N_\xi(\xi_N)} p^N_\xi(\xi_N)\, d\xi_N \, , \\
&=& \int_{ B} \frac{p^N_{\beta, \alpha}\big(\ag \Delta_{N}(X^{N})\big)}{p^N_\xi\big(\ag \Delta_{N}(X^{N})\big)} \, dP^{N}_{\alpha}(X^{N}) \, ,
\end{eqnarray*}
from which we immediately conclude.

\subsubsection{The Trace Class Operator}

The expression of the path-wise expression of the Radon-Nikodym extends to the infinite-dimensional representation of $\ag X$ and $\bg X$. 
This extension involves technical analysis on the infinite-dimensional Hilbert space $l^{2}(\mathbb{R})$. We have shown in Proposition \ref{prop:LiftOperator} that the application $\ab G$ was bounded for the spectral norm. 
This property will allow us to prove that the infinite-dimensional Gauss-Markov processes are absolutely continuous and to compute the Radon-Nikodym derivative of one measure with respect to the other one, proof that will essentially be comparable to the proof of Feldman-Hajek theorem (see e.g.\cite[Theorem 2.23]{DaPrato:1992}). 

We have for any $\xi$ in $l^{2}(\mathbb{R})$, the inequality 
\[\Vert \ab G(\xi) \Vert_{2} \leq \Vert \ab G \Vert_{2} \cdot \Vert  \xi \Vert_{2}\]
implying that $\ab G$ maps $l^{2}(\mathbb{R})$ into $l^{2}(\mathbb{R})$. We can then define the adjoint operator $\ab G^{T}$ from $l^{2}(\mathbb{R})$ to $l^{2}(\mathbb{R})$, which is given by:
\begin{equation}
\forall \quad \xi, \eta \in  l^{2}(\mathbb{R}) \, , \quad \left( \ab G^{T} ( \eta) ,  \xi \right) = \left( \eta , \ab G (\xi) \right) \, .
\nonumber
\end{equation}
Let us now consider the self-adjoint operator $\ab S = \ab G^{T} \circ \ab G: l^{2}(\mathbb{R}) \rightarrow l^{2}(\mathbb{R})$. This operator is the infinite dimensional counterpart of the matrix $\ab S_N$.
  
\begin{lemma}\label{lemCoefGG}
	Consider the coefficients of the matrix representation of $\ab S$ in the natural basis $e_{n,k}$ of $l^{2}(\mathbb{R})$ given as $\ab S^{n,k}_{p,q} = \left( e_{n,k}, \ab S(e_{p,q}) \right)$, we have 
	\begin{equation}\label{eq:coefExprGG}
	\ab S^{n,k}_{p,q}  
	=
	\int_{0}^{1} \left( \ag \phi_{n,k}(t) + \frac{\big( \alpha(t) - \beta(t) \big)}{\sqrt{\Gamma(t)}} \ag \psi_{n,k}(t) \right) \left( \ag \phi_{p,q}(t) + \frac{\big( \alpha(t) - \beta(t) \big)}{\sqrt{\Gamma(t)}} \ag \psi_{p,q}(t) \right) \, dt \, .
	\end{equation}	
\end{lemma}
  
\begin{proof}  
Assume that $\max(n,p)\leq N$ and that $(n,k)$ and $(p,q) \in \I_N$. With the notations used previously with $U$ an open neighbourhood of $[0,1]$, we have:
\begin{eqnarray}\label{eq:infiniteCoefExprGG}
\ab S^{n,k}_{p,q}
&=&
\Big( \ab G(e_{n,k}) , \ab G(e_{p,q}) \Big) \nonumber\\
&= &
\sum_{(i,j) \in \I} \int_{U} \bg \delta_{i,j}(t) \ag \psi_{n,k}(t) \, dt \, \int_{U} \bg \delta_{i,j}(t)  \ag \psi_{p,q} (t) \, dt
\end{eqnarray}
Since we have by definition $\ag \psi_{n,k} = \ag \mathcal{K}[\ag \phi_{n,k}]$, $\bg \delta_{i,j} = \bg \mathcal{D}[\bg \phi_{i,j}]$ and $\ag \mathcal{K}^{-1} = \ag \mathcal{D}^{*}$ and $\bg \mathcal{D}^{-1} = \bg \mathcal{K}^{*}$ on the space  $D'_{0}(U)$, we have
\begin{eqnarray}
 \bg \phi_{i,j}(t) &=&   \bg \mathcal{K}^{*}[ \bg \delta_{i,j}](t) = f_{\beta}(t) \int_{U} g_{\beta}(s)  \, \bg \delta_{i,j}(s) \, ds \label{eq:eqInt1} \, , \\
 \ag \phi_{n,k}(t) &=& \ag \mathcal{D}^{*}[\ag \psi_{n,k}](t)= \frac{1}{f_{\alpha}(t)} \frac{d}{dt} \left( \frac{\ag \psi_{n,k}(t)}{g_{\alpha}(t)}\right)  \label{eq:eqInt2} \, .
\end{eqnarray}
From there using \eqref{eq:eqInt1}, we can write the integration by parts formula in the sense of the generalized functions to get
 \begin{eqnarray}
\int_{U} \bg \delta_{i,j}(t) \ag \psi_{n,k}(t) \, dt
&=&
\int_{U} \frac{1}{f_{\beta}(t)} \frac{d}{dt} \left( \frac{\ag \psi_{n,k}(t)}{g_{\beta}(t)} \right) \, \bg \phi_{i,j}(t) \, dt \\
&=&
\int_{U} \bg \mathcal{D}^{*}[\ag \psi_{n,k}](t) \, \bg \phi_{i,j}(t) \, dt \, .
\nonumber
\end{eqnarray}
We now compute using \eqref{eq:eqInt2}
\begin{eqnarray} \label{eq:derivExpr}
\bg \mathcal{D}^{*}[\ag \psi_{n,k}](t)
&=& 
\frac{1}{f_{\beta}(t)} \frac{d}{dt} \left( \frac{\ag \psi_{n,k}(t)}{g_{\alpha}(t)} \frac{g_{\alpha}(t)}{g_{\beta}(t)} \right) \nonumber\\
&=&
\frac{g_{\alpha}(t) f(t)}{g_{\beta}(t) f_{\beta}(t)} \ag \phi(t) - \frac{d}{dt} \left( \frac{g_{\alpha}(t)}{g_{\beta}(t)} \right) \frac{ \ag \psi_{n,k}(t)}{g_{\beta}(t) f_{\beta}(t)} 
\end{eqnarray}
Specifying $g_{\alpha}$, $g_{\beta}$, $f_{\alpha}$ and $f_{\beta}$, and recalling the relations 
\begin{eqnarray}
g_{\alpha}(t) f_{\alpha}(t) = g_{\beta}(t) f_{\beta}(t) = \sqrt{\Gamma(t)} \quad \mathrm{and} \quad \frac{d}{dt} \left( \frac{g_{\alpha}(t)}{g_{\beta}(t)} \right) = \big( \alpha(t) - \beta(t) \big)  \frac{g_{\alpha}(t)}{g_{\beta}(t)} \, ,
\nonumber
\end{eqnarray}
we rewrite the function \eqref{eq:derivExpr} in $L^{2}[0,1]$ as
\begin{eqnarray}
\bg \mathcal{D}^{*}[\ag \psi_{n,k}](t) = \frac{1}{f_{\beta}(t)} \frac{d}{dt} \left( \frac{\ag \psi_{n,k}(t)}{g_{\beta}(t)} \right) = \ag \phi_{n,k}(t) + \frac{\big( \alpha(t) - \beta(t) \big)}{\sqrt{\Gamma(t)}} \ag \psi_{n,k}(t) \, .
\nonumber
\end{eqnarray}
Now, since the family $\bg \phi_{n,k}$ forms a complete orthonormal system of $L^{2}[0,1]$ and is zero outside $[0,1]$, by Parseval identity, expression \eqref{eq:infiniteCoefExprGG} can be written as the scalar product:
\begin{eqnarray}
\ab S^{n,k}_{p,q}  
= 
\int_{0}^{1} \left( \ag \phi_{n,k}(t) + \frac{\big( \alpha(t) - \beta(t) \big)}{\sqrt{\Gamma(t)}} \ag \psi_{n,k}(t) \right) \left( \ag \phi_{p,q}(t) + \frac{\big( \alpha(t) - \beta(t) \big)}{\sqrt{\Gamma(t)}} \ag \psi_{p,q}(t) \right) \, dt \, ,
\nonumber
\end{eqnarray}
which ends the proof of the lemma.
\end{proof}  

Notice that we can further simplify expression \eqref{eq:coefExprGG}
\begin{eqnarray}
\lefteqn{
\ab S^{n,k}_{p,q}  
=  \delta^{n,k}_{p,q} + \int_{0}^{1}  \frac{\big( \alpha(t) - \beta(t) \big)^{2}}{\Gamma(t)}  \big( \ag \psi_{n,k}(t) \ag \psi_{p,q}(t) \big) \, dt +}
\nonumber\\
&&
\hspace{100pt} \int_{0}^{1}  \frac{\big( \alpha(t) - \beta(t) \big)}{\sqrt{\Gamma(t)}} \big( \ag \phi_{n,k}(t) \, \ag \psi_{p,q}(t)  +\ag \phi_{p,q}(t) \, \ag \psi_{n,k}(t) \big)\, dt 
\, .
\end{eqnarray}
We are now in a position to show that the operator $\ab S - Id$ can be seen as the limit of the finite dimensional operator $\ab S_{N} -Id_{\uxi \Omega_{N}}$, in the following sense:

\begin{theorem} \label{thCompactGG}
The operator $\ab S - Id: l^{2}(\mathbb{R}) \rightarrow l^{2}(\mathbb{R})$ is a trace class operator, whose trace is given by:
\begin{equation}
\mathrm{Tr}\left(\ab S - Id \right) =  \int_{0}^{1} \left( \alpha(t) - \beta(t) \right) \, dt
+
\int_{0}^{1} \frac{h_{\alpha}(t)}{{f_{\alpha}(t)}^{2}}  \big( \alpha(t) - \beta(t) \big)^{2} \, dt 
\end{equation}
\end{theorem}

We prove this essential point in Appendix \ref{append:TraceClass}.
The proof consists in showing that the operator  $\ab S - Id$ is isometric to a Hilbert-Schmidt operator whose trace can be computed straightforwardly.


\subsubsection{The Girsanov Theorem}

We now proceed to prove the Girsanov theorem by extending the domain of the quadratic form associated with $\ab S  - Id$ to the space $\uxi \Omega'$, which can only be done in law.

\begin{theorem} \label{RadonTheorem}
In the infinite dimensional case, the Radon-Nikodym derivative of $P_{\beta} = P_{{\bg X}^{-1}}$ with respect to $P_{\alpha} = P_{{\ag X}^{-1}}$ reads
\begin{equation}\label{eq:GirsaStrato}
\frac{d P_{\beta}(\omega)}{d P_{\alpha}(\omega)} = \exp{ \left(- \frac{1}{2} \left( \int_{0}^{1} \beta(t) - \alpha(t) \, dt + {\Xi(\omega)}^{T} \left(\ab S  - Id_{\uxi \Omega'}\right) \Xi(\omega)\right) \right)} \, .
\end{equation}
which in terms of It\^o stochastic integral reads:
\begin{eqnarray}\label{eq:GirsaIto}
\lefteqn{
\frac{d P_{\beta}(\omega)}{d P_{\alpha}(\omega)}  = \exp\left(  
\int_{0}^{1} \frac{\beta(t) - \alpha(t)}{{f_{\alpha}}^{2}(t)} \, \frac{\ag X_{t}(\omega)}{g_{\alpha}(t)}  d \! \left( \frac{\ag X_{t}(\omega)}{g_{\alpha}(t)} \right) - \right.}
\nonumber\\
&& \qquad \qquad \qquad \qquad \left. \frac{1}{2} \int_{0}^{1} \frac{\big( \beta(t)-\alpha(t)\big)^{2} }{{f_{\alpha}}^{2}(t)} \, \left( \frac{\ag X_{t}(\omega)}{g_{\alpha}(t)} \right)^{2} \, dt 
\right) \, .
\nonumber
\end{eqnarray}
\end{theorem}

\noindent In order to demonstrate the Girsanov theorem from our geometrical point of view, we need to establish the following result:

\begin{lemma}\label{lemGGextension}
The positive definite quadratic form on $l^{2}(\mathbb{R}) \times l^{2}(\mathbb{R})$ associated with operator $\ab S - Id: l^{2}(\mathbb{R}) \rightarrow l^{2}(\mathbb{R})$ is well-defined on $\uxi \Omega'$.
Moreover for all $\uxi \Omega'$, 
\begin{eqnarray} \label{eq:StratoExp}
\lefteqn{
\big(\xi, (\ab S - Id_{\uxi \Omega'})(\xi \big) = }
\nonumber\\
&&
2 \, \int_{0}^{1} \frac{\alpha(t) - \beta(t)}{{f_{\alpha}}^{2}(t)} \, \frac{\ag X_{t}(\xi)}{g_{\alpha}(t)} \circ d \left( \frac{\ag X_{t}(\xi)}{g_{\alpha}(t)} \right) + \int_{0}^{1} \frac{\big(\alpha(t)- \beta(t)\big)^{2} }{{f_{\alpha}}^{2}(t)} \, \left( \frac{\ag X_{t}(\xi)}{g_{\alpha}(t)} \right)^{2} \, dt \, , 
\nonumber
\end{eqnarray}
where $\ag X_{t}(\xi) = \ag \Phi(\xi)$ and $\circ$ refers to the Stratonovich integral and the equality is true in law. . 
\end{lemma}

\begin{proof}[Proof of Theorem \ref{RadonTheorem}]
We start by writing the finite-dimensional Radon-Nikodym derivative 
\begin{eqnarray}
\lefteqn{
\frac{d P_{\beta}(\omega)}{d P_{\alpha}(\omega)} = }
\nonumber\\
&& \ab J_{N} \cdot \lim_{N \to \infty} \exp{ \left(- \frac{1}{2} \left( {\Xi_{N}(\omega)}^{T} \left(\ab S_{N} - Id_{\uxi \Omega_{N}}\right) \Xi_{N}(\omega)\right) \right)} \, .
\end{eqnarray}
By Proposition \ref{prop:LiftOperator}, we have 
\begin{eqnarray}
\ab J = \lim_{N \to \infty} \ab J_{N} = \frac{1}{2} \int_{0}^{1} (\alpha(t) - \beta(t)) \, dt \, .
\nonumber
\end{eqnarray}
If, as usual, $\Xi$ denotes a recursively indexed infinite-dimensional vectors of independent variables with law $\mathcal{N}(0,1)$ and $\Xi_{N} = \uxi P_{N} \circ \Xi$, writing $\xi_{n,k} = \Xi_{n,k}(\omega)$, we have
\begin{eqnarray}
\lefteqn{
{\Xi_{N}(\omega)}^{T} \left( \ab S_{N} - Id_{\uxi \Omega_{N}}\right) \Xi_{N}(\omega)
=}
\nonumber\\
&&
 \sum_{n = 0}^{N} \hspace{17pt} \sum_{p = 0}^{N}  \hspace{5pt} \sum_{ \hspace{5pt} 0 \leq k < 2^{n\!-\!1} }  \hspace{-5pt}  \sum_{ \hspace{5pt} 0 \leq q < 2^{p\!-\!1} }  \xi_{n,k} \left[\ab S_{N} -Id_{\uxi \Omega_{N}} \right]^{n,k}_{p,q} \xi_{p,q} \, .
 \nonumber
\end{eqnarray}
We know that $\xi$ is almost surely in $\uxi \Omega'$ and, by Lemma \ref{lemGGextension}, we also know that  on $\uxi \Omega' \times \uxi \Omega'$:
\begin{equation*}
\lim_{N \to \infty} \big(\xi, (\ab S_{N} - Id_{\uxi \Omega_{N}})(\xi) \big) = \big(\xi, (\ab S - Id_{\uxi \Omega'})(\xi) \big) \, ,
\end{equation*}
so that we can effectively write the infinite-dimensional Radon-Nikodym derivative on $\uxi \Omega'$ as the point-wise limit of the finite dimensional ones on $\uxi \Omega_{N}$ through the projectors $\uxi P_{N}$:
\begin{equation*}
\frac{d P_{\beta}}{d P_{\alpha}} (\omega)= \lim_{N \to \infty} \frac{d P^{N}_{\beta}}{d P^{N}_{\alpha}}(\omega)  \,
\end{equation*}
which directly yields formula \eqref{eq:GirsaStrato}

The derivation of Girsanov formula \eqref{eq:GirsaIto} from \eqref{eq:GirsaStrato} comes from the relationship between Stratonovich and It\^o formula, since the quadratic variation of $\ag X_t / g_{\alpha}(t)$ and $(\alpha(t)-\beta(t))/ f_{\alpha}^2(t) \times \ag X_t / g_{\alpha}(t)$ reads
\begin{eqnarray}
\int_{0}^{1} E\left( \int_{0}^{t} f_{\alpha}(s) \, dW_s, \frac{\alpha(t) - \beta(t)}{{f_{\alpha}}^{2}(t)} \int_{0}^{t} f_{\alpha}(s) \, dW_s \right) = \int_{0}^{1} \big( \alpha(t) - \beta(t) \big) \, dt \, .
\nonumber
\end{eqnarray}
Therefore, the expression for the Radon-Nikodym derivative in Lemma \ref{lemGGextension} can be written in terms of It\^o integrals
\begin{eqnarray} \label{eq:ItoExp}
\lefteqn{
\big(\xi, (\ab S - Id_{\uxi \Omega'})(\xi \big) =  \int_{0}^{1} \big( \alpha(t) - \beta(t) \big) \, dt}
\nonumber\\
&&
2 \, \int_{0}^{1} \frac{\alpha(t) - \beta(t)}{{f}^{2}(t)} \, \frac{\ag X_{t}(\xi)}{g_{\alpha}(t)}  d \left( \frac{\ag X_{t}(\xi)}{g_{\alpha}(t)} \right) + \int_{0}^{1} \frac{\big(\alpha(t)- \beta(t)\big)^{2} }{{f}^{2}(t)} \, \left( \frac{\ag X_{t}(\xi)}{g_{\alpha}(t)} \right)^{2} \, dt \, ,
\nonumber
\end{eqnarray}
and the Radon-Nikodym derivative:
\begin{eqnarray}
\lefteqn{
\frac{d P_{\beta}}{d P_{\alpha}}(\omega)  = \exp \left(  
\int_{0}^{1} \frac{\beta(t) - \alpha(t)}{{f_{\alpha}}^{2}(t)} \, \frac{\ag X_{t}(\omega)}{g_{\alpha}(t)}  d \! \left( \frac{\ag X_{t}(\omega)}{g_{\alpha}(t)} \right) \right. }
\nonumber\\
&& \qquad \qquad \qquad \qquad  \left. -\frac{1}{2} \int_{0}^{1} \frac{\big( \beta(t)-\alpha(t)\big)^{2} }{{f_{\alpha}}^{2}(t)} \, \left( \frac{\ag X_{t}(\omega)}{g_{\alpha}(t)} \right)^{2} \, dt 
\right) \, .
\nonumber
\end{eqnarray}
\end{proof}
Observe that if $\alpha(t) = 0$, we recover the familiar expression
\begin{eqnarray*}
\frac{d P_{\beta}}{d P_{\alpha}}(\omega) = \exp{\left( \int_{0}^{1} \frac{\beta(t)}{\sqrt{\Gamma(t)}} W_{t}(\omega) \, dW_{t}(\omega)  -\frac{1}{2} \int_{0}^{1} \frac{\beta(t)^{2}}{\Gamma(t)} W(\omega)^{2}_{t} \, dt \right)} \, .\
\end{eqnarray*}

\section*{Conclusion and Perspectives}

The discrete construction we present displays both analytical and numerical interests for further applications. 
From the analysis viewpoint, three Haar-like properties make our decomposition particularly suitable for certain numerical computations: i) all basis elements have compact support on an open interval that has the structure of dyadic rational endpoints; ii) these intervals are nested and become smaller for larger indices of the basis element, and iii) for any interval endpoint, only a finite number of basis elements is nonzero at that point. 
Thus the expansion in our basis, when evaluated at an interval endpoint (e.g. dyadic rational), terminates in a finite number of steps. 
Moreover, the very nature of the construction based on an increasingly refined description of the sample paths paves the way to coarse-graining approaches similar to wavelet decompositions in signal processing. 
In view of this, our framework offers promising applications: 

\paragraph{Dichotomic Search of First Hitting-times}
The first application we envisage concerns the problem of first-hitting times. 
Because of its manifold applications, finding the time when a process first exits a given region, is a central question of stochastic calculus.
However, closed-form theoretical results are scarce and one often has to resort to numerical algorithms~\cite{Ricciardi:1987}.
In this regard, the multi-resolution property suggests an exact scheme to simulate sample paths of a Gaussian Markov process $\bX$ in an iterative ``top-down'' fashion. 
Assume the intervals are dyadic rational, and that we have a conditional knowledge of a sample path on the dyadic points of $D_{N} = \lbrace k2^{-N} \vert 0 \leq k \leq 2^{N}\rbrace$, one can decide to further the simulation of this sample path at any time $t$ in $D_{N+1}$ by drawing a point according to the conditional law of $\bX_{t}$ knowing ${\lbrace \bX_t \rbrace}_{t \in D_{N}} $, which is simply expressed in the framework of our construction. 
This property can be used to great advantage in numerical computations such as dichotomic search algorithms for first passage times: 
the key element is to find an estimate of the true conditional probability that a hitting time has occurred when knowing the value of the process at two given time, one in the past and on in the future.
With such an estimate,  an efficient strategy to look for passage times consists in refining the sample path when and only when its trajectory is estimated likely to actually cross the barrier. 
Thus the sample path of the process is represented at poor temporal resolution when far from the boundary and at increasingly higher resolution closer to the boundary.
Such an algorithmic principle achieves a high level of precision in the computation of the first hitting time, while demanding far less operation than usual stochastic Runge-Kutta scheme. 
Note that this approach has been successfully implemented for the one-dimensional case~\cite{Taillefumier:2010} and at stake is now to generalize to the computation of exit times in any dimension and for general smooth sets~\cite{Gobet2000167,Baldi:2000,Baldi:2002}. \\


\paragraph{Gaussian Deformation Modes in Nonlinear Diffusions}
The present study is developed for Gauss-Markov systems. However, many models arising in applied science present nonlinearities, and in that case, the construction based on a sum of Gaussian random variable will not generalize. However, the Gaussian case treated here can nevertheless be applied to perturbation of nonlinear differential equations will small noise. 
Let $\bm{F}:\R \times\R^d \mapsto \R^d$ be a nonlinear time-varying vector field and let us assume that $\bX_0(t)$ is a stable (attractive) solution of the dynamical system:
\[\frac{d\bm{X}}{dt} = \bm{F(t,X)}.\]
This function $\bX_0$ can for instance be a fixed point (in which case it is a constant), a cycle (in which case it is periodic) or a general attractive orbit of the system.  In the deterministic case, any solution having its initial condition in a given neighbourhood $\mathcal{B}$ in $\R\times \R^d$ of the solution will asymptotically converge towards the solution, and therefore perturbations of the solutions are bounded. Let us now consider that the system is subject to a small amount of noise and define $\bm{Y}\in \R^d$ the solution of the stochastic nonlinear differential equation:
\[d\bm{Y}_t = \bm{F}(t,\bm{Y}_t)\,dt + \varepsilon \sqrt{\bm{\Gamma}(t, Y_t)} d\bm{W}_t\]
Assuming that the noise is small (i.e. $\varepsilon$ is a small parameter), because of the attractivity of the solution $\bm{X}_0(t)$, the function $\bm{Y}(t)$ will remain very close to $\bm{X}_0(t)$ (at least in a bounded time interval). In this region, we define $\varepsilon \bm{Z}_t=\bm{Y}_t-\bm{X}_0(t)$. This stochastic variable is solution of the equation:
\begin{align*}
	d\bm{Z}_t &= \frac{1}{\varepsilon} \left(\bm{F}(t,\bm{X}_0(t) + \varepsilon\bm{Z}_t) - \bm{F}(t,\bm{X}_0(t) + \varepsilon \sqrt{\bm{\Gamma}(t, \bm{X}_0(t) + \varepsilon \bm{Z}_t)} d\bm{W}_t \right)\\ 
	& = (\nabla_x \bm{F})(t,\bm{X}_0(t)) \bm{Z}_t + \sqrt{\bm{\Gamma}(t, \bm{X}_0(t))} d\bm{W}_t + O(\varepsilon)
\end{align*}
The solution at the first order in $\varepsilon$ is therefore the multidimensional Gaussian process with non-constant coefficients:
\[d\bm{Z}_t = (\nabla_x \bm{F})(t,\bm{X}_0(t)) \bm{Z}_t +  \sqrt{\bm{\Gamma}(t, \bm{X}_0(t))} d\bm{W}_t \, .\]
and our theory describes the solutions in a multi-resolution framework. 
Notice that, in that perspective, our basis functions $\bpsi_{n,k}$ can be seen as increasingly finer modes of deformation of a deterministic trajectory.
This approach appears particularly relevant to the theory of weakly interconnected neural oscillators in computational neuroscience
~\cite{Izhikevich:2007}. 
Indeed, one of the most popular approach of this field, the phase model theory, formally consists in studying how perturbations are integrated in the neighborhood of an attracting cycle~\cite{Ermentrout:86,Ermentrout:95}.

\bigskip
The present paper proposes a very convenient way of defining finite-dimensional approximations of Gauss-Markov processes that are optimal in the sense of being minimizers of the Dirichlet energy, and that is consistent with standard concepts of stochastic calculus. 
It also sheds new light on the structure of the space of Gauss-Markov processes, by the study of the operators transforming a process into another, which can be seen as an infinite-dimensional rotation in the space of coefficients. 
The exploration of this new view might help further uncovering the geometric structure of this space.
 All these instances are exemplar of how our multi-resolution description of Gauss-Markov processes offers a simple yet rigorous tool to broach open problems, promising original applications in theoretical and applied sciences. 

\section*{Acknowledgements}
We wish to thank Prof. Marcelo Magnasco for many illuminating discussions. This work was partially supported by NSF grant EF-0928723 and ERC grant NERVI-227747.

\appendix
\section{Formulae of the Basis for the Integrated Wiener Proces}\label{append:IP}

In the case of the primitive of the Wiener process, straightforward linear algebra computations leads to the two basis of functions $\left((\psi_{n,k})_{1,1},(\psi_{n,k})_{2,1}\right)$ and $\left((\psi_{n,k})_{1,2},(\psi_{n,k})_{2,2}\right)$ having the expressions:

\begin{eqnarray*}
(\psi_{n,k})_{1,1}(t)
&=&
\left\{
\begin{array}{ccc}
 &\displaystyle  (\sigma_{n,k})_{1,1} \left( \frac{t-l_{n,k}}{m_{n,k} - l_{n,k}} \right)^{2} \left( 1+2 \frac{m_{n,k} - t}{m_{n,k} -l_{n,k}}\right)  \, , & l_{n,k}\leq t \leq m_{n,k}\\
 &\displaystyle (\sigma_{n,k})_{1,1} \left( \frac{r_{n,k}-t}{r_{n,k} - m_{n,k}} \right)^{2} \left( 1+2 \frac{t-m_{n,k} }{r_{n,k} -m_{n,k}}\right)   \, , & m_{n,k} \leq t \leq r_{n,k}\\
\end{array}
\right.
\, ,
\\
\\
(\psi_{n,k})_{2,1}(t)
&=&
\left\{
\begin{array}{ccc}
& \displaystyle  \;(\sigma_{n,k})_{1,1} \, 6 \frac{(t-l_{n,k})(m_{n,k}-t)}{(m_{n,k} - l_{n,k})^{3}}   \, , & l_{n,k}\leq t \leq m_{n,k}\\
 &   \displaystyle - \; (\sigma_{n,k})_{1,1}\, 6 \frac{(r_{n,k}-t)(t-m_{n,k})}{(r_{n,k} - m_{n,k})^{3}}  \, ,& m_{n,k} \leq t \leq r_{n,k}\\
\end{array}
\right.
\, ,
\\
\\
(\psi_{n,k})_{1,2}(t)
&=&
\left\{
\begin{array}{ccc}
 & - (\sigma_{n,k})_{2,2} \left( m_{n,k} - t \right) \displaystyle \left( \frac{t-l_{n,k}}{m_{n,k} - l_{n,k}} \right)^{2}  \, , & l_{n,k}\leq t \leq m_{n,k}\\
 &(\sigma_{n,k})_{2,2} \left( t-m_{n,k} \right)   \displaystyle \left( \frac{r_{n,k}-t}{r_{n,k} - m_{n,k}} \right)^{2} \, , & m_{n,k} \leq t \leq r_{n,k}\\
\end{array}
\right.
\, ,
\\
\\
(\psi_{n,k})_{2,2}(t)
&=&
\left\{
\begin{array}{ccc}
 & \displaystyle (\sigma_{n,k})_{2,2}  \left( \frac{t-l_{n,k}}{m_{n,k} - l_{n,k}} \right)^{2} \left( 1-2 \frac{m_{n,k} - t}{t -l_{n,k}}\right)  \, , & l_{n,k}\leq t \leq m_{n,k}\\
 & \displaystyle (\sigma_{n,k})_{2,2} \left( \frac{r_{n,k}-t}{r_{n,k} - m_{n,k}} \right)^{2} \left( 1-2 \frac{t-m_{n,k} }{r_{n,k} -t}\right)\, ,& m_{n,k} \leq t \leq r_{n,k}\\
\end{array}
\right.
\, ,
\end{eqnarray*}
where 
\[
(\sigma_{n,k})_{1,1} = \sqrt{
\frac{1}{196}(r_{n,k} -l_{n,k})^{3}} \qquad \text{and} \qquad (\sigma_{n,k})_{2,2}= \sqrt{\frac{1}{32}(r_{n,k} -l_{n,k})} \,.
\]
are the diagonal components of the (diagonal) matrix $\sigma_{n,k}$
As expected, we notice that the differential structure of the process is conserved at any finite rank, since we have:
\begin{eqnarray*}
\frac{d}{dt}\Big((\psi_{n,k})_{1,1}(t)\Big) =  (\psi_{n,k})_{2,1}(t)\, , \quad  \frac{d}{dt}\Big((\psi_{n,k})_{1,2}(t)\Big) =  (\psi_{n,k})_{2,2}(t) \, .
\end{eqnarray*}

\section{Formulae of the Basis for the Doubly-Integrated Wiener Proces}\label{append:DIP}

For the doubly integrated Wiener process, the construction of the three dimensional process involves a family of thee $3$-dimensional functions, which constitutes the columns of a $3\times 3$ matrix that we denote $\psi$.
This basis has again simple expression when $m_{n,k}$ is the middle of the interval $[l_{n,k},r_{n,k}]$:

\begin{eqnarray*}
(\psi_{n,k})_{1,1}(t)
&=& \begin{cases} 
\frac{\sqrt{5}}{60 \,(r_{n,k}-l_{n,k})^{5/2}} (t-l_{n,k})^3\,(l_{n,k}^2-7\,l_{n,k}\,t+5\,l_{n,k}\,r_{n,k}-25 t\,r_{n,k}+16\,t^2+10\,r_{n,k}^2) \\ 
-\frac{\sqrt{5}}{60\,(r_{n,k}-l_{n,k})^{5/2}} (r_{n,k}-t)^3\,(r_{n,k}^2-7\,r_{n,k}\,t+5\,l_{n,k}\,r_{n,k}-25\,t\,l_{n,k}+16\,t^2+10\,l_{n,k}^2)
\end{cases}
\\
(\psi_{n,k})_{1,2}(t)
&=& \begin{cases} 
\frac{\sqrt{3}}{12\,(r_{n,k}-l_{n,k})^{5/2}} \, (l_{n,k}-t)^3\,(2\,r_{n,k}+l_{n,k}-3\,t)\,(r_{n,k}+l_{n,k}-2\,t) \\ 
\frac{\sqrt{3}}{12\,(r_{n,k}-l_{n,k})^{5/2}} \, (r_{n,k}-t)^3\,(r_{n,k}+2\,l_{n,k}-3\,t)\,(r_{n,k}+l_{n,k}-2\,t)
\end{cases}
\\
(\psi_{n,k})_{1,3}(t)
&=& \begin{cases} 
- \frac{1}{6\,(r_{n,k}-l_{n,k})^{5/2}} \, (t-l_{n,k})^3\,(2\,t-r_{n,k} -l_{n,k} r_{n,k}+l_{n,k}-3\,t)^2\,(r_{n,k}+l_{n,k}-2\,t) \\ 
- \frac{1}{6\,(r_{n,k}-l_{n,k})^{5/2}} \, (r_{n,k}-t)^3\,(2\,t-r_{n,k} -l_{n,k} r_{n,k}+l_{n,k} -3\,t)^2\,(r_{n,k}+l_{n,k}-2\,t)
\end{cases}
\\
(\psi_{n,k})_{2,1}(t)
&=& \begin{cases} 
\frac{\sqrt{1}}{6 \,(r_{n,k}-l_{n,k})^{5/2}} (t-l_{n,k})^2\,(3\,r_{n,k}+l_{n,k}-4*t)\,(r_{n,k}+l_{n,k}-2\,t) \\ 
-\frac{\sqrt{1}}{6 \,(r_{n,k}-l_{n,k})^{5/2}} (r_{n,k}-t)^2\,(3\,l_{n,k}+r_{n,k}-4*t)\,(r_{n,k}+l_{n,k}-2\,t)
\end{cases}
\\
(\psi_{n,k})_{2,2}(t)
&=& \begin{cases} 
-\frac{\sqrt{3}}{6\,(r_{n,k}-l_{n,k})^{5/2}} \, (t-l_{n,k})^2\,(3\,r_{n,k}+2\,l_{n,k}-5\,t)\,(r_{n,k}+2\,l_{n,k}-3\,t) \\ 
-\frac{\sqrt{3}}{6\,(r_{n,k}-l_{n,k})^{5/2}} \, (r_{n,k}-t)^2\,(3\,l_{n,k}+2\,r_{n,k}-5\,t)\,(l_{n,k}+2\,r_{n,k}-3\,t) \\ 
\end{cases}
\\
(\psi_{n,k})_{2,3}(t)
&=& \begin{cases} 
- \frac{1}{6\,(r_{n,k}-l_{n,k})^{5/2}} \, (t-l_{n,k})^3\,(2\,t-r_{n,k}-l_{n,k} r_{n,k}+l_{n,k}-3\,t)^2\,(r_{n,k}+l_{n,k}-2\,t) \\ 
- \frac{1}{6\,(r_{n,k}-l_{n,k})^{5/2}} \, (r_{n,k}-t)^3\,(2\,t-r_{n,k}-l_{n,k} r_{n,k}+l_{n,k}-3\,t)^2\,(r_{n,k}+l_{n,k}-2\,t)
\end{cases}
\\
(\psi_{n,k})_{3,1}(t)
&=& \begin{cases} 
\frac{\sqrt{5}}{3 \,(r_{n,k} -l_{n,k})^{5/2}} (t-l_{n,k})\,(4\,l_{n,k}^2-17\,l_{n,k}\,t+9\,l_{n,k}\,r_{n,k}-15\,t\,r_{n,k}+3\,r_{n,k}^2+16\,t^2) \\ 
\frac{\sqrt{5}}{3 \,(r_{n,k} -l_{n,k})^{5/2}} (r_{n,k}-t)\,(4\,r_{n,k}^2-17\,t\,r_{n,k}+9\,l_{n,k}\,r_{n,k}-15\,t\,l_{n,k}+3\,l_{n,k}^2+16\,t^2) 
\end{cases}
\\
(\psi_{n,k})_{3,2}(t)
&=& \begin{cases} 
\frac{\sqrt{3}}{(r-l)^{5/2}} (l-t)\,(r+l-2\,t)\,(r+4\,l-5\,t)\\
\frac{\sqrt{3}}{(r-l)^{5/2}} (r-t)\,(r+l-2\,t)\,(l+4\,r-5\,t)\\
\end{cases}
\\
(\psi_{n,k})_{3,3}(t)
&=& \begin{cases} 
\frac{1}{3(r-l)^{5/2}}\,(l-t)\,(19\, l^2-56\,l\,t+18\,l\,r-24\,t\,r+3\,r^2+40\,t^2)\\
\frac{1}{3(r-l)^{5/2}}\,(r-t)\,(19 \,r^2-56\,r\,t+18\,l\,r-24\,t\,l+3\,l^2+40\,t^2)
\end{cases}
\end{eqnarray*}
Notice again that the basis functions satisfy the relationships:
\[\frac{d}{dt}\Big((\psi_{n,k})_{i,j}(t)\Big) =  (\psi_{n,k})_{i+1,j}(t)\] 
for $i\in\{1,2\}$ and $j\in \{1,2,3\}$.
These functions also  form a tri-orthogonal basis of functions, which makes easy to simulate sample paths of the doubly integrated Wiener process, as show in Figure ~\ref{fig:DIPFig}.
\begin{figure}
	\centering
		\includegraphics[width=.7\textwidth]{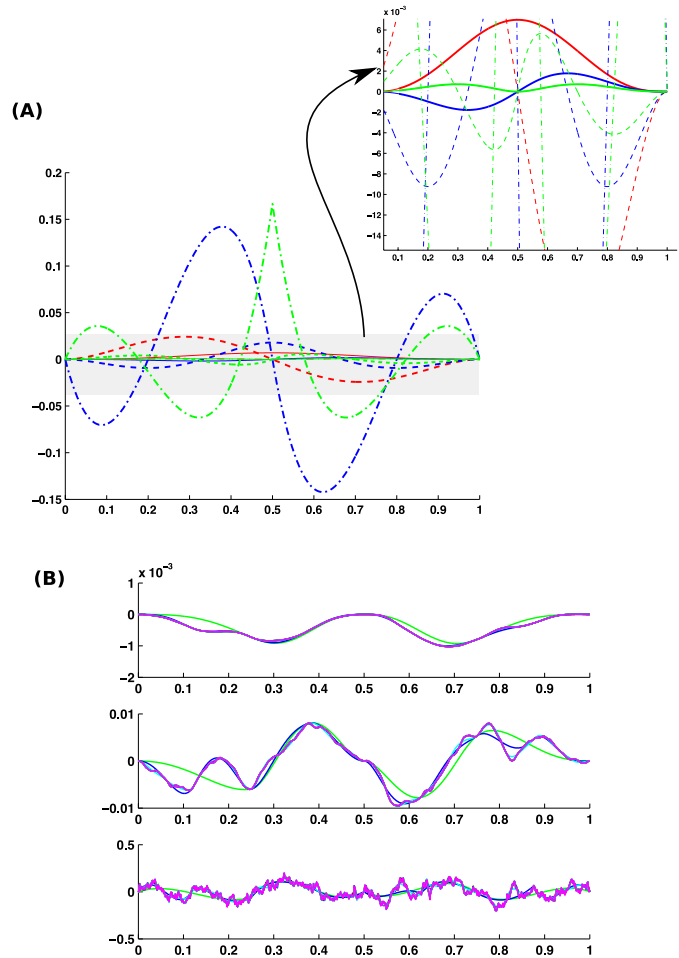}
	\caption{(A) Basis for the construction of the Doubly Integrated Wiener process ($d=3$) and (B) $10$-steps construction of the process. }
	\label{fig:DIPFig}
\end{figure}

\newpage

\section{Properties of the Lift Operators}\label{append:LiftOperator}
This appendix is devoted to the proofs of the properties of the lift operator enumerated in Proposition \ref{prop:LiftOperator}. The proposition is split into three lemmas for the sake of clarity.

\begin{lemma} \label{lemG}
The operator ${}_{ \alpha, \beta}G$ is a linear measurable bijection. Moreover, for every $N>0$, the function  ${}_{ \alpha, \beta}G_{N} = P_{N} \circ {}_{ \alpha, \beta}G \circ I_{N}: \uxi \Omega' _{N} \to \uxi \Omega _{N}$ is a finite-dimensional linear operator, whose matrix representation is triangular in the natural basis of $\uxi \Omega_{N}$ and whose eigenvalues ${}_{ \alpha, \beta}\nu_{n,k}$ are given by
\begin{equation}
{}_{ \alpha, \beta}\nu_{n,k} 
=
\frac{g_{\alpha}(m_{n,k})}{g_{\beta}(m_{n,k})} \frac{\bg M_{n,k}}{\ag M_{n,k}}
, \quad 0 \leq n \leq N, \quad 0\leq k < 2^{N-1} \, .
\nonumber
\end{equation}

Eventually, $\ab G$ is a bounded operator for the spectral norm with
\begin{eqnarray}
\big \Vert \ab G  \big \Vert_{2} = \sup_{n} \sup_{k} {}_{ \alpha, \beta}\nu_{n,k} \leq \frac{\sup g_{\alpha}}{\inf g_{\beta}} \frac{\sup f_{\beta}}{\inf f} \frac{\sup f^{2}_{\alpha}}{\inf f^{2}_{\beta}} < \infty \, ,
\nonumber
\end{eqnarray}
and the determinant of ${}_{ \alpha, \beta}G_{N}$ denoted $\ab J_{N}$ admits a limit when $N$ tends to infinity
\begin{equation}
\lim_{N \to \infty}  \det\left( {}_{ \alpha, \beta}G_{N} \right) =  \lim_{N \to \infty} \ab J_{N} = \exp{\left(\frac{1}{2}\left( \int_{0}^{1} (\alpha(t) - \beta(t)) \, dt \right) \right)} ={}_{ \alpha, \beta}J\, .
\nonumber
\end{equation}
\end{lemma}

\begin{proof}
	All these properties are deduced from the properties of the functions $\Delta$ and $\Psi$ derived previously.\\
	\emph{i}) ${}_{ \alpha, \beta}G=  \bg \Delta \circ \ag \Psi$ is a linear measurable bijection of $\uxi \Omega'$ as the composed application of two linear bijective measurable functions $\ag \Delta:\ux \Omega'  \rightarrow \uxi \Omega'$ and $\ag \Psi: \uxi \Omega' \rightarrow \ux \Omega'$.\\
\emph{ii}) Since we have the expressions of the matrices of the finite-dimensional linear transformations, it is easy to write the linear transformation of  ${}_{ \alpha, \beta}G_{N}$ on the natural basis $e_{n,k}$ as:
\begin{equation} \label{eq:AppGN} 
{}_{ \alpha, \beta}G_{N}\left(  \xi \right)_{n,k} 
=
\int_{U} \bg \delta_{n,k}(t) \big( \ag \Psi(\xi) \big)(t) \, dt
=
\sum_{(p,q)\in \I_N} \left( \int_{U} \bg \delta_{n,k}(t) \, \ag \psi_{p,q}(t) \, dt \right) \,  \xi_{p,q} \, ,
\end{equation}
leading to the coefficient expression
\begin{equation}
{}_{ \alpha, \beta}G^{n,k}_{p,q}
=
\int_{U} \bg \delta_{n,k}(t) \, \ag \psi_{p,q}(t) \, dt
=
\sum_{i,j \in \I_N} \bg \Delta^{n,k}_{i,j} \cdot \ag \Psi^{i,j}_{p,q}  
 \, . \nonumber
\end{equation}
where we have dropped the index $N$ since the expression of the coefficients do not depend on it.
We deduce from the form of the matrices $\bg \Delta_{N}$ and $\ag \Psi_{N}$, that the application ${}_{ \alpha, \beta}G_{N}$ has a matrix representation in the basis $e_{n,k}$ of the form
\begin{displaymath}
{}_{ \alpha, \beta}G_{N} 
=
\left[
\begin{array}{c|c|cc|cccc|c}
  \ab G^{0,0}_{0,0} &   &  &   &  &   &  & & \\
  \hline
  \ab G^{0,0}_{1,0}& \ab G^{1,0}_{1,0}  &  &   &  &   &  &  &\\
   \hline
  \ab G^{0,0}_{2,0}&  \ab G^{1,0}_{2,0} & \ab G^{2,0}_{2,0} &   &  &   &  & &\\
  \ab G^{0,0}_{2,1}&  \ab G^{1,0}_{2,1} &  & \ab G^{2,1}_{2,1}  &  &   &  & & \\
    \hline
  \ab G^{0,0}_{3,0}&  \ab G^{1,0}_{3,0} & \ab G^{2,0}_{3,0} &   & \ab G^{3,0}_{3,0}&   &  &  &\\
  \ab G^{0,0}_{3,1}&  \ab G^{1,0}_{3,1} &  \ab G^{2,0}_{3,1} &   &  & \ab G^{3,1}_{3,1}  &  & & \\
  \ab G^{0,0}_{3,2}&  \ab G^{1,0}_{3,2} &  &  \ab G^{2,1}_{3,2}  &  &   &  \ab G^{3,2}_{3,2} & & \\
  \ab G^{0,0}_{3,3}&  \ab G^{1,0}_{3,3} &   &  \ab G^{2,1}_{3,3}  &  &   &  &  \ab G^{3,3}_{3,3}& \\
  \hline
  \vdots &   &  &   &  &   &  & & \ddots
\end{array}
\right]
\end{displaymath}
where we only represent the non-zero terms.

The eigenvalues of the operator are therefore the diagonal elements $\ab G^{n,k}_{n,k}$, that are easily computed from the expressions of the general term of the matrix:
\begin{align*}
	\ab G^{n,k}_{n,k} &= \bg \Delta_{n,k}^{n,k} \ag \psi_{n,k}^{n,k}\\
	&= \frac{\bg M_{n,k}}{g_{\beta}(m_{n,k})} \; \frac{g_{\alpha}(m_{n,k})}{\ag M_{n,k}}\\
	&= \frac{g_{\alpha}(m_{n,k})}{g_{\beta}(m_{n,k})} \; \frac{\bg M_{n,k}}{\ag M_{n,k}}
\end{align*}
\noindent \emph{iii}) From the expression of $\ag M_{n,k} = \sqrt{(h_{\alpha}(r)-h_{\alpha}(m))(h_{\alpha}(m)-h_{\alpha}(l)) / (h_{\alpha}(r)-h_{\alpha}(l))}$, we deduce the inequalities
\begin{eqnarray}
\frac{\sup f}{\inf f^{2}_{\alpha}} 2^{n+1} \leq \ag M_{n,k} \leq \frac{\sup f}{\inf f^{2}_{\alpha}} 2^{n+1} \, ,
\end{eqnarray}
from which follows the given upper-bound to the singular values.\\
\emph{iv}) 
${}_{ \alpha, \beta}G_{N}$ is a finite-dimensional triangular linear matrix in the basis $\lbrace e_{n,k} \rbrace$.
Its determinant is simply given as the product
\begin{equation}
{}_{ \alpha, \beta}J_{N} = \prod_{n=0}^{N} \prod_{\hspace{5pt} 0 \leq k < 2^{n\!-\!1}}   \ab G^{n,k}_{n,k} =  \prod_{n=0}^{N} \prod_{\hspace{5pt} 0 \leq k < 2^{n\!-\!1}} {}_{ \alpha, \beta}\nu_{n,k} \, ,
\end{equation}
where we noticed that the eigenvalues ${}_{ \alpha,\beta}\nu_{n,k}$ are of the form
\begin{equation}
{}_{ \alpha, \beta}\nu_{n,k} = \frac{g_{\alpha}(m_{n,k})}{g_{\beta}(m_{n,k})} \frac{\bg M_{n,k}}{\ag M_{n,k}} \, .
\nonumber
\end{equation}
Since for every $N>0$, we have ${}_{ \alpha, \beta}G_{N} = {}_{ \scriptscriptstyle \zeta ,\beta \displaystyle}G_{N} \circ {}_{ \scriptscriptstyle \alpha,\zeta \displaystyle}G_{N}$ which entails ${}_{ \alpha, \beta}J_{N} = {}_{ \scriptscriptstyle \zeta,\beta \displaystyle}J_{N} \cdot \left({}_{ \scriptscriptstyle \zeta,\alpha \displaystyle}J_{N}\right)^{-1}$, it is enough to show that we have
\begin{equation}
\lim_{N \to \infty} {}_{ \alpha, 0}J_{N} =  \exp{ \frac{1}{2} \int_{0}^{1} \alpha(t) \, dt } \, .
\nonumber
\end{equation}
Now, writing for every $0 \leq s < t\leq 1$ the quantity
\begin{equation}
{}_{ \alpha} \mathcal{V}_{t,s} = \int_{s}^{t} \Gamma(u) \, e^{2 \int_{u}^{t} \alpha(v)\, dv } \, du \, .
\nonumber
\end{equation}
we have 
\begin{equation}
\left( \frac{g_{\alpha}(m_{n,k})}{\ag M_{n,k}} \right)^{2}= \frac{{}_{ \alpha} \mathcal{V}_{l_{n,k},m_{n,k}} \cdot {}_{ \alpha} \mathcal{V}_{m_{n,k},r_{n,k}}}{{}_{ \alpha} \mathcal{V}_{l_{n,k},r_{n,k}}} \, ,
\nonumber
\end{equation}
so that, ${}_{ \alpha, 0}J_{N}$ is a telescoping product that can be written
\begin{equation}
\left( {}_{ \alpha, 0}J_{N}\right)^{2} = \prod_{k=0}^{2^N} \frac{\ag \mathcal{V}_{k2^{-N},(k\!+\!1)2^{-N}} } {{}_{ 0} \mathcal{V}_{k2^{-N},(k\!+\!1)2^{-N}} } \, .
\nonumber
\end{equation}
If $\alpha$ is H\"{o}lder continuous, there exists $\delta>0$ and $C>0$ such that
\begin{equation}
 \sup_{0 \leq s,t \leq 1} \frac{\vert \alpha(t)-\alpha(s) \vert}{\vert t-s\vert^{\delta}} < C \, ,
 \nonumber
\end{equation}
and introducing for any $0 \leq s<t \leq 1$, the quantity $\mathcal{Q}_{t,s}$
\begin{equation}
\mathcal{Q}_{t,s} =  e^{- \int_{s}^{t} \alpha(v) \, dv  } \, \cdot \frac{{}_{ \alpha} \mathcal{V}_{t,s}}{{}_{ 0} \mathcal{V}_{t,s}}  
=
\left \vert \frac{\int_{s}^{t} \Gamma(u) \, e^{ \int_{u}^{t} \alpha(v) \, dv - \int_{s}^{u} \alpha(v) \, dv} \, du}{\int_{s}^{t} \Gamma(u) \, du} \right \vert  \, ,
\nonumber
\end{equation}
we have that $\underline{\mathcal{Q}}_{t,s} \leq \mathcal{Q}_{t,s} \leq \overline{\mathcal{Q}}_{t,s}$ with
\begin{eqnarray}
\underline{\mathcal{Q}}_{t,s} = \frac{\int_{s}^{t} \Gamma(u) \, e^{-\frac{C}{1+\delta} \big( (t-u)^{1+\delta}+ (u-s)^{1+\delta}\big)} \, du}{\int_{s}^{t} \Gamma(u) \, du}
\quad \mathrm{and} \quad
\overline{\mathcal{Q}}_{t,s} = \frac{\int_{s}^{t} \Gamma(u) \, e^{ \frac{C}{1+\delta} \big( (t-u)^{1+\delta}+ (u-s)^{1+\delta}\big)} \, du}{\int_{s}^{t} \Gamma(u) \, du} \nonumber \, .
\end{eqnarray}
After Taylor expanding the exponential in the preceding definitions, we have
\begin{eqnarray}
\underline{\mathcal{Q}}_{t,s} &\geq &1- \frac{2C (\sup_{0 \leq t \leq 1} \Gamma(t))}{\inf_{0 \leq t \leq 1} \Gamma(t)} \frac{(t-s)^{(1+\delta)}}{(1+\delta)(2+\delta)} + o\left( (t-s)^{(1+\delta)} \right)
\nonumber\\
\overline{\mathcal{Q}}_{t,s} &\leq& 1+ \frac{2C  (\sup_{0 \leq t \leq 1} \Gamma(t))}{\inf_{0 \leq t \leq 1} \Gamma(t)} \frac{(t-s)^{(1+\delta)}}{(1+\delta)(2+\delta)} + o\left( (t-s)^{(1+\delta)} \right) \, ,
\nonumber
\end{eqnarray}
Now, from
\begin{equation}
\prod_{k=0}^{2^N} \mathcal{Q}_{k2^{-N},(k\!+\!1)2^{-N}} = 1 + o\left( 2^{-N} \right) \, ,
\nonumber
\end{equation}
we can directly conclude that
\begin{equation}
\lim_{N \to \infty} {}_{ \alpha, 0}J_{N} 
= 
e^{  \frac{1}{2} \int_{0}^{1} \alpha(t) \, dt } \sqrt{ \lim_{N \to \infty} \prod_{k=0}^{2^N} \mathcal{Q}_{k2^{-N},(k\!+\!1)2^{-N}}}
=
e^{  \frac{1}{2} \int_{0}^{1} \alpha(t) \, dt }  \, .
\nonumber
\end{equation}
\end{proof}

Notice that if $\alpha = \beta$, ${}_{ \alpha, \alpha}G$ is the identity and ${}_{ \alpha, \alpha}J = 1$ as expected.

\noindent Similar properties are now proved for the process lift operator $\ab H$. 
\begin{lemma}\label{lemH}
The function $\ab H$ is a linear measurable bijection.\\
Moreover, for every $N>0$, the function  $\ab H_{N} = P_{N} \circ {}_{ \alpha, \beta}H \circ I_{N}: \ux \Omega' _{N} \to \ux \Omega _{N}$ is a finite-dimensional linear operator, whose matrix representation is triangular in the natural basis of $\ux \Omega_{N}$ and whose eigenvalues are given by
\begin{equation}
{}_{ \beta, \alpha} \nu_{n,k} 
=
( {}_{ \alpha, \beta} \nu_{n,k} )^{-1} 
=
\frac{g_{\beta}(m_{n,k})}{g_{\alpha}(m_{n,k})} \frac{\ag M_{n,k}}{\bg M_{n,k}}
 \, .
\nonumber
\end{equation}
Eventually, $\ab H$ is a bounded operator for the spectral norm with
\begin{eqnarray}
\big \Vert \ab H  \big \Vert_{2} = \big \Vert {}_{\beta,\alpha} G  \big \Vert_{2}  = \sup_{n} \sup_{k} {}_{ \beta, \alpha}\nu_{n,k} \leq \frac{\sup g_{\beta}}{\inf g_{\alpha}} \frac{\sup f}{\inf f_{\beta}} \frac{\sup f^{2}_{\beta}}{\inf f^{2}_{\alpha}} < \infty \, ,
\nonumber
\end{eqnarray}
and the determinant of ${}_{ \alpha, \beta}H_{N}$   admits a limit when $N$ tends to infinity
\begin{equation}
\lim_{N \to \infty}  \det \left( \ab H \right) =  \exp{\left(\frac{1}{2}\left( \int_{0}^{1} (\beta(t) - \alpha(t)) \, dt \right) \right)} = \ba J \, .
\nonumber
\end{equation}
\end{lemma}

\begin{proof}
\emph{i}) The function ${}_{ \alpha, \beta}H=  \bg \Psi \circ \ag \Delta$ is a linear measurable bijection of $\ux \Omega'$ onto $\ux \Omega'$, because $\ag \Delta:\ux \Omega'  \rightarrow \uxi \Omega'$ and $\ag \Psi: \uxi \Omega' \rightarrow \ux \Omega'$ are linear bijective measurable functions.\\
\emph{ii}) We write the linear transformation of  ${}_{ \alpha, \beta}H$ for $x$ in $\ux \Omega'$ as
\begin{eqnarray} \label{eq:AppF} 
\ab H[x](t) 
&=&
\sum_{(n,k)\in \I} \bg \psi_{n,k}(t) \int_{U} \ag \delta_{n,k}(s) \, x(s) \, ds 
\\
&=&
\int_{U} \left( \sum_{(n,k)\in \I}  \bg \psi_{n,k}(t)  \ag \delta_{n,k}(t) \right) \, x(t) \, dt
\nonumber
\end{eqnarray}
If we denote the class of $x$ in $\ux \Omega'$ by $x = \lbrace x_{i,j} \rbrace = \lbrace x(m_{i,j}) \rbrace$, $(i,j)\in \I_N$, we can write \eqref{eq:AppF} as
\begin{eqnarray} \label{eq:AppFN} 
\ab H_{N}(x)_{i,j}
=
\sum_{(k,l)\in \I_N}
\left( \sum_{(p,q)\in \I_N} \bg \Psi^{i,j}_{p,q} \cdot \ag \Delta^{p,q}_{k,l} \right) \, x_{k,l} 
 \, ,
\nonumber
\end{eqnarray}
from which we deduce the expression of the coefficients of the matrix $\ab H_{N}$
\begin{eqnarray}
\ab H^{i,j}_{k,l} = \sum_{(p,q)\in\I_N} \bg \Psi^{i,j}_{p,q} \cdot \ag \Delta^{p,q}_{k,l}
\end{eqnarray}
where as usual we drop the index $N$.
Because of the the form of the matrices $\ag \Delta_{N}$ and $\bg \Psi_{N}$, the matrix $\ab H_{N}$ in the basis $f_{i,j}$ has the following triangular form:
\begin{displaymath}
{}_{ \alpha, \beta}H_{N} 
=
\left[
\begin{array}{c|c|cc|cccc|c}
  \ab H^{0,0}_{0,0} &   &  &   &  &   &  & & \\
  \hline
  \ab H^{0,0}_{1,0}& \ab H^{1,0}_{1,0}  &  &   &  &   &  &  &\\
   \hline
  \ab H^{0,0}_{2,0}&  \ab H^{1,0}_{2,0} & \ab H^{2,0}_{2,0} &   &  &   &  & &\\
  \ab H^{0,0}_{2,1}&  \ab H^{1,0}_{2,1} &  & \ab H^{2,1}_{2,1}  &  &   &  & & \\
    \hline
  \ab H^{0,0}_{3,0}&  \ab H^{1,0}_{3,0} & \ab H^{2,0}_{3,0} &   & \ab H^{3,0}_{3,0}&   &  &  &\\
  \ab H^{0,0}_{3,1}&  \ab H^{1,0}_{3,1} &  \ab H^{2,0}_{3,1} &   &  & \ab H^{3,1}_{3,1}  &  & & \\
  \ab H^{0,0}_{3,2}&  \ab H^{1,0}_{3,2} &  & \ab H^{2,1}_{3,2}    &   &  &\ab H^{3,2}_{3,2} & & \\
  \ab H^{0,0}_{3,3}&  \ab H^{1,0}_{3,3} &  & \ab H^{2,1}_{3,3}    &   &  &  & \ab H^{3,3}_{3,3}& \\
  \hline
  \vdots &   &  &   &  &   &  & & \ddots
\end{array}
\right] \, .
\end{displaymath}
From the matrix representation $\ag \Delta$ and $\bg \Psi$, the diagonal terms of $\ab H$ read:
 \begin{eqnarray}
{}_{ \alpha, \beta}H^{i,j}_{i,j} = \bg \Psi^{i,j}_{i,j} \cdot \ag \Delta^{i,j}_{i,j} = \bg \psi_{i,j}(m_{i,j}) \frac{\ag M_{i,j}}{g_{\alpha}(m_{i,j})} = \frac{g_{\beta}(m_{i,j})}{g_{\alpha}(m_{i,j})} \frac{\ag M_{i,j}}{\bg M_{i,j}} = {\ba \nu_{i,j}} = {\ab \nu_{i,j}}^{-1} \, .
\nonumber
\end{eqnarray}
\emph{iii)} The upper-bound directly follows form the fact that $\ba \nu_{i,j} = (\ab \nu_{i,j})^{-1}$.\\
\emph{iv)} Since $\ba \nu_{i,j} = (\ab \nu_{i,j})^{-1}$, the value of the determinant of $\ab H_{N}$ is clearly the inverse of the determinant of $\ab G_{N}$, so that $\lim_{N \to \infty} \det(\ab H_{N}) = (\ab J)^{-1} = \ba J$.
\end{proof}

Note that lemma \ref{lemH} directly follows from the fact that $\ag \Psi$ and $\ag \Delta$ are inverse from each other and admit a triangular matrix representation.

\section{Construction and Coefficient Applications}\label{append:ConstructCoeff}

In this appendix we provide the proofs of the main properties used in the paper regarding the construction and the coefficient applications. 
\subsection{The Construction application}
We start by addressing the case of the construction application introduced in section \ref{app:ConsAppl}. 

We start by proving Theorem \ref{lem:convergence}
\begin{proposition}
For every $\bxi$ in $\uxi \Omega'$, $ \bPsi^{N}(\bxi)$ converges uniformly toward a continuous function in $C_{0}\big([0,1],\R^{d}\big)$. We will denote this function $\bPsi(\bxi)$, defined as:
\[
\bPsi: \begin{cases}
\uxi \Omega' &\longrightarrow C_0([0,1],\R^d)\\
\bxi &\longmapsto \sum_{(n,k)\in \I}\psi_{n,k}(t) \cdot  \bxi_{n,k}
\end{cases}\]
and this application will be referred to as the construction application. 
\end{proposition}

\begin{proof}
	For the sake of simplicity, we will denote for any functions $\bm{A}:[0,1] \to \R^{m \times d}$, the uniform norm as $\vert \bm{A} \vert_{\infty} = \sup_{0 \leq t \leq 1} \vert \bm{A}(t) \vert$, where $\vert \bm{A}(t) \vert = \sup_{0 \leq i <m} \left( \sum_{0}^{d-1} \vert A_{i,j}(t) \vert \right)$ is the operator norm induced by the uniform norms. We will also denote the $i$-th line of $\bm{A}$ by $l_{i}(\bm{A})$ (it is a $\R^{d}$-valued function) and	the $j$-th column of $\bm{A}$ by $c_{j}(\bm{A})$.
	
Let $\bxi \in \uxi\Omega'$ fixed. These coefficients induce a sequence of continuous functions $\bpsi^N(\bxi)$ through the action of the sequence of the partial construction applications. To prove that this sequence converges towards a continuous function, we show that it uniformly converges, which implies the result of the proposition using the fact that a uniform limit of continuous functions is a continuous function. Moreover, since the functions take values in $\R^d$ which is a complete space, we show that for any sequence of coefficients $\xi\in\uxi\Omega'$, the sequence of functions $\psi^N(t)$ constitute a Cauchy sequence for the uniform norm. 

By definition of $\uxi \Omega'$, for every $\bxi$ in $\uxi \Omega'$, there exist $\delta<1$ and $n_{\xi}$ such that, for every $n>n_{\xi}$, we have
\begin{equation*}
\sup_{0 \leq k < 2^{n-1}} \vert \bxi_{n,k} \vert < 2^{\frac{n\delta}{2}} \, .
\end{equation*}
which implies that for $N>n_{\xi}$, we have
\begin{align}
\nonumber\big \vert \bPsi^{N}(\bxi)(t) -   \bPsi^{n_{\xi}}(\bxi)(t) \big \vert
 & \leq
\displaystyle{\sum_{(n,k) \in \I_{N} \setminus \I_{n_{\xi}}} \vert  \bpsi_{n,k}(t) \cdot \bxi_{n,k} \vert} \\
\label{eq:proofupperbound}& \leq \displaystyle{\sum_{n=n_{\xi}}^{\infty} 2^{n\delta/2} \vert \bpsi_{n,k} \vert}.
\end{align}

We therefore need to upperbound the uniform norm of the function $\bpsi_{n,k}$. To this purpose, we use the definition of $\bpsi_{n,k}$ given by equation \eqref{eq:psinkKPhi}:
\begin{eqnarray*}
	\bpsi_{n,k}(t) &= g(t) \cdot \int_0^t f(s)\cdot \bPhi_{n,k}(s)\, ds \, .
\end{eqnarray*}
The coefficient in position $(i,j)$ of the integral term in the righthand side of the previous inequality can be written as a function of the lines and columns of $f$ and $\bPhi_{n,k}$, and can be upperbounded using the Cauchy-Schwarz inequality on $L^{2}\big([0,1],\R^{d}\big)$ as follows:
\begin{eqnarray*}
\left( \mathbbm{1}_{[0,t]} \cdot c_{i}(\bm{f}^{T}) , c_{j}( \bphi_{n,k}) \right)
&=&
 \int_{U} \mathbbm{1}_{[0,t] \cap S_{n,k}}(s) \big( l_{i}(\bm{f})(s) \cdot c_{j}( \bphi_{n,k})(s) \big) \, ds \\
 &\leq& 
 \Vert \mathbbm{1}_{[0,t] \cap S_{n,k}} \, l_{i}(\bm{f}) \Vert_2 \, \Vert c_{j}( \bphi_{n,k})\Vert_2 .
\end{eqnarray*}
Since the columns of $\bPhi_{n,k}$ form an orthogonal basis of functions for the standard scalar product in $L^2\big([0,1],\R^{d}\big)$ ( see Proposition \ref{orthoProp}), we have $\Vert c_{j}( \bphi_{n,k}) \Vert_2 = 1$. Moreover,since $\bm{f}$ is bounded continuous on $[0,1]$, we can define constants $K_{i} = \sup_{0 \leq t \leq 1} \Vert  l_{i}(\bm{f})(t) \Vert < \infty$ and write
\begin{eqnarray*}
\Vert \mathbbm{1}_{[0,t]\cap S_{n,k}} \, l_{i}(\bm{f})  \Vert^{2} 
&=& 
\int_{U}   \mathbbm{1}_{[0,t]\cap S_{n,k}}(s) \, \big( l_{i}(\bm{f})(s)^{T} \cdot l_{i}(\bm{f})(s) \big) \, ds \\
&\leq&
\int_{U}  \mathbbm{1}_{[0,t]\cap S_{n,k}}(s) \, K_{i}^{2} \, ds \\
&=& 2^{-n+1} K^{2}_{i} < \infty \, .
\end{eqnarray*}
Setting $K = \max_{0 \leq i < d} K_{i}$, for all $(n,k)$ in $\I$, the $\R^{d \times d}$-valued functions
\begin{equation*}
\bm{\kappa}_{n,k}(t) = \int_{U} \mathbbm{1}_{[0,t]}(s) \bm{f}(s) \, \bphi_{n,k}(s) \, ds \, ,
\end{equation*}
satisfy $\Vert \bm{\kappa}_{n,k} \Vert_{\infty} = \sup_{0 \leq t \leq 1}\vert \bm{\kappa}_{n,k}(t) \vert \leq K \, 2^{-\frac{n+1}{2}} $.\\

Moreover, since $\bm{g}$ is also bounded continuous on $[0,1]$, there exist $L$ such that $\Vert \bm{g} \Vert_{\infty}  = \sup_{0 \leq t \leq 1}\vert \bm{g}(t) \vert \leq L$, and we finally have for all $0 \leq t \leq 1$: 
\[\Vert \bpsi_{n,k} \Vert_{\infty} \leq \Vert \bm{g} \Vert_{\infty} \Vert \bm{\kappa}_{n,k}\Vert_{\infty} \leq L \, K  \, 2^{-\frac{n+1}{2}}.\]

We now use this bound and equation \eqref{eq:proofupperbound}, we have:
\begin{equation} \label{eq:UnifBound}
\big \vert \bPsi^{N}(\bxi)(t) -   \bPsi^{n_{\xi}}(\bxi)(t) \big \vert
\leq
\sum_{(n,k) \in \I_{N} \setminus \I_{n_{\xi}}} \vert  \bpsi_{n,k}(t) \cdot \bxi_{n,k} \vert \leq \frac{L \, K}{\sqrt{2}} \sum_{n=n_{\xi}}^{\infty} \left( 2^{\frac{\delta-1}{2}} \right)^{n} \, .
\end{equation}
and since $\delta<1$,  for the continuous functions $t \mapsto \bPsi^{N}_{t}(\bxi)$ forms a uniformly convergent sequence of functions for the $d-$dimensional uniform norm. This sequence therefore converges towards a continuous function, and $\bPsi$ is well defined on $\uxi\Omega'$ and takes values in $ C_0([0,1],\R^d)$.

\end{proof}

This proposition being proved, we dispose of the map $\bPsi = \lim_{N \to \infty} \bPsi^{N}$. We now turn to prove different useful properties on this function. We denote $\mathcal{B}\left( C_{0}\big([0,1],\R^{d}\big) \right)$ the Borelian sets of the $d$-dimensional Wiener space $C_{0}\big([0,1],\R^{d}\big)$.

\begin{lemma}\label{lem:injection}
The function $\bPsi: \left( \uxi \Omega',\mathcal{B}\left(\uxi \Omega' \right)\right)     \rightarrow    \left(C_{0}\big([0,1],\R^{d}\big),\mathcal{B}\left(   C_{0}\big([0,1],\R^{d}\big) \right)  \right)$ is a linear injection.
\end{lemma}

\begin{proof}
The application $\bPsi$ is clearly linear.
The injective property simply results from the existence of the dual family of distributions $\bdelta_{n,k}$.
Indeed, for every $\bxi$,  $\bxi'$ in $\uxi \Omega'$, we have $\bPsi(\xi) = \bPsi(\xi')$ entails that for all $n, k$, $\bxi_{n,k} = \mathcal{P}\left(  \bdelta_{n,k},  \bPsi(\xi) \right) =  \mathcal{P}\left(  \bdelta_{n,k}, \bPsi(\xi') \right) = \bxi'_{n,k}$.
\end{proof}

In the one-dimensional case, as mentioned in the main text, because the uniform convergence of the sample paths is preserved as long as $\alpha$ is continuous and $\Gamma$ is non-zero through \eqref{eq:UnifBound}, the definition $\ux \Omega'$ does not depend on $\alpha$ or $\Gamma$, and the space $\ux \Omega'$ is large enough to contain reasonably regular functions:

\begin{proposition}\label{lem:Hold}
In the one-dimensional case, the space ${}_{ \scriptscriptstyle x \displaystyle}\Omega'$ contains the space of uniformly H\"{o}lder continuous functions $H$ defined as
\begin{equation}
H = \bigg \lbrace x \in C[0,1] \, \Big \vert \, \exists \delta > 0,   \sup_{0 \leq s,t \leq 1} \frac{\vert x(t)-x(s) \vert}{\vert t-s\vert^{\delta}} < +\infty \bigg \rbrace \, .
\nonumber\\
\end{equation}
\end{proposition}

\begin{remark}
This point can be seen as a direct consequence of the characterization of the local H\"{older} exponent of a continuous real function in term of the asymptotic behavior of its coefficients in the decomposition on the Schauder basis~\cite{Meyer:1998}.
\end{remark}

\begin{proof}
To underline that we place ourself in the one-dimensional case, we drop the bold notations that indicates multidimensional quantities.
Suppose $x$ is uniformly H\"{o}lder continuous for a given $\delta>0$, there always exist  $\xi$ such that $\Psi^{N}(\xi)$ coincides with $x$ on $D_{N}$:  
it is enough to take $\xi$ such that for all $(n,k)$ in $\I_{N}$, $\xi_{n,k}  = \left( \delta_{n,k} , x \right)$.
We can further write for $n>0$
\begin{eqnarray}
\left( x , \delta_{n,k} \right)   
&=&
M_{n,k} \frac{x(m_{n,k})}{g(m_{n,k})}  - \left( L_{n,k} \frac{x(l_{n,k})}{g_{\alpha}(l_{n,k})} +R_{n,k} \frac{x(r_{n,k})}{g(r_{n,k})} \right)
 \, , \nonumber\\
 &=&
L_{n,k} \left( \frac{x(m_{n,k})}{g(m_{n,k})} - \frac{x(l_{n,k})}{g(l_{n,k})} \right) + R_{n,k} \left( \frac{x(m_{n,k})}{g(m_{n,k})} - \frac{x(r_{n,k})}{g(r_{n,k})} \right) \, .
\nonumber
\end{eqnarray}
For a given function $\alpha$, posing $N_{\alpha} = \frac{\sup_{0 \leq t \leq 1} f_{\alpha}(t)}{\inf_{0 \leq t \leq 1} f_{\alpha}^{2}(t)}$, we have 
\begin{equation}
\ag M_{n,k}  \leq  N_{\alpha} \, 2^{\frac{n+1}{2}} \, , \quad \ag  L_{n,k}  \leq N_{\alpha} \, 2^{\frac{n-1}{2}} \, , \quad \ag R_{n,k} \leq N_{\alpha} \, 2^{\frac{n-1}{2}} \, .
\nonumber
\end{equation}
Moreover, if $\alpha$ is in $H$, it is straightforward to see that $g_{\alpha}$ has a continuous derivative.
Then, since $x$ is $\delta$-H\"{o}lder, for any $\epsilon \geq 0$, there exists $C>0$ such that $\vert t-s \vert \leq \epsilon$ entails
\begin{equation}
\bigg \vert \frac{x(t)}{g(t)} - \frac{x(s)}{g(s)} \bigg \vert \leq C \epsilon^{\delta}\, ,
\nonumber
\end{equation}
from which we directly deduce
\begin{equation}
\big \vert \xi_{n,k} \big \vert \leq \frac{N_{\alpha} \, C}{\sqrt{2}} \, 2^{n \left(\frac{1}{2} -2 \delta \right)} \, .
\nonumber
\end{equation}
This demonstrates that $\lbrace \xi_{n,k} \rbrace$ belongs to $\uxi \Omega'$ and ends the proof.
\end{proof}

We equip the space $\ux \Omega'$ with the topology induced by the uniform norm on $C_{0}\big([0,1],\R^{d}\big)$.
As usual, we denote $\mathcal{B}(\ux \Omega')$ the corresponding Borelian sets. We now show Proposition \ref{lem:bijection}

\begin{proposition}
The function $\bPsi: \left( \uxi \Omega',\mathcal{B}\left(\uxi \Omega' \right)\right)     \rightarrow    \left(\ux \Omega', \mathcal{B}(\ux \Omega')  \right)$ is a bounded continuous bijection.
\end{proposition}

\begin{proof}
Consider an open ball $\ux B(x,\epsilon)$ of $\ux \Omega'$ of radius $\epsilon$.
If we take $M = L \, K / \sqrt{2}$ as defined in \eqref{eq:UnifBound}, we can choose a real $\delta > 0$ such that
\begin{equation}
\delta < \epsilon M \left( \sum_{n= 0}^{\infty} 2^{-n/2} \right)^{-1} \, .
\nonumber
\end{equation}
Let us consider $\bxi$ in $\uxi \Omega'$ such that $\bPsi(\xi) = x$.
Then by \eqref{eq:UnifBound}, we immediately have that, for all $\bxi'$ in the ball of radius $\uxi B(\xi, \delta)$ of $\uxi \Omega$, $\Vert  \bPsi( \bxi -\bxi')\Vert_{\infty} \leq \epsilon$.
This shows that $\bPsi^{-1}(\ux B(x,\epsilon))$ is open and that $\bPsi$ is continuous for the 
$d-$dimensional uniform norm topology.
\end{proof}

\subsection{The Coefficient application}
In this section of the appendix we show some useful properties of the coefficient application introduced in section \ref{app:CoefAppl}. 

\begin{lemma}\label{lemXi}
The function $\bDelta:\left(C_{0}\big([0,1],\R^{d}\big) ,\mathcal{B}\left(   C_{0}\big([0,1],\R^{d}\big) \right)  \right) \rightarrow \left(  \uxi \Omega,\mathcal{B}\left(\uxi \Omega\right)\right) $ is a measurable linear injection. 
\end{lemma}

\begin{proof}
\emph{i}) 
The function $\bDelta$ is clearly linear.\\
\emph{ii}) 
To prove that, $\bDelta$ is injective, we show that for $\bm{x}$ and $\bm{y}$ in $C_{0}\big([0,1],\R^{d}\big)$, $\bm{x} \neq \bm{y}$ implies that $\bDelta(\bm{x}) \neq  \bDelta(y)$. To this end, we fix $\bm{x} \neq \bm{y}$ in $C_{0}\big([0,1],\R^{d}\big)$ equipped with the uniform norm, and 
consider the continuous function
\begin{equation}
\bm{d}_{N}(t) = \sum_{(n,k) \in \I_N}  \bpsi_{n,k}(t) \left( \bDelta(\bm{x})_{n,k} -  \bDelta(\bm{y})_{n,k} \right)\,.
\nonumber
\end{equation}
This function coincides with $\bm{x}-\bm{y}$ on every dyadic numbers in $D_{N}$ and has zero value if $\bDelta(\bm{x}) =  \bDelta(\bm{y})$.
Since $\bm{x} \neq \bm{y}$, there exists $s$ in $]0,1[$ such that $\bm{x}(s) \neq \bm{y}(s)$, and by continuity of $\bm{x}-\bm{y}$, there exists an $\varepsilon>0$ such that $\bm{x} \neq \bm{y}$ on the ball $]s-\varepsilon, s+ \varepsilon[$.
But, for $N$ large enough, there exists $k$, $0 \leq k < 2^{N\!-\!1}$ such that  $\vert s - k2^{-N} \vert < \varepsilon $.
We then necessarily have that $\bDelta(f) \neq \bDelta(g)$, otherwise, we would have $d_{N}(k2^{-N}) = (\bm{x}-\bm{y})(k2^{-N}) = 0$, which would contradict the choice of $\varepsilon$. \\
\emph{iii})
Before proving the measurability of $\bDelta$, we need the following observation.
Consider for $N>0$, the finite dimensional linear function $\bDelta_{N}$ 
\begin{eqnarray}
C_{0}\big([0,1],\R^{d}\big) & \longrightarrow &\left( \R^{d}\right)^{2^{N-1}} 
\nonumber\\
\bm{x} & \longmapsto &  \bDelta_{N}(x) = \large\{  \bDelta(x)_{N,k} \large\}_{(N,k) \in \I_{N}} \, .
\nonumber
\end{eqnarray}
Since for all $(N,k)$, the matrices $\bM_{N,k}$, $\bR_{N,k}$, $\bL_{N,k}$ are all bounded, the function $\bDelta_{N}:\left(C_{0}\big([0,1],\R^{d}\big) ,\mathcal{B}\left(  C_{0}\big([0,1],\R^{d}\big)\right)  \right) \rightarrow \left(  \left( \R^{d}\right)^{2^{N-1}},\mathcal{B}\left(\left( \R^{d}\right)^{2^{N-1}}\right)\right) $ is a continuous linear application.
To show that the function $\bDelta$ is measurable, it is enough to show that the pre-image by $\bDelta$ of the generative cylinder sets of $\mathcal{B}\left(\uxi \Omega \right)$ belong to $\mathcal{B}\left( C_{0}\big([0,1],\R^{d}\big) \right)$.\\
For any $N\! \geq \! 0$, take an arbitrary Borel set 
\begin{equation}
B = \prod_{(n,k) \in \I_{N}} B_{n,k} \quad \in \quad \mathcal{B}\left(\left(\mathbb{R}^{d}\right)^{\I_{N}}\right) \, ,
\end{equation}
and define the cylinder set $\mathcal{C}_{N}(B)$ as
\begin{equation}
\mathcal{C}_{N}(B)
=
\Big \lbrace  \bxi \in \uxi \Omega  
\, \Big \vert \,
\forall \; (n,k) \in  {I}_{N}\, , \bxi_{n,k} \in B_{n,k}\, \Big \rbrace \, ,
\nonumber\\
\end{equation}
and we write the collection of cylinder sets $C$ as 
\begin{equation}
C = \bigcup_{n \geq 0} C_{N} \quad \mathrm{with} \quad C_{N}  =  \bigcup_{B \in \left(\mathbb{R}^{d}\right)^{\I_{N}}} \mathcal{C}_{N}(B) \, .
\nonumber
\end{equation}
We proceed by induction on $N$ to show that the pre-image by $\bDelta$ of any cylinder set in $C$  is in $\mathcal{B}\left( C_{0}\big([0,1],\R^{d}\big) \right)$. For $N\!=\!0$, a cylinder set of $C_{0}$ is of the form $B_{0,0}$ in $\mathcal{B}\left(\R^{d}\right)$, $ \bDelta^{-1}(B) = \lbrace \bm{x} \in C_{0}\big([0,1],\R^{d}\big) \, \vert \, \bm{x}(1) \in \,  \bL^{T}_{0,0} \, \bm{g}^{-1}(r_{0,0}) \left( B_{0,0} \right) \rbrace$, which is measurable for being a cylinder set of $\mathcal{B}\left( C_{0}\big([0,1],\R^{d}\big) \right)$. Suppose now that for $N>0$, for any set $A$ in $C_{N-1}$, the set $\bDelta^{-1}(A)$ is measurable. Then considering a set $A$ in $C_{N}$, there exist  $B$ in $\mathcal{B}\left( \left( \R^d \right)^{\I_{N}} \right)$ such that $A = \mathcal{C}_{N}(B)$.
Define $A'$ in $C_{N}$ such that $A'= \mathcal{C}_{N-1}(B')$, where  
\begin{equation}
 B' = \prod_{(n,k) \in \I_{N-1}} B_{n,k} \, .
\nonumber
\end{equation}
and remark that $A= \mathcal{C}_{N}(B) \subset A'= \mathcal{C}_{N}(B')$.
Clearly, we have that $A =  A' \cap D$, where we have defined the cylinder set $D$ as
\begin{equation}
D = \mathcal{C}_{I_{N}}\left(\prod_{ (N,k) \in \I_{N,k} } B_{N,k}\right) \, .
\nonumber
\end{equation}
Having defined the function $\bDelta_{N}$,
we now have have $ \bDelta^{-1}(A) = \bDelta^{-1}(A' \cap D) =  \bDelta^{-1}(A') \cap \bDelta^{-1}(D) =  \bDelta^{-1}(A') \cap \bDelta^{-1}_{N}(D)$.
Because of the continuity of $\bDelta_{N}$, $\bDelta_{N}^{-1}(D)$ is a Borel set of $\mathcal{B}\left( C_{0}\big([0,1],\R^{d}\big) \right)$.
Since, by hypothesis of recurrence, $\bDelta^{-1}(A')$ is in $\mathcal{B}\left( C_{0}\big([0,1],\R^{d}\big) \right)$, $\bDelta^{-1}(A)$ is also in $\mathcal{B}\left(C_{0}\big([0,1],\R^{d}\big) \right)$ as the intersection of two Borel sets.
The proof of the measurability of $\bDelta$ is complete.
\end{proof}

We now demonstrate Proposition \ref{lemInv}. 

\begin{proposition}
The function $\bDelta:\left(\ux \Omega' ,\mathcal{B}\left(  \ux \Omega' \right)  \right) \rightarrow \left(  \uxi \Omega',\mathcal{B}\left(\uxi \Omega' \right)\right) $ is a measurable linear bijection whose inverse is $\bPsi = \bDelta^{-1}$. 
\end{proposition}

\begin{proof}
	Let $x\in \ux \Omega'$ be a continuous function. We have:
	\begin{align*}
		\bPsi(\bDelta(x))(t) &= \sum_{(n,k)\in \I} \bpsi_{n,k}(t)\cdot \bDelta_{n,k}\\
		&= \sum_{(n,k)\in \I} \bpsi_{n,k}(t)\cdot \mathcal{P}(\bdelta_{n,k},x)
	\end{align*}
	This function is equal to $x(t)$ for any $t\in D$ the set of dyadic numbers. Since $D$ is dense in $[0,1]$ and both $x$ and $\bPsi(\bDelta(x))$ are continuous, the two functions, coinciding on the dyadic numbers, are equal for the uniform distance, and hence $\bPsi(\bDelta(x)) = x$. 
\end{proof}


\section{It\^o Formula}\label{append:Ito}

In this section we provide rigorous proofs of Proposition \ref{prop:IPP} and Theorem \ref{theo:Ito} related to It\^o formula. 

\begin{proposition}[Integration by parts]
	Let $(X_t)$ and $(Y_t)$ be two one-dimensional Gauss-Markov processes starting from zero. Then we have the following equality in law:
	\[X_t\,Y_t= \int_0^t X_s \circ dY_s + \int_0^t Y_s \circ dX_s\]
	where $\int_0^t A_s \circ dB_s$ for $A_t$ and $B_t$ two stochastic processes denotes the Stratonovich integral. In terms of It\^o's integral, this formula is written:
	\[X_t\,Y_t= \int_0^t X_s dY_s + \int_0^t Y_s dX_s + \langle X,Y \rangle_t\]
	where the brackets denote the mean quadratic variation. 
\end{proposition}

\begin{proof}
	We assume that $X$ and $Y$ satisfy the equations:
	\[\begin{cases}
	dX_t &= \alpha_\dX(t) X_t + \sqrt{\Gamma_\dX(t)} \, dW_t\\
	dY_t &= \alpha_\dY(t) X_t + \sqrt{\Gamma_\dY(t)} \, dW_t
	\end{cases}\]
	and we introduce the functions  $f_\dX,\,f_\dY, \, g_\dX, \, g_\dY$ such that $X_t=g_\dX(t) \int_0^t f_\dX(s)\,$ and  $Y_t=g_\dY(t) \int_0^t f_\dY(s)\,$.
	
	 We define $(\uX \psi_{n,k})_{(n,k)\in \I}$and  $(\uY \psi_{n,k})_{(n,k)\in \I}$ the construction bases of the processes $X$ and $Y$. 
	 Therefore, using Theorem \ref{theo:PrincipTheo}, there exist $(\uX \Xi_{n,k})_{(n,k) \in \I}$ and  $(\uY \Xi_{p,q})_{(p,q) \in \I}$ standard normal independent variables such that $X = \sum_{(n,k) \in \I}\uX  \psi_{n,k} \cdot \uX \Xi_{n,k}$ and $Y =\sum_{(p,q) \in \I} \uY \psi_{n,k} \cdot \uY \Xi_{n,k}.$ 
	 and we know that the processes $X$ and $Y$ are almost-sure uniform limits when $N\to \infty$ of the processes $X^N$ and $Y^N$ defined as the partial sums:
	\[X^N =\sum_{(n,k)\in\I_N} \uX  \psi_{n,k} \cdot \uX \Xi_{n,k} \qquad \text{and} \qquad Y^N =\sum_{(p,q)\in\I_N}  \uY \psi_{n,k} \cdot \uY \Xi_{n,k}.\]
	Using the fact that the functions $\uX \psi_{n,k}$ and  $\uY \psi_{n,k}$ have piecewise continuous derivatives, we have:
	\begin{eqnarray*}
		X^N_t\,Y^N_t &=& \sum_{(n,k) \in I_N}  \sum_{(p,q) \in I_N} \uX \psi_{n,k}(t)\, \uY \psi_{p,q}(t) \, \uX \Xi_{n,k}  \, \uY \Xi_{p,q}\\
		&=& \sum_{(n,k) \in I_N}  \sum_{(p,q) \in I_N} \uX \Xi_{n,k}  \, \uY \Xi_{p,q} \int_0^t \frac{d}{ds}\left( \uX \psi_{n,k}(s) \uY \psi_{p,q}(s)\right)\,dt\\
		&=& \sum_{(n,k) \in I_N}  \sum_{(p,q) \in I_N} \uX \Xi_{n,k}  \, \uY \Xi_{p,q} \int_0^t \left( \uX \psi _{n,k}'(s)  \uY \psi_{p,q}(s) +\uX \psi _{n,k}(s)  \uY \psi'_{p,q}(s)\right)\,ds\\
	\end{eqnarray*}
	Therefore we need to evaluate the piecewise derivative of the functions $\uX \psi_{n,k}$ and  $\uY \psi_{n,k}$. We know that 
	\[\frac{1}{f_\dX}\left(\frac{\uX \psi _{n,k}}{g_\dX}\right)'(t) = \uX \phi_{n,k}(t)\]
	which entails:
	\[\uX \psi'_{n,k} = \alpha_\dX\, \uX \psi_{n,k} + g_\dX \,f_\dX \,\uX \phi_{n,k} = 	\alpha_\dX \uX \psi_{n,k} + \sqrt{\Gamma_\dX} \, \uX \phi_{n,k}\]
	and similarly so for the process $Y$. Therefore, we have:
	\begin{eqnarray*}
	X^N \,Y^N = A^{n,k}_{p,q} + B^{n,k}_{p,q} + C^{n,k}_{p,q} + D^{n,k}_{p,q} \, ,
	\end{eqnarray*}
	with 
	\begin{eqnarray*}
 	A_t & = & \sum_{(n,k) \in I_N}  \sum_{(p,q) \in I_N} \left(  \int_0^t \alpha_\dX(s) \, \uX \psi_{n,k}(s)\, \uY \psi_{p,q}(s)\,ds \right) \uX \Xi_{n,k}  \, \uY \Xi_{p,q} \\
 	B_t & = & \sum_{(n,k) \in I_N}  \sum_{(p,q) \in I_N} \left( \int_0^t  \sqrt{\Gamma_\dX(s)} \uX \phi_{n,k}(s)\, \uY  \psi_{p,q}(s) \,ds \right) \uX \Xi_{n,k}  \, \uY \Xi_{p,q} \\
  	C_t & = &  \sum_{(n,k) \in I_N}  \sum_{(p,q) \in I_N} \left( \int_0^t \alpha_\dY(s) \, \uX \psi_{n,k}(s)\, \uY \psi_{p,q}(s)  \,ds \right) \uX \Xi_{n,k}  \, \uY \Xi_{p,q} \\
  	 D_t & = &\sum_{(n,k) \in I_N}  \sum_{(p,q) \in I_N}  \left( \int_0^t \sqrt{\Gamma_\dY(s)} \, \uY \phi_{n,k}(s)\,\uX \psi_{p,q}(s) \right) \,ds \uX \Xi_{n,k}  \, \uY \Xi_{p,q} 
	\end{eqnarray*}
	We easily compute:
	\begin{eqnarray*}
		A_t+C_t = \int_0^t (\alpha_\dX(s)+\alpha_\dY(s)) X^N(s)\,Y^N(s)\,ds.
	\end{eqnarray*}
	For $t\in [0,1]$ as it is our case, $X^N(s)$ and $Y^N(s)$ are both almost surely finite for all $t$ in $[0,1]$. For almost all $\uY \xi$ and $\uY \xi$ drawn with respect to the law of the Gaussian infinite vector $\Xi$, we therefore have, by the Lebesgue's dominated convergence theorem that this integral converges almost surely towards
	\[\int_0^t (\alpha_\dX(s)+\alpha_\dY(s)) X(s)\,Y(s)\,ds.\]
	The other two terms $B_t$ and $D_t$ necessitate a more thorough analysis, and we treat it as follows. Let us start by considering the first one of this term:
	\begin{eqnarray*}
		B_t &=&\int_0^1 \mathbbm{1}_{[0,t]}(s) \sqrt{\Gamma_\dX(s)} \sum_{(n,k) \in I_N}  \sum_{(p,q) \in I_N} \uX \phi_{n,k}(s)\, \uY \psi_{p,q}(s) \, \uX \Xi_{n,k}  \, \uY \Xi_{p,q} \,ds\\
		&=& \sum_{t_i\in \mathcal{D}_N \setminus\{1\}} \int_{t_i}^{t_{i+1}} \mathbbm{1}_{[0,t]}(s) \sqrt{\Gamma_\dX(s)} \left(\sum_{(n,k)\in\I_N} \uX  \phi_{n,k}(s) \cdot \uX \Xi_{n,k}\right)\,Y^{N}(s) \,ds\\
	 &=& \sum_{t_i\in \mathcal{D}_N \setminus\{1\}} \int_{t_i}^{t_{i+1}} \mathbbm{1}_{[0,t]}(s) \sqrt{\Gamma_\dX(s)} \left(\frac{1}{f_\dX(s)}\left(\frac{X^N}{g_\dX}(s)\right)'\right) \, Y^N(s)\,ds
	\end{eqnarray*}
	Let us now have a closer look at the process $X^N_t$ for $t \in [t_i,t_{i+1}]$ where $[t_i, t_{i+1}]=S_{N,i}$ for $i$ such that $(N,i)\in \I$.
	Because of the structure of our construction, we have:
	\begin{equation}\label{eq:XNt}
		Y^N(t) = \frac{g_\dY(t)}{g_\dY(t_i)}\frac{h_\dY(t_{i+1})-h_\dY(t)}{h_\dY(t_{i+1})-h_\dY(t_i)}\cdot Y_{t_i} + \frac{g_\dY(t)}{g_\dY(t_{i+1})}\frac{h_\dY(t)-h_\dY(t_{i})}{h_\dY(t_{i+1})-h_\dY(t_i)}\cdot Y_{t_{i+1}}
	\end{equation}
	and 
	\begin{equation}\label{eq:DXNt}
			\frac{1}{f_\dX(t)}\left(\frac{X^N}{g_\dX}\right)'(t) = \frac{f_\dX(t)}{h_\dX(t_{i+1})-h_\dX(t_{i})}\left(\frac{X_{t_{i+1}}}{g_\dX(t_{i+1})} - \frac{X_{t_{i}}}{g_\dX(t_{i})}\right).
		\end{equation}
	We therefore have:
	\begin{eqnarray*}
		&&\int_{t_i}^{t_{i+1}}  \left(\frac{\mathbbm{1}_{[0,t]}(s) \sqrt{\Gamma_\dX(s)}}{f_\dX(s)}\frac{d}{ds}\left(\frac{X^N(s)}{g_\dX(s)}\right)\right) \, Y^N(s)\,ds\\
		&& \qquad =\int_{t_i}^{t_{i+1}} \frac{ \mathbbm{1}_{[0,t]}(s) \sqrt{\Gamma_\dX(s)} f_\dX(s)}{{h_\dX(t_{i+1})-h_\dX(t_{i})}} g_\dY(s)\Bigg [ \frac{h_\dY(t_{i+1})-h_\dY(s)}{h_\dY(t_{i+1})-h_\dY(t_i)} \frac{Y_{t_i}}{g_\dY(t_i)} +\\
		&& \qquad \qquad \qquad \frac{h_\dY(s)-h_\dY(t_i)}{h_\dY(t_{i+1})-h_\dY(t_i)} \frac{Y_{t_{i+1}}}{g_\dY(t_{i+1})}\Bigg]\; ds\; \left(\frac{X_{t_{i+1}}}{g_\dX(t_{i+1})} - \frac{X_{t_{i}}}{g_\dX(t_{i})} \right)\\
		&& \qquad = \left[ v_i(t)^N Y_{t_i}^N +w_i(t)^N Y_{t_{i+1}}^N\right]  \left(\frac{X_{t_{i+1}}}{g_\dX(t_{i+1})} - \frac{X_{t_{i}}}{g_\dX(t_{i})} \right)
	\end{eqnarray*}
	with 
	\[
	\begin{cases}
		v_i^N(t) &= \displaystyle \int_{t_i}^{t_{i+1}} \frac{\mathbbm{1}_{[0,t]}(s)  \sqrt{\Gamma_\dX}(s) f_\dX(s)}{h_\dX(t_{i+1}) - h_\dX(t_i)} \frac{g_\dY(s)}{g_\dY(t_i)} \frac{h_\dY(t_{i+1})-h_\dY(s)}{h_\dY(t_{i+1})-h_\dY(t_i)}\, ds\\
	w_i^N(t) &= \displaystyle \int_{t_i}^{t_{i+1}}\frac{\mathbbm{1}_{[0,t]}(s)  \sqrt{\Gamma_\dX}(s) f_\dX(s)}{h_\dX(t_{i+1}) - h_\dX(t_i)} \frac{g_\dY(s)}{g_\dY(t_{i+1})} \frac{h_\dY(s)-h_\dY(t_{i})}{h_\dY(t_{i+1})-h_\dY(t_i)}\,ds
	\end{cases}
		\]
		Let us denote $\delta_N$ the time step of the partition $\delta_N = \max_{t_i\in \mathcal{D}_N\setminus \{1\}} (t_{i+1}-t_i)$, which is smaller than $\rho^N$ with $\rho \in (0,1)$ from the assumption made in section \ref{sec:supports}. Moreover, we know that the functions $g_\dX$, $g_\dY$, $h_\dX$ and $h_\dY$ are continuously differentiable, and since $\sqrt{\Gamma}_\dX$ and  $\sqrt{\Gamma}_\dY$ are $\delta$-H\"older, so are $f_\dX$ and $f_\dY$. When $N \to \infty$ (i.e. when $\delta_N \to 0$), using Taylor and H\"older expansions for the differential functions, we can further evaluate the integrals we are considering. Let us first assume that $t>t_{i+1}$. We have:
		\begin{align*}
			v_i(t)^N &= \int_{t_i}^{t_{i+1}} \frac{\sqrt{\Gamma_\dX}(s) f_\dX(s)}{h_\dX(t_{i+1}) - h_\dX(t_i)} \frac{g_\dY(s)}{g_\dY(t_i)} \frac{h_\dY(t_{i+1})-h_\dY(s)}{h_\dY(t_{i+1})-h_\dY(t_i)}\, ds\\
			& =\frac{ \sqrt{\Gamma_\dX}(t_i)\big(1+O(\delta_N^\delta)\big) \, f_\dX(t_i)\big(1+O(\delta_N^\delta)\big)}{f_\dX(t_{i})^2 (t_{i+1}-t_i)\big(1+ O(\delta_N)\big)} \\
			& \qquad \frac{g_\dY(t_i) \big(1+ O(\delta_N)\big)}{g_\dY(t_i)}\int_{t_i}^{t_{i+1}}  \frac{f_\dY(t_i)^2\,(t_{i+1}-s)\big(1+O(\delta_N)\big)}{f_\dY(t_i)^2\,(t_{i+1}-t_i)\big(1+O(\delta_N)\big)}\, ds\\
			&= \frac{\sqrt{\Gamma_\dX}(t_i) }{f_\dX(t_i)(t_{i+1}-t_i)} \left( \int_{t_i}^{t_{i+1}}  \frac{t_{i+1}-s}{t_{i+1}-t_i}\, ds \right) \bigg(1+O(\delta_N)+O(\delta_N^{\delta})\bigg)\\
			& = \frac 1 2 g_\dX(t_i) + O(\delta_N + \delta_N^{\delta})
		\end{align*}
Similarly, we show that that $w_i^N(t) = \frac 1 2 g_\dX(t_{i})+ O(\delta_N + \delta_N^{\delta})$ when $N \to \infty$. If $t<t_i$, we have $v_i^N(t) = w_i^N(t)= 0$ and for $t$ in $[t_{i_0}, t_{i_0}+1)$ we have:
\[\begin{cases}
	v_{i_0}^N(t) &= \displaystyle \frac{g_\dX(t_{i_0})}{ 2}  \left(\frac{t_{{i_0}+1}-t}{t_{{i_0}+1}-t_{i_0}}\right)^{2} + O(\delta_N+\delta_N^{\delta})\\
	w_{i_0}^N(t) &= \displaystyle  \frac{g_\dX(t_{i_0})}{ 2}\left(\frac{t-t_{i_0}}{t_{{i_0}+1}-t_{i_0}}\right)^{2}+ O(\delta_N+\delta_N^{\delta}) = v_{i_0}^N(t)+ O(\delta_N+\delta_N^{\delta})
\end{cases}\] 
We then finally have:
		\begin{multline*}
			B_t =  \sum_{t_i\in \mathcal{D}_N;\; t_{i+1}\leq t} \frac{ g_\dX(t_i)}{2} \left(Y_{t_i}^N + Y^N_{t_{i+1}}\right) \left(\frac{X_{t_{i+1}}^N}{g_\dX(t_{i+1})} - \frac{X_{t_{i}}^N}{g_\dX(t_{i})} \right) \\
			+ \frac { g_\dX(t_{i_0})}{2}\left(\frac{t_{{i_0}+1}-t}{t_{{i_0}+1}-t_{i_0}}\right)^{2}\left(Y_{t_{i_0}}^N + Y^N_{t}\right) \left(\frac{X_{t_{i+1}}^N}{g_\dX(t)} - \frac{X_{t_{i_0}}^N}{g_\dX(t_{i_0})} \right) + O(\delta_N+\delta_N^{\delta})
		\end{multline*}
Moreover, we observe that the process $X_t/g_\dX(t) = \int_{t_{i}}^{t_{i+1}} f_\dX(s)\, dW_s$ is a martingale, and by definition of Stratonovich integral for martingale processes, we have:
\[
	 B_t \mathop{\longrightarrow}\limits_{N\to \infty} \int_{0}^t g_\dX(s) Y_s \circ d \left(X_s/g_\dX(s) \right) = \int_{0}^t \sqrt{\Gamma_\dX}(s) Y_s \circ dW(s)
			\]
		where $\circ$ is used to denote the Stratonovich stochastic integral and the limit is taken in distribution. Notice that the fact that the sum converges towards Stratonovich integral does not depend on the type of sequence of partition chosen which can be different from the dyadic partition. 
		Putting all these results together, we obtain the equality in law:
		\begin{eqnarray*}
			X_t\,Y_t &= &\int_0^t \alpha_\dX(s) X_s \,Y_s \,ds + \int_0^t \sqrt{\Gamma_\dX}(s) Y_s \circ dW_s +	\\
			&&	\qquad	\int_0^t \alpha_\dY(s) X_s \,Y_s \,ds + \int_0^t \sqrt{\Gamma_\dY}(s) X_s \circ dW_s \, ,
		\end{eqnarray*}
		which is exactly the integration by parts formula we were searching for. The integration by parts formula for It\^o stochastic integral directly comes from the relationship between Stratonovich and It\^o stochastic integral. 
\end{proof}

\begin{theorem}[It\^o]
	Let $X$ be a Gauss-Markov process and $f$ in $C^2(\R)$. The process $f(X_t)$ is a Markov process and satisfies the relation:
	\begin{equation}\label{eq:Ito2}
		f(X_t)=f(X_0)+\int_0^t f'(X_s)dX_s + \frac 1 2 \int_0^t f''(X_s) d\langle X\rangle_s
	\end{equation}
\end{theorem}

\begin{proof}
	The integration by parts formula directly implies It\^o's formula through a density argument, as follows. Let $\mathcal{A}$ be the set of functions $f \in C^2([0,1],\R)$ such that equation \eqref{eq:Ito2} is true. It is clear that $\mathcal{A}$ is a vector space. Moreover, because of the result of Proposition \ref{prop:IPP}, the space $\mathcal{A}$ is an algebra. Since all constant functions and the identity function $f(x)=x$ trivially belong to $\mathcal{A}$, the algebra $\mathcal{A}$ contains all polynomial functions. 
	
	Let now $f \in C^2([0,1],\R)$. There exists a sequence of polynomials $P_k$ such that $P_k$ (resp. $P_k$, $P_k''$) uniformly converges towards $f$ (resp. $f'$, $f''$). Let us denote $U_n$ the sequence of stopping times
	\[U_n=\inf \{t\in [0,1]\;;\; \vert X_t\vert > n\}.\]
	This sequences grows towards infinity. We have:
	\[
		P_k(X_{t\wedge U_n}) - P_k(X_{0}) = \int_0^t P_k'(X_s)\mathbbm{1}_{[0,U_n]}(s) dX_s + \frac 1 2 \int_0^t P_k''(X_s) \mathbbm{1}_{[0,U_n]}(s) d\langle X \rangle_s
	\]
	On the interval $[0,U_n]$, we have $X_t \leq n$, which allows to use Lebesgue's dominated convergence theorem on each term of the equality. We have:
	\begin{multline*}
		\mathbbm{E}\left [\Bigg \vert \int_0^t P_k'(X_s)\mathbbm{1}_{[0,U_n]}(s) dX_s- \int_0^t F'(X_s)\mathbbm{1}_{[0,U_n]}(s) dX_s \Bigg \vert^2\right] \\= \mathbbm{E}\left [\int_0^t \left \vert P_k'(X_s) - F'(X_s)\right\vert ^2 \mathbbm{1}_{[0,U_n]}(s) d\langle X\rangle _s \right]
	\end{multline*}
	which converges towards zero because of Lebesgue's theorem for Steljes integration. The same argument directly applies to the other term. Therefore, letting $k \to \infty$, we proved It\^o's formula for $X_{t\wedge U_n}$, and eventually letting $n\to \infty$ obtain the desired formula. 
\end{proof}


\section{Trace Class Operator}\label{append:TraceClass}

In this section, we demonstrate the Theorem \ref{thCompactGG}, which proves instrumental to extend the finite-dimensional change of variable formula to the infinite dimensional case.
The proof relies on the following lemma:

\begin{lemma}\label{lemGGFF}
The operator $\ab S - Id: l^{2}(\mathbb{R}) \rightarrow l^{2}(\mathbb{R})$ is isometric to the operator $\ab R: l^{2}(\mathbb{R}) \rightarrow l^{2}(\mathbb{R})$ defined by
\begin{eqnarray}
\ab R[x] = \int_{0}^{1} \ab R(t,s) \, x(s) \, ds \, ,
\end{eqnarray}
with the kernel
\begin{eqnarray}
\ab R(t,s)
=
 \big( \alpha(t \vee s) - \beta(t \vee s) \big) \frac{f_{\alpha}(t \wedge s)}{f_{\alpha}(t \vee s)}
 + 
f _{\alpha}(t) \left( \int_{t \vee s}^{1} \frac{\big( \alpha(u) - \beta(u) \big)^{2}}{{f_{\alpha}}^{2}(u)} \, du \right) f_{\alpha} (s) \, .
 \nonumber
\end{eqnarray}
\end{lemma}

\begin{proof}
Notice first that
\begin{eqnarray}
 \big( \alpha(t \vee s) - \beta(t \vee s) \big) \frac{f_{\alpha}(t \wedge s)}{f_{\alpha}(t \vee s)}
=
\bold{1}_{\lbrace s < t \rbrace} \big( \alpha(t) - \beta(t) \big) \frac{f_{\alpha}(s)}{f_{\alpha}(t)}
+
\bold{1}_{\lbrace s \geq t \rbrace} \big( \alpha(s) - \beta(s) \big) \frac{f_{\alpha}(t)}{f_{\alpha}(s)} \, ,
 \nonumber
\end{eqnarray}
which leads to write in $L^{2}[0,1]$, for any $(n,k)$ and $(p,q)$ in $\I$:
\begin{eqnarray*}
\left(  \ag \phi_{n,k} ,\ab R[\ag  \phi_{p,q}] \right)
=
\int_{0}^{1} \int_{0}^{1} \ab R(t,s) \ag \phi_{n,k}(t) \ag  \phi_{p,q}(s) \, dt \, ds
=
A^{n,k}_{p,q} + B^{n,k}_{p,q} + C^{n,k}_{p,q} \, ,
\end{eqnarray*}
with
\begin{eqnarray}
A^{n,k}_{p,q}
 &=&
 \int_{0}^{1}   \frac{\alpha(t) - \beta(t)}{f_{\alpha}(t)} \ag \phi_{n,k}(t) \left(  \int_{0}^{t} f_{\alpha}(s)  \ag  \phi_{p,q}(s) \, ds \right) \, dt  \, ,  \nonumber\\
&=&
 \int_{0}^{1}  \frac{\alpha(t) - \beta(t)}{\sqrt{\Gamma(t)}} \ag \phi_{n,k}(t)   \ag  \psi_{p,q}(t)  \, dt \, ,
 \nonumber\\
B^{n,k}_{p,q}
 &=&
 \int_{0}^{1} \ag f(t) \phi_{n,k}(t)   \left(  \int_{t}^{1}  \frac{\alpha(s) - \beta(s)}{f_{\alpha}(s)}  \ag  \phi_{p,q}(s) \, ds \right) \, dt   \, , \nonumber\\
&=&
 \int_{0}^{1} \frac{\alpha(t) - \beta(t)}{\sqrt{\Gamma(t)}} \ag \psi_{n,k}(t)   \ag  \phi_{p,q}(t)  \, dt \, ,
 \nonumber
\end{eqnarray}
and 
\begin{eqnarray}
C^{n,k}_{p,q}
&=& 
 \int_{0}^{1} \int_{0}^{1}  \ag \phi_{n,k}(t) f_{\alpha} (t) \left( \int_{s,t} \frac{\big( \alpha(u) - \beta(u) \big)^{2}}{{f_{\alpha}}^{2}(u)} \, du \right) f (s) \ag  \phi_{p,q}(s) \, dt \, ds \, , \nonumber\\
&=&
 \int_{0}^{1} \int_{0}^{1} \int_{0}^{1} \frac{\big( \alpha(u) - \beta(u) \big)^{2}}{{f_{\alpha}}^{2}(u)} \, \nonumber\\
 && \qquad \qquad \left( \bold{1}_{[0,u]}(t)  f_{\alpha} (t) \ag \phi_{n,k}(t) \right) \, \left( \bold{1}_{[0,u]}(s) f_{\alpha} (s) \ag  \phi_{p,q}(s) \right) \, dt \, ds \, du  \, , \nonumber\\
 &=&
 \int_{0}^{1}  \frac{\big( \alpha(u) - \beta(u) \big)^{2}}{\Gamma(u)} \,   \ag \psi_{n,k}(t) \,  \ag \psi_{p,q}(t) \, dt \, ds \, .
 \nonumber
\end{eqnarray}
This proves that $\left(  \ag \phi_{n,k} ,\ab R[\ag  \phi_{p,q}] \right) = [\ab S - Id]^{n,k}_{p,q}$.
Therefore, if we denote the isometric linear operator
\begin{eqnarray}
\ag \Phi: l^{2}(\mathbb{R}) &\longrightarrow& L^{2}(\mathbb{R}) \nonumber\\
\xi & \mapsto& \ag \Phi[\xi] = \sum_{n=0}^{\infty}\sum_{ \hspace{5pt} 0 \leq k < 2^{n\!-\!1} } \ag \phi_{n,k} \cdot \xi_{n,k} \, ,
 \end{eqnarray}
we clearly have ${\ag \Phi}^{T} \circ \ab R \circ \ag \Phi = \ab S - Id$ with ${\ag \Phi}^{T} = {\ag \Phi}^{-1}$.
\end{proof}

We now proceed to demonstrate that $\ab S - Id$ is a trace-class operator.

\begin{proof}[Proof of Theorem \ref{thCompactGG}]
Since the kernel $\ab R(t,s)$ is  integrable in $L^{2}([0,1] \times [0,1])$, the integral operator $\ab R:L^{2}[0,1] \rightarrow L^{2}[0,1]$ is a Hibert-Schmidt operator and thus is compact.
Moreover it is a trace-class operator, since we have
\begin{eqnarray}
\mathrm{Tr}(\ab R)
& =&
\int_{0}^{1} \left( \alpha(t) - \beta(t) \right) \, dt
+
\int_{0}^{1} {f_{\alpha}}^{2}(t) \left( \int_{t}^{1} \frac{\big( \alpha(u) - \beta(u) \big)^{2}}{{f_{\alpha}}^{2}(u)} \, ds \right) \, dt
\nonumber\\
&=& 
\int_{0}^{1} \left( \alpha(t) - \beta(t) \right) \, dt
+
\int_{0}^{1} \frac{h_{\alpha}(t)}{{f_{\alpha}(t)}^{2}}  \big( \alpha(t) - \beta(t) \big)^{2} \, dt 
\nonumber
\end{eqnarray}
Since $\ab S - Id$ and $\ab R$ are isometric through $\ag \Phi$, the compactness of $\ab S - Id$ is equivalent to the compactness of $\ab R$.
Moreover the traces of both operator coincide:
\begin{eqnarray}
\sum_{n=0}^{\infty}\sum_{ \hspace{5pt} 0 \leq k < 2^{n\!-\!1} } \ab S^{n,k}_{n,k} 
&=&
\sum_{n=0}^{\infty}\sum_{ \hspace{5pt} 0 \leq k < 2^{n\!-\!1} } \int_{0}^{1}  \int_{0}^{1} \ag \phi_{n,k}(t) \, \ag \phi_{n,k}(s) \, \ab R(t,s) \, ds \, dt \, ,  \nonumber\\
&=&
\int_{0}^{1}  \int_{0}^{1} \left( \sum_{n=0}^{\infty}\sum_{ \hspace{5pt} 0 \leq k < 2^{n\!-\!1} }   \ag \phi_{n,k}(t) \, \ag \phi_{n,k}(s) \right) \, \ab R(t,s) \, ds \, dt \, , \nonumber\\
&=&
\int_{0}^{1}  \, \ab R(t,t) \, ds \, dt \, ,
\nonumber
\end{eqnarray}
using the result of Corollary \ref{cor:identDecomp}. 
\end{proof}

\section{Girsanov formula}\label{append:Girsanov}

In this section we provide the quite technical proof of Lemma \ref{lemGGextension} which is useful in proving Girsanov's formula:

\begin{lemma}
The positive definite quadratic form on $l^{2}(\mathbb{R}) \times l^{2}(\mathbb{R})$ associated with operator $\ab S - Id: l^{2}(\mathbb{R}) \rightarrow l^{2}(\mathbb{R})$ is well-defined on $\uxi \Omega'$.
Moreover for all $\uxi \Omega'$, 
\begin{eqnarray} \label{eq:StratoExp2}
\lefteqn{
\big(\xi, (\ab S - Id_{\uxi \Omega'})(\xi \big) = }
\nonumber\\
&&
2 \, \int_{0}^{1} \frac{\alpha(t) - \beta(t)}{{f}^{2}(t)} \, \frac{\ag X_{t}(\xi)}{g_{\alpha}(t)} \circ d \left( \frac{\ag X_{t}(\xi)}{g_{\alpha}(t)} \right) + \int_{0}^{1} \frac{\big(\alpha(t)- \beta(t)\big)^{2} }{{f}^{2}(t)} \, \left( \frac{\ag X_{t}(\xi)}{g_{\alpha}(t)} \right)^{2} \, dt \, , \nonumber
\end{eqnarray}
where $\ag X_{t}(\xi) = \ag \Psi(\xi)$ and $\circ$ refers to the Stratonovich integral and the equality is true in law. 
\end{lemma}

\begin{proof}
The proof of this lemma uses quite similar material as in the proof of It\^o's theorem. However since this result is central for giving insight on the way our geometric considerations relate to Girsanov's theorem, we provide the detailed proof here.\\
Consider $\xi$ in $\uxi \Omega'$, denote $\xi_{N} = \uxi P_{N}(\xi)$ and write
\begin{eqnarray} \label{eq:quadQuantity}
\big(\xi_{N}, (\ab S_{N} - Id_{\uxi \Omega_{N}})(\xi_{N}) \big) 
=  
 \sum_{(n,k)\in \I_N} \sum_{ (p,q)\in \I_N} \Big( A^{n,k}_{p,q} + B^{n,k}_{p,q} \Big) \: \xi_{n,k}  \, \xi_{p,q}
 \nonumber
\end{eqnarray}
where we have posited
\begin{eqnarray}
A^{n,k}_{p,q} &=& 
2 \int_{0}^{1}  \frac{ \alpha(t) - \beta(t)}{\sqrt{\Gamma(t)}} \big(  \ag \phi_{n,k}(t) \, \ag \psi_{p,q}(t)  +\ag \phi_{p,q}(t) \, \ag \psi_{n,k}(t) \big)\, dt  \, ,\nonumber\\
B^{n,k}_{p,q} &=& \int_{0}^{1} \frac{\big(\alpha(t)- \beta(t)\big)^{2} }{\Gamma(t)}  \big( \ag \psi_{n,k}(t) \, \ag \psi_{p,q}(t) \big) \, dt \, . \nonumber
\end{eqnarray}
It is easy to see, using similar arguments as in the proof of the integration by parts formula, Proposition \ref{prop:IPP}:
\begin{eqnarray} \label{eq:strato1}
A^{N}(\xi) &=&   \sum_{(n,k)\in \I_N} \sum_{ (p,q)\in \I_N} A^{n,k}_{p,q}\: \xi_{n,k}  \, \xi_{p,q} \nonumber\\
 &=&
 2 \, \int_{0}^{1} \frac{\alpha(t) - \beta(t)}{\sqrt{\Gamma(t)}} \: \frac{\ag X_{t}^{N}\!(\xi)}{f_{\alpha}(t)} \:  \frac{d}{dt} \! \left( \frac{\ag X_{t}^{N}(\xi)}{g_{\alpha}(t)} \right) \, dt \, ,
 \nonumber
\end{eqnarray}
and 
\begin{eqnarray}
 \sum_{(n,k)\in \I_N} \sum_{ (p,q)\in \I_N}  B^{n,k}_{p,q}\: \xi_{n,k}  \, \xi_{p,q}
 =
 \int_{0}^{1} \frac{\big(\alpha(t)- \beta(t)\big)^{2}}{\Gamma(t)} \: {\ag X_{t}^{N}\!(\xi)}^{2} \:  dt \, ,
 \nonumber
\end{eqnarray}
Because of the uniform convergence property of $X^N$ towards $X$ and the fact that it has almost surely bounded sample paths, the latter sum converges towards
\begin{eqnarray}
 \int_{0}^{1} \frac{\big(\alpha(t)- \beta(t)\big)^{2}}{{f}^{2}(t)} \: {\left( \frac{\ag X_{t} (\xi)}{g_{\alpha}(t)}\right)}^{2} \:  dt \, ,
 \nonumber
\end{eqnarray}
Now writing quantity $A^{N}(\xi)$ as the sum of elementary integrals between the points of discontinuity $t_{i} = i2^{-N}$, $0 \leq i \leq 2^{N}$
\begin{eqnarray}
A^{N}(\xi) = -2 \,  \sum_{i=0}^{2^{N}-1}  \int_{t_{i}}^{t_{i+1}}    \frac{\alpha(t) - \beta(t)}{\sqrt{\Gamma(t)}} \: \frac{\ag X_{t}^{N}\!(\xi)}{f(t)} \:  \frac{d}{dt} \! \left( \frac{\ag X_{t}^{N}(\xi)}{g_{\alpha}(t)} \right) \, dt \, ,
\nonumber
\end{eqnarray}
and using the identities of equations \eqref{eq:XNt} and \eqref{eq:DXNt}, we then have
\begin{eqnarray} \label{eq:strato2}
A^{N}(\xi) = -2 \,  \sum_{i=0}^{2^{N}-1}  \left(  w^{N}_{i} \frac{ \ag X_{t_{i}}(\xi)}{g_{\alpha}(t_{i})}  + w^{N}_{i+1}   \frac{\ag X_{t_{i+1}}(\xi)}{g_{\alpha}(t_{i+1})}  \right) \: \left( \frac{\ag X_{t_{i+1}}(\xi)}{g_{\alpha}(t_{i+1})}  - \frac{ \ag X_{t_{i}}(\xi)}{g_{\alpha}(t_{i})}  \right) \, ,
\nonumber
\end{eqnarray}
where we denoted
\begin{eqnarray}
w^{N}_{i} &=& \frac{\int_{t_{i}}^{t_{i+1}} \big(\alpha(t)- \beta(t)\big) \big( h(t_{i+1}) - h(t)  \big) \, dt}{\big( h(t_{i+1}) - h(t_{i}) \big)^{2}} \, ,
\nonumber\\
w^{N}_{i+1} &=& \frac{\int_{t_{i}}^{t_{i+1}} \big(\alpha(t)- \beta(t)\big)  \big( h(t_{i+1}) - h(t)  \big) \, dt}{\big( h(t_{i+1}) - h(t_{i}) \big)^{2}} \, .
\nonumber
\end{eqnarray}
Let us define the function $w$ in $C[0,1]$ by
\begin{equation}
w(t) = \frac{ \alpha(t)- \beta(t)}{{f}^{2}(t)}
\nonumber
\end{equation}
If $\alpha$ and $\beta$ is uniformly $\delta$-H\"{o}lder continuous, so is $w$.
Therefore, there exist an integer $N'>0$ and a real $M>0$ such that if $N>N'$, for all $0 \leq i < 2^{N}$, we have
\begin{eqnarray}
\big \vert w^{N}_{i} - w(t_{i}) \big \vert &=& \left \vert  \displaystyle \int_{t_{i}}^{t_{i+1}} \big( w(t) - w(t_{i}) \big)  \, \frac{ \frac{d}{dt} \! \left( \big( h(t_{i+1}) - h(t)  \big)^{2} \right)}{\big( h(t_{i+1}) - h(t_{i}) \big)^{2}} \, dt \right \vert  \, ,
\nonumber\\
& \leq & M \left  \vert  \displaystyle \int_{t_{i}}^{t_{i+1}} ( t -t_{i} )^{\delta}  \, \frac{ \frac{d}{dt} \! \left( \big( h(t_{i+1}) - h(t)  \big)^{2} \right)}{\big( h(t_{i+1}) - h(t_{i}) \big)^{2}} \, dt \right \vert  \, ,
\nonumber\\
& \leq & M  \left( \frac{(t-t_{i})^{\delta+1}}{2 (\delta + 1)} \right) + M \left  \vert \displaystyle \int_{t_{i}}^{t_{i+1}} \frac{(t-t_{i})^{\delta+1}}{\delta + 1}  \frac{  \big( h(t_{i+1}) - h(t)  \big)^{2}}{\big( h(t_{i+1}) - h(t_{i}) \big)^{2}}  \, dt \right \vert  \, ,
\nonumber\\
& \leq & M  \left( \frac{(t-t_{i})^{\delta+1}}{2 (\delta + 1)} \right) + M  \left( \frac{(t-t_{i})^{\delta+2}}{2 (\delta + 1)(\delta + 2)} \right)  \, ,
\nonumber
\end{eqnarray}
which shows that for $\vert w^{N}_{i} - w(t_{i}) \vert = O(2^{-N(1+\delta)})$, and similarly $\vert w^{N}_{i+1} - w(t_{i+1}) \vert = O(2^{-N(1+\delta)})$ as well.
As a consequence, expression \eqref{eq:strato2} converges when $N$ tends to infinity toward the desired Stratonovich integral.
\end{proof}

\bibliographystyle{plain}
\bibliography{MyBib}

\begin{thebibliography}{10}

\bibitem{Allaire:2007}
Gr{\'e}goire Allaire.
\newblock {\em Numerical analysis and optimization}.
\newblock Numerical Mathematics and Scientific Computation. Oxford University
  Press, Oxford, 2007.
\newblock An introduction to mathematical modelling and numerical simulation,
  Translated from the French by Alan Craig.

\bibitem{aronszajn:1950}
N.~Aronszajn.
\newblock Theory of reproducing kernels.
\newblock {\em Trans. Amer. Math. Soc.}, 68:337--404, 1950.

\bibitem{Baldi:1995}
Paolo Baldi.
\newblock Exact asymptotics for the probability of exit from a domain and
  applications to simulation.
\newblock {\em Ann. Probab.}, 23(4):1644--1670, 1995.

\bibitem{Baldi:2002}
Paolo Baldi and Lucia Caramellino.
\newblock Asymptotics of hitting probabilities for general one-dimensional
  pinned diffusions.
\newblock {\em Ann. Appl. Probab.}, 12(3):1071--1095, 2002.

\bibitem{Baldi:2000}
Paolo Baldi, Lucia Caramellino, and Maria~Gabriella Iovino.
\newblock Pricing complex barrier options with general features using sharp
  large deviation estimates.
\newblock In {\em Monte {C}arlo and quasi-{M}onte {C}arlo methods 1998
  ({C}laremont, {CA})}, pages 149--162. Springer, Berlin, 2000.

\bibitem{Lyons:2002}
R.~F. Bass, B.~M. Hambly, and T.~J. Lyons.
\newblock Extending the {W}ong-{Z}akai theorem to reversible {M}arkov
  processes.
\newblock {\em J. Eur. Math. Soc. (JEMS)}, 4(3):237--269, 2002.

\bibitem{Bogachev:1998}
Vladimir~I. Bogachev.
\newblock {\em Gaussian measures}, volume~62 of {\em Mathematical Surveys and
  Monographs}.
\newblock American Mathematical Society, Providence, RI, 1998.

\bibitem{Brodski:1971}
M.~S. Brodski{\u\i}.
\newblock {\em Triangular and {J}ordan representations of linear operators}.
\newblock American Mathematical Society, Providence, R.I., 1971.
\newblock Translated from the Russian by J. M. Danskin, Translations of
  Mathematical Monographs, Vol. 32.

\bibitem{Brodski:1968}
V.~M. Brodski{\u\i} and M.~S. Brodski{\u\i}.
\newblock The abstract triangular representation of bounded linear operators
  and the multiplicative expansion of their eigenfunctions.
\newblock {\em Dokl. Akad. Nauk SSSR}, 181:511--514, 1968.

\bibitem{Ricciardi:1987}
A.~Buonocore, A.~G. Nobile, and L.~M. Ricciardi.
\newblock A new integral equation for the evaluation of first-passage-time
  probability densities.
\newblock {\em Adv. in Appl. Probab.}, 19(4):784--800, 1987.

\bibitem{Ciesielski}
Z.~Ciesielski.
\newblock H\"older conditions for realizations of {G}aussian processes.
\newblock {\em Trans. Amer. Math. Soc.}, 99:403--413, 1961.

\bibitem{DaPrato:1992}
Giuseppe Da~Prato and Jerzy Zabczyk.
\newblock {\em Stochastic equations in infinite dimensions}, volume~44 of {\em
  Encyclopedia of Mathematics and its Applications}.
\newblock Cambridge University Press, Cambridge, 1992.

\bibitem{Dahmen:2000}
W.~Dahmen, B.~Han, R.-Q. Jia, and A.~Kunoth.
\newblock Biorthogonal multiwavelets on the interval: cubic {H}ermite splines.
\newblock {\em Constr. Approx.}, 16(2):221--259, 2000.

\bibitem{Meyer:1998}
K.~Daoudi, J.~L{\'e}vy~V{\'e}hel, and Y.~Meyer.
\newblock {C}onstruction of continuous functions with prescribed local
  regularity.
\newblock {\em Constr. Approx.}, 14(3):349--385, 1998.

\bibitem{deBoor:2001}
Carl de~Boor.
\newblock {\em A practical guide to splines}, volume~27 of {\em Applied
  Mathematical Sciences}.
\newblock Springer-Verlag, New York, revised edition, 2001.

\bibitem{deBoor:2008}
Carl de~Boor, Christian Gout, Angela Kunoth, and Christophe Rabut.
\newblock Multivariate approximation: theory and applications. {A}n overview.
\newblock {\em Numer. Algorithms}, 48(1-3):1--9, 2008.

\bibitem{Dolph}
C.~L. Dolph and M.~A. Woodbury.
\newblock On the relation between {G}reen's functions and covariances of
  certain stochastic processes and its application to unbiased linear
  prediction.
\newblock {\em Trans. Amer. Math. Soc.}, 72:519--550, 1952.

\bibitem{Yor:2001}
Catherine Donati-Martin, Raouf Ghomrasni, and Marc Yor.
\newblock On certain {M}arkov processes attached to exponential functionals of
  {B}rownian motion; application to {A}sian options.
\newblock {\em Rev. Mat. Iberoamericana}, 17(1):179--193, 2001.

\bibitem{Ermentrout:95}
Bard Ermentrout.
\newblock Type i membranes, phase resetting curves, and synchrony.
\newblock {\em Neural Comput}, 8:979--1001, 1995.

\bibitem{Ermentrout:86}
G.~B. Ermentrout and N.~Kopell.
\newblock Parabolic bursting in an excitable system coupled with a slow
  oscillation.
\newblock {\em SIAM J. Appl. Math.}, 46(2):233--253, 1986.

\bibitem{Friz:2010}
Peter~K. Friz and Nicolas~B. Victoir.
\newblock {\em Multidimensional stochastic processes as rough paths}, volume
  120 of {\em Cambridge Studies in Advanced Mathematics}.
\newblock Cambridge University Press, Cambridge, 2010.
\newblock Theory and applications.

\bibitem{Fukushima:1994}
Masatoshi Fukushima, Y{\=o}ichi {\=O}shima, and Masayoshi Takeda.
\newblock {\em Dirichlet forms and symmetric {M}arkov processes}, volume~19 of
  {\em de Gruyter Studies in Mathematics}.
\newblock Walter de Gruyter \& Co., Berlin, 1994.

\bibitem{Gobet2000167}
Emmanuel Gobet.
\newblock Weak approximation of killed diffusion using euler schemes.
\newblock {\em Stochastic Processes and their Applications}, 87(2):167 -- 197,
  2000.

\bibitem{Gohberg:1970}
I.~C. Gohberg and M.~G. Kre{\u\i}n.
\newblock {\em Theory and applications of {V}olterra operators in {H}ilbert
  space}.
\newblock Translated from the Russian by A. Feinstein. Translations of
  Mathematical Monographs, Vol. 24. American Mathematical Society, Providence,
  R.I., 1970.

\bibitem{Goldman:1971}
Malcolm Goldman.
\newblock On the first passage of the integrated {W}iener process.
\newblock {\em Ann. Mat. Statist.}, 42:2150--2155, 1971.

\bibitem{Hsing:2009}
Tailen Hsing and Haobo Ren.
\newblock An {RKHS} formulation of the inverse regression dimension-reduction
  problem.
\newblock {\em Ann. Statist.}, 37(2):726--755, 2009.

\bibitem{Izhikevich:2007}
Eugene~M. Izhikevich.
\newblock {\em Dynamical systems in neuroscience: the geometry of excitability
  and bursting}.
\newblock Computational Neuroscience. MIT Press, Cambridge, MA, 2007.

\bibitem{Kailath:1978}
T.~Kailath, A.~Vieira, and M.~Morf.
\newblock Inverses of {T}oeplitz operators, innovations, and orthogonal
  polynomials.
\newblock {\em SIAM Rev.}, 20(1):106--119, 1978.

\bibitem{Karatzas}
Ioannis Karatzas and Steven~E. Shreve.
\newblock {\em Brownian motion and stochastic calculus}, volume 113 of {\em
  Graduate Texts in Mathematics}.
\newblock Springer-Verlag, New York, second edition, 1991.

\bibitem{Kimel3}
George Kimeldorf and Grace Wahba.
\newblock Some results on {T}chebycheffian spline functions.
\newblock {\em J. Math. Anal. Appl.}, 33:82--95, 1971.

\bibitem{Kimel1}
George~S. Kimeldorf and Grace Wahba.
\newblock A correspondence between {B}ayesian estimation on stochastic
  processes and smoothing by splines.
\newblock {\em Ann. Math. Statist.}, 41:495--502, 1970.

\bibitem{Kimel2}
George~S. Kimeldorf and Grace Wahba.
\newblock Spline functions and stochastic processes.
\newblock {\em Sankhy\=a Ser. A}, 32:173--180, 1970.

\bibitem{Kolsrud:1988}
Torbj{\"o}rn Kolsrud.
\newblock Gaussian random fields, infinite-dimensional {O}rnstein-{U}hlenbeck
  processes, and symmetric {M}arkov processes.
\newblock {\em Acta Appl. Math.}, 12(3):237--263, 1988.

\bibitem{Kuo:1975}
Hui~Hsiung Kuo.
\newblock {\em Gaussian measures in {B}anach spaces}.
\newblock Lecture Notes in Mathematics, Vol. 463. Springer-Verlag, Berlin,
  1975.

\bibitem{Lefebvre:1989}
Mario Lefebvre and {\'E}ric Leonard.
\newblock On the first hitting place of the integrated {W}iener process.
\newblock {\em Adv. in Appl. Probab.}, 21(4):945--948, 1989.

\bibitem{Levy}
Paul L{\'e}vy.
\newblock {\em Processus {S}tochastiques et {M}ouvement {B}rownien. {S}uivi
  d'une note de {M}. {L}o\`eve}.
\newblock Gauthier-Villars, Paris, 1948.

\bibitem{Lyons:1994}
Terry Lyons.
\newblock Differential equations driven by rough signals. {I}. {A}n extension
  of an inequality of {L}. {C}. {Y}oung.
\newblock {\em Math. Res. Lett.}, 1(4):451--464, 1994.

\bibitem{Mallat:1989}
Stephane~G. Mallat.
\newblock Multiresolution approximations and wavelet orthonormal bases of
  {$L\sp 2({\bf R})$}.
\newblock {\em Trans. Amer. Math. Soc.}, 315(1):69--87, 1989.

\bibitem{McKean:1963}
H.~P. McKean, Jr.
\newblock A winding problem for a resonator driven by a white noise.
\newblock {\em J. Math. Kyoto Univ.}, 2:227--235, 1963.

\bibitem{Meyer:1999}
Yves Meyer, Fabrice Sellan, and Murad~S. Taqqu.
\newblock Wavelets, generalized white noise and fractional integration: the
  synthesis of fractional {B}rownian motion.
\newblock {\em J. Fourier Anal. Appl.}, 5(5):465--494, 1999.

\bibitem{Pitt:1971kl}
Loren~D. Pitt.
\newblock A {M}arkov property for {G}aussian processes with a multidimensional
  parameter.
\newblock {\em Archive for Rational Mechanics and Analysis}, 43(5):367--391,
  1971.

\bibitem{Protter:2004}
Philip~E. Protter.
\newblock {\em Stochastic integration and differential equations}, volume~21 of
  {\em Applications of Mathematics (New York)}.
\newblock Springer-Verlag, Berlin, second edition, 2004.
\newblock Stochastic Modelling and Applied Probability.

\bibitem{Rako:2005}
Alain Rakotomamonjy and St\'{e}phane Canu.
\newblock Frames, reproducing kernels, regularization and learning.
\newblock {\em J. Mach. Learn. Res.}, 6:1485--1515, 2005.

\bibitem{Schauder:27}
Julius Schauder.
\newblock Bemerkungen zu meiner {A}rbeit ``{Z}ur {T}heorie stetiger
  {A}bbildungen in {F}unktionalr\"aumen''.
\newblock {\em Math. Z.}, 26(1):417--431, 1927.

\bibitem{Schauder:28}
Juljusz Schauder.
\newblock Eine {E}igenschaft des {H}aarschen {O}rthogonalsystems.
\newblock {\em Math. Z.}, 28(1):317--320, 1928.

\bibitem{Smola:2001}
Bernhard Sch{\"o}lkopf, Ralf Herbrich, and Alex~J. Smola.
\newblock A generalized representer theorem.
\newblock In {\em Computational learning theory ({A}msterdam, 2001)}, volume
  2111 of {\em Lecture Notes in Comput. Sci.}, pages 416--426. Springer,
  Berlin, 2001.

\bibitem{Scholkopf:2001}
Bernhard Sch\"{o}lkopf and Alexander~J. Smola.
\newblock {\em Learning with Kernels: Support Vector Machines, Regularization,
  Optimization, and Beyond (Adaptive Computation and Machine Learning)}.
\newblock The MIT Press, 1st edition, December 2001.

\bibitem{Sekine:2001}
Jun Sekine.
\newblock {I}nformation geometry for {S}ymmetric {D}iffusions.
\newblock {\em Potential Analysis}, 14(1):1--30, 2001.

\bibitem{Strook:2006}
Daniel~W. Stroock and S.~R.~Srinivasa Varadhan.
\newblock {\em Multidimensional diffusion processes}.
\newblock Classics in Mathematics. Springer-Verlag, Berlin, 2006.
\newblock Reprint of the 1997 edition.

\bibitem{Taillefumier:2010}
Thibaud Taillefumier and Marcelo Magnasco.
\newblock A {F}ast {A}lgorithm for the {F}irst-{P}assage {T}imes of
  {G}auss-{M}arkov processes with {H}\"{o}lder {C}ontinuous {B}oundaries.
\newblock {\em Journal of Statistical Physics}, 140(6):1--27, 2010.

\bibitem{ThibaudJOSP:2008}
Magnasco~M. Taillefumier~T.
\newblock A {H}aar-like {C}onstruction for the {O}rnstein {U}hlenbeck
  {P}rocess.
\newblock {\em J. Stat. Phys.}, 132(2):397--415, July 2008.

\bibitem{Touboul:2008}
Jonathan Touboul and Olivier Faugeras.
\newblock A characterization of the first hitting time of double integral
  processes to curved boundaries.
\newblock {\em Adv. in Appl. Probab.}, 40(2):501--528, 2008.

\bibitem{Wahba:1990}
Grace Wahba.
\newblock {\em Spline models for observational data}, volume~59 of {\em
  CBMS-NSF Regional Conference Series in Applied Mathematics}.
\newblock Society for Industrial and Applied Mathematics (SIAM), Philadelphia,
  PA, 1990.

\end{thebibliography}

\end{document}